\documentclass[12pt]{article}  
\usepackage{amsmath}
\usepackage{amssymb}
\usepackage{theorem}
\usepackage{euscript}
\usepackage{pstricks}
\usepackage{authblk}
\usepackage{verbatim}
\usepackage{graphics}
\usepackage{graphicx}
\usepackage{hyperref}
\usepackage{imakeidx}
\makeindex

\usepackage{color}

\usepackage{hyperref}
\hypersetup
{
	colorlinks,
	citecolor=black,
	filecolor=black,
	linkcolor=black,
	urlcolor=blue
}
\hypersetup{linktocpage}
\newtheorem{theorem}{Theorem}[section]
\newtheorem{lemma}[theorem]{Lemma}
\newtheorem{corollary}[theorem]{Corollary}
\numberwithin{equation}{section}

\theoremstyle{definition}
 
\theoremstyle{remark}
\newtheorem{remark}[theorem]{Remark}

\newcommand{\brac}[1]{\left(#1\right)}

\newcommand{\brab}[1]{\left\{#1\right\}}

\newcommand{\bj}{{\boldsymbol{j}}}
\newcommand{\bk}{{\boldsymbol{k}}}

\newcommand{\bh}{{\boldsymbol{h}}}

\newcommand{\br}{{\boldsymbol{r}}}
\newcommand{\bs}{{\boldsymbol{s}}}
\newcommand{\bt}{{\boldsymbol{t}}}
\newcommand{\bx}{{\boldsymbol{x}}}

\newcommand{\bF}{{\boldsymbol{F}}}
\newcommand{\bH}{{\boldsymbol{H}}}
\newcommand{\bX}{{\boldsymbol{X}}}
\newcommand{\by}{{\boldsymbol{y}}}

\newcommand{\bW}{{\boldsymbol{W}}}

\newcommand{\bone}{{\boldsymbol{1}}}

\newcommand{\bPhi}{{\boldsymbol{\Phi}}}
\newcommand{\bLambda}{{\boldsymbol{\Lambda}}}

\newcommand{\bxi}{{\boldsymbol{\xi}}}
\newcommand{\rd}{{\rm d}} 
\newcommand{\Int}{{\rm Int}}

\def\II{\mathbb I}

\def\ZZd{{\mathbb Z}^d}
\def\IId{{\mathbb I}^d}
\def\ZZ{{\mathbb Z}}
\def\ZZdp{{\mathbb Z}^d_+}
\def\RR{{\mathbb R}}
\def\RRd{{\mathbb R}^d}

\def\NN{{\mathbb N}}
\def\NNd{{\NN}^d}

\def\II{{\mathbb I}}

\def\NN{{\mathbb N}}
\def\RR{{\mathbb R}}

\def\FF{{\mathcal F}}

\def\TT{{\mathbb T}}
\def\IId{{\mathbb I}^d}

\def\NNd{{\mathbb N}^d}
\def\RRd{{\mathbb R}^d}

\def\TTd{{\mathbb T}^d}

\def\ZZd{{\mathbb Z}^d}

\def\Hh{{\mathcal H}}

\def\Pp{{\mathcal P}}
\def\Qq{{\mathcal Q}}

\def\II{{\mathbb I}}

\def\ZZ{{\mathbb Z}}
\def\NN{{\mathbb N}}
\def\RR{{\mathbb R}}

\def\FF{{\mathbb F}}
\def\TT{{\mathbb T}}
\def\IId{{\mathbb I}^d}

\def\NNd{{\mathbb N}^d}
\def\RRd{{\mathbb R}^d}
\def\TTd{{\mathbb T}^d}

\def\ZZdp{{\mathbb N}^d_0}

\def\supp{\rm{supp}}

\def\sign{\rm{sign}}

\def\Wpmix{W^r_p(\IId)}

\def\Wpgamma{W^r_p(\RRd;\mu)}
\def\Wap{W^r_p}
\def\Wa{W^r_2(\RRd,\gamma)}
\def\Lqgamma{L_q(\RRd;\mu)}

\def\Wrp{W^r_p(\TTd)}
\def\Urp{\bW^r_p(\TTd)}
\def\BWpmix{\bW^r_p(\IId)}

\def\BWpgamma{\bW^r_p(\RRd;\mu)}

\newcommand{\norm}[2]{\left\|{#1}\right\|_{#2}}

\title{\sffamily Sparse-grid sampling recovery and numerical integration  of functions having  mixed smoothness}

\author[a]{Dinh D\~ung}
\affil[a]{Information Technology Institute, Vietnam National University, Hanoi
	\protect\\
	144 Xuan Thuy, Cau Giay, Hanoi, Vietnam
	\protect\\
	Email: dinhzung@gmail.com}

\date{\today}
\begin{document}
\maketitle

\begin{abstract}
 We give a short survey of recent results on sparse-grid linear  algorithms of approximate recovery and integration of functions possessing a unweighted or weighted Sobolev mixed smoothness based on their sampled values at a certain finite set. Some of them are extended to more general cases.
	
	
	
\end{abstract}

\tableofcontents

\section{Introduction}
\label{Introduction}
In recent decades, there has been increasing interest in solving approximation and numerical problems that involve
functions depending on a large number $d$ of variables.  Without further assumptions the computation time typically grows exponentially in $d$, and
the problems become intractable even for mild dimensions $d$. 
This is the so-called curse of conditionality coined by Bellman \cite{Bel57B}.  
In sampling recovery and numerical integration, a classical model in attempt to overcome it which has been widely studied,  is to impose certain mixed smoothness or more general anisotropic smoothness conditions on the function to be approximated, and to employ sparse grids for construction of approximation algorithms for sampling recovery or numerical integration. 

In this survey paper, we discuss sparse-grid  linear algorithms of approximate reconstruction and integration of functions possessing a Sobolev mixed smoothness, based on their sampling values at a certain finite set. The problem of sampling recovery  is treated in both unweighted setting for functions defined on a compact subset in $\RRd$, and weighted setting for functions defined on $\RRd$. The optimality of sampling recovery is considered in terms of linear  sampling widths. The problem of numerical integration  and  optimal quadrature  is treated in weighted setting for functions defined on the whole  $\RRd$. We focus our attention on linear sampling algorithms and quadratures based on sparse grids. Some known results are extended to more general cases. We do not consider  the problem of dependence on dimension and postpone discussion on this problem  to other surveys.

We begin with some notions related to linear sampling recovery and numerical integration of functions based on  their values. Let  $\Omega$ be a domain in  $\RRd$,
$X$  a normed space of functions on $\Omega$ and  $f \in X$.
Given $\bX_k= \{\bx_i\}_{i=1}^k\subset \Omega$ and a collection $\bPhi_k:= \brab{\varphi_i}_{i=1}^k$  of $k$ functions in $X$, to approximately recover $f$  from the sampled values $\brab{f(\bx_i)}_{i=1}^k$ we can use a linear sampling algorithm defined by
\begin{equation} \label{R_n(f)}
	S_k(f): = 	S_k(\bX_k,\bPhi_k,f) := \sum_{i=1}^k  f(\bx_i) \varphi_i
\end{equation}
 For convenience, we assume that some points from $\bX_k$ and some functions from $\bPhi_k$ may coincide.
 
If $\bW$ is a compact set in $X$, for $n\in \NN$ we define the  linear sampling $n$-width  of the set $\bW$ in $X$ as
$$
\varrho_n(\bW,X):=\inf_{\bX_k, \,
	\bPhi_k, \ k \le n} \  \ \sup_{f\in \bW}
\|f- S_k(\bX_k,\bPhi_k,f)\|_X.
$$
This quantity characterizes the optimal sampling recovery for the set $\bW$ by linear algorithms  from $n$ sampled values. Some different concepts of optimal sampling recovery are discussed in Subsections \ref{Sampling recovery} and \ref{Optimalities in sampling recovery}.

Let $v$ be a nonnegative Lebesgue measurable  function on $\Omega$. Denote by $\mu_v$ the measure on $\Omega$ defined via the density function $v$. We are interested in numerical approximation of   weighted integrals 
\begin{equation} \label{If}
	\int_{\Omega} f(\bx) \mu_v(\rd\bx)
	=
	\int_{\RRd}f(\bx) v(\bx) \rd\bx.
\end{equation}
Given $\bX_k= \{\bx_i\}_{i=1}^k\subset \Omega$ and a collection $\bLambda_k:= \brab{\lambda_i}_{i=1}^k$  of integration weights, to approximate them we use quadratures  of the form
\begin{equation} \label{Q_nf-introduction}
	Q_k(f): = \sum_{i=1}^k \lambda_i f(\bx_i)
\end{equation}
with the convention $Q_0(f):=0$.  For convenience, we assume that some of the integration nodes may coincide. 

Let $\bW$ be a set of continuous functions on $\Omega$.  Denote by $\Qq_n$  the family of all quadratures $Q_k$ of the form \eqref{Q_nf-introduction} with $k \le n$. The optimality  of  quadratures from $\Qq_n$ for  $f \in \bW$  is characterized by the quantity 
\begin{equation} \label{Int_n}
	\Int_n(\bW) :=\inf_{Q_n \in \Qq_n} \ \sup_{f\in \bW} 
	\bigg|\int_{\Omega} f(\bx)  \, \mu_v(\rd\bx) - Q_n(f)\bigg|.
\end{equation}

For sampling recovery and numerical integration problems considered in the present paper,  the set  $\bW$ is usually the unit ball of a normed space $W$ of functions on $\Omega$, which is also called function class.

The unweighted sampling recovery of  functions on a compact set having a mixed smoothness, is a classical topic in multivariate approximation. 
There is a large number of works devoted to this topic. However, here are  still  many open problems on the right asymptotic order of linear sampling $n$-widths and asymptotically optimal sampling algorithms  of sampling recovery. 
We refer the reader to the book \cite[Section 5]{DTU18B} for a detailed survey and bibliography on this research direction.  The present paper updates it with some recent results mainly on linear sampling constructive algorithms, in particular,  on Smolyak sparse grids for periodic functions having a mixed Sobolev smoothness. We also discuss the right asymptotic order of linear sampling $n$-widths of the unit ball of unweighted Sobolev spaces (unweighted Sobolev function class) of a mixed smoothness, and  asymptotic optimality of sampling algorithms  in terms of   sampling $n$-widths.

Concerning weighted sampling recovery,   for  the Gaussian-weighted  Sobolev space $W^\alpha_p(\mathbb{R}^d; \gamma)$ of  mixed smoothness $r \in \mathbb{N}$ for $1 < p < \infty$, in \cite{DK2023}, we proved the asymptotic order  of sampling $n$-widths in the Gaussian-weighted space $L_q(\mathbb{R}^d; \gamma)$ of the unit ball of $W^\alpha_p(\mathbb{R}^d; \gamma)$ (Gaussian-weighted Sobolev function class) for $1 \leq q < p < \infty$ and  $q=p=2$, and propose a novel method for constructing asymptotically optimal linear sampling algorithms.
In \cite{DK2023}, we investigated the numerical approximation of integrals over $\mathbb{R}^d$ equipped with the standard Gaussian measure $\gamma$  for  integrands  belonging to the Gaussian-weighted  Sobolev spaces $W^\alpha_p(\mathbb{R}^d; \gamma)$ of  mixed smoothness $r \in \mathbb{N}$ for $1 < p < \infty$. We proved the the right asymptotic order  of the quantity  of  optimal quadrature and proposed a novel method for constructing asymptotically optimal quadratures. In the present paper, we extend these results on sampling recovery and numerical integration to weighted  Sobolev spaces $W^\alpha_p(\mathbb{R}^d; \mu)$, where $\mu$ is the measure having the tensor product of  Freud-type weights as the density function.

In \cite{DD2023}, we studied the  numerical approximation of weighted integrals over $\mathbb{R}^d$  for  integrands from weighted  Sobolev spaces of mixed smoothness $W^r_1(\mathbb{R}^d; \mu)$, where $\mu$ is the measure with the density function as the tensor product of a Freud-type weight. We proved upper and lower bounds  of the asymptotic order of optimal quadrature for functions from  these spaces.  In the one-dimensional case $(d=1)$,  we obtained the right asymptotic order the quantity of optimal quadrature. For $d \ge 2$, the upper bound is performed by sparse-grid quadratures with integration nodes on step hyperbolic crosses in the function domain $\mathbb{R}^d$.  We will give a brief description of the results and their proofs from \cite{DD2023}. 

We briefly describe the structure and content of the present paper.

In Section \ref{Unweighted sampling recovery}, we consider  the problem of  sampling recovery of  functions from a certain set  by linear algorithms  using its values at a certain finite set of points, and their asymptotic optimality in terms of linear sampling $n$-widths. We  focus our attention on functions mainly from unweighted periodic Sobolev spaces $W^r_p(\TTd)$ of a mixed smoothness, and lessly from unweighted  periodic H\"older-Nikol'skii spaces $H^r_p(\TTd)$ of a mixed smoothness.
As a related problem we also treat  linear approximations of functions from these spaces. In Subsection \ref{Introducing remarks Sec 1}, we introduce  notions of Kolmogorov and linear $n$-widths -- the well-known characterizations of optimal linear approximation; definitions of  Sobolev spaces $W^r_p(\TTd)$ and Besov spaces $B_{p, \theta}^r(\TTd)$  of periodic functions having a mixed smoothness.  Subsection~\ref{B-spline Q-I} presents Littlewood-Paley-type theorems on B-spline quasi-interpolation sampling representation for spaces $W^r_p(\TTd)$ and $B_{p, \theta}^r(\TTd)$ which play a central role in linear sampling recovery by sparse-grid Smolyak algorithms of functions from these spaces. Subsection \ref{Sampling recovery} is devoted to the sampling recovery by Smolyak sparse-grid algorithms for functions from unweighted periodic Sobolev spaces $W^r_p(\TTd)$ and Besov spaces $B_{p, \theta}^r(\TTd)$ of a mixed smoothness, and their asymptotic optimality in terms of  Smolyak and linear sampling $n$-widths.
In Subsection \ref{Sampling recovery in RK Hilbert spaces}, we present some recent results of \cite{DKU2023} on inequality between the linear sampling  and Kolmogorov $n$-widths and as consequences some new results on right asymptotic order of linear sampling $n$-widths for the unweighted  Sobolev function class $\bW^r_p(\TTd)$ and H\"older-Nikol'skii function class $\bH^r_p(\TTd)$
Subsection~\ref{Optimalities in sampling recovery} discusses 
some different concepts of optimality in sampling recovery and their relations for unweighted Sobolev spaces of a mixed smoothness.

In Section \ref{Weighted sampling recovery},
we consider the problems of linear sampling recovery and  approximation of functions in weighted Sobolev spaces. The optimality of linear sampling recovery and approximation is treated in terms of linear sampling, linear and Kolmogorov $n$-widths. In Subsection \ref{Introducing remarks Sec 3}, we introduce weighted Sobolev spaces of a mixed smoothness, and recall  previous results of \cite{DK2023} on right asymptotic order of sampling, Kolmogorov and linear $n$-widths of the Gaussian-weighted Sobolev  function classes $\bW^r_p(\RRd;\gamma)$ of a mixed smoothness  in the Gaussian-weighted space $L_q(\RRd;\gamma)$. In Subsection \ref{The case 1 le q < p < infty}, we extend these results  to the  weighted Sobolev  function class $\BWpgamma$ and weighted space $L_q(\RRd;\mu)$ for the measure $\mu$ associated with a Freud-type weight as the density function for the case $1\le q < p < \infty$. 

In Section \ref{Numerical weighted integration p}, we present and extend some recent results of \cite{DK2023} on numerical weighted integration over $\mathbb{R}^d$ for functions from weighted  Sobolev spaces $W^r_p(\mathbb{R}^d;\mu)$ of  mixed smoothness $r \in \mathbb{N}$ for $1<p<\infty$ and the measure $\mu$ associated with a Freud-type weight as the density function. In Subsection \ref{Introducing remarks sec 4},
we recall the previous  results from  \cite{DK2023} on numerical weighted integration over $\mathbb{R}^d$ for functions from Gaussian-weighted  Sobolev spaces $W^r_p(\mathbb{R}^d;\gamma)$ of  mixed smoothness $r \in \mathbb{N}$ for $1<p<\infty$. In Subsection \ref{Assembling quadratures}, based on a quadrature on  the  $d$-cube  for numerical unweighted integration of functions from classical Sobolev spaces of mixed smoothness $r$, by assembling we construct  a quadrature on $\RRd$ for numerical integration of functions from weighted Sobolev spaces $\Wpgamma$ which preserves   the  convergence rate. In Subsection \ref{Optimal numerical integration}, we prove the right asymptotic order of optimal quadrature for  the weighted  Sobolev function class $\bW^r_p(\mathbb{R}^d;\mu)$ for $1<p<\infty$.

In Section	\ref{Numerical weighted integration 1}, we present 
some recent results of \cite{DD2023} on numerical weighted integration over $\mathbb{R}^d$ for functions from weighted  Sobolev spaces  $W^r_1(\mathbb{R}^d;\mu)$ of  mixed smoothness $r \in \mathbb{N}$, in particular, 
upper and lower bounds of the quantity of  optimal quadrature $\Int_n\big(\bW^r_1(\mathbb{R}^d;\mu)\big)$ for the measure $\mu$ associated with Freud-type weights as the density function.
In Subsection \ref{Introducing remarks sec 5}, we briefly describe these results and then give comments on related works. In Subsection \ref{Univariate integration},
 for one-dimensional numerical integration, we present the right asymptotic order of the quantity of optimal quadrature $\Int_n\big(\bW^r_1(\mathbb{R};\mu)\big)$ and its shortened proof  from \cite{DD2023}. In Subsection \ref{Multivariate integration}, for multivariate numerical integration, we present some results and their shortened proofs from \cite{DD2023} on upper and lower bounds of   the quantity of optimal quadrature $\Int_n\big(\bW^r_1(\mathbb{R}^d;\mu)\big)$,  and   a construction of  quadratures based on  step-hyperbolic-cross grids of integration nodes which gives the upper bounds. In Subsection \ref{Extension}, we extend the results of the previous subsection to  Markov-Sonin weights.

\medskip
\noindent
{\bf Notation.} 
 Denote  $\bone:= (1,...,1) \in \RRd$; for $\bx \in \RRd$, $\bx=:\brac{x_1,...,x_d}$;
$|\bx|_p:= \brac{\sum_{j=1}^d |x_j|^p}^{1/p}$ $(1 \le p < \infty)$ and $|\bx|_\infty:= \max_{1\le j \le d} |x_j|$; we use the abbreviation: $|\bx|:= |\bx|_2$.  For $\bx, \by \in \RRd$, the inequality $\bx \le \by$ means $x_i \le y_i$ for every $i=1,...,d$.  For $x \in \RR$, denote $\sign (x):= 1$ if $x \ge 0$, and $\sign (x):= -1$ if  $x < 0$. We use letters $C$  and $K$ to denote general 
positive constants which may take different values. For the quantities $A_n(f,\bk)$ and $B_n(f,\bk)$ depending on 
$n \in \NN$, $f \in W$, $\bk \in \ZZd$,  
we write  $A_n(f,\bk) \ll B_n(f,\bk)$, $f \in W$, $\bk \in \ZZd$ ($n \in \NN$ is specially dropped),  
if there exists some constant $C >0$ such that 
$A_n(f,\bk) \le CB_n(f,\bk)$ for all $n \in \NN$,  $f \in W$, $\bk \in \ZZd$ (the notation $A_n(f,\bk) \gg B_n(f,\bk)$ has the obvious opposite meaning), and  
$A_n(f,\bk) \asymp B_n(f,\bk)$ if $A_n(f,\bk) \ll B_n(f,\bk)$
and $B_n(f,\bk) \ll A_n(f,\bk)$.  Denote by $|G|$ the cardinality of the set $G$. 
For a normed space $X$, denote by the boldface $\bX$ the unit ball in $X$.

\section{Unweighted sampling recovery}
\label{Unweighted sampling recovery}

\subsection{Introducing remarks}
\label{Introducing remarks Sec 1}
In this section, we consider one of basic problems in approximation theory. Namely, we are interested in sampling recovery, i.e.,  approximate reconstruction of a function by sparse-grid linear algorithms using its values at a certain finite set of points. We mainly focus our attention on functions from unweighted Sobolev spaces of a mixed smoothness. 
It is related to the well-known linear and Kolmogorov $n$-widths in linear approximation. Let us recall them.

Let $n \in \NN$ and 
let $X$ be a normed space and $\bW$ a central symmetric compact set in $X$.
Then the Kolmogorov $n$-width  of $\bW$ is defined by
\begin{equation*}
d_n(\bW,X)= \inf_{L_{n}}\sup_{f\in \bW}\inf_{g\in L_n}\|f-g\|_X, 
\end{equation*}
where the left-most  infimum is  taken over all  subspaces $L_{n}$ of dimension  at most $n$ in $X$.  
The linear $n$-width of the set $F$ which is defined by
$$
\lambda_n(\bW,X):=\inf_{A_n} \sup_{f\in \bW} \|f-A_n(f)\|_X,
$$
where the infimum is taken over all linear operators $A_n$ in  $X$ with  rank $\, A_n\leq n$.

The concepts of Kolmogorov $n$-widths and   linear $n$-widths are related to linear approximation. Namely, $d_n(\bW,X)$ characterizes the optimal approximation  of elements from $X$ by linear subspaces of dimension at most $n$, and $\lambda_n(\bW,X)$ by linear methods of rank at most $n$.

In the present paper, we pay our attention mainly to the problem of linear sampling recovery and its optimality in terms of linear sampling $n$-widths for Sobolev classes of a mixed smoothness. However, it is useful and convenient to consider this problem together with the problem of linear and Kolmogorov $n$-widths.

Obviously, we have the inequalities
\begin{equation}\label{eq-relations}
d_n(\bW,X) \leq \lambda_n(\bW,X)\leq \varrho_n(\bW,X).
\end{equation}

Notice that  if $X$ is a Hilbert space, then 
$$
\lambda_n(\bW,X) = d_n(\bW,X).
$$

We are interested in linear sampling recovery and approximation of functions having a mixed smoothness on $\RRd$ which are $1$-periodic at each variable. It is convenient to consider them as functions defined in the $d$-torus $\TTd = [0,1]^d$ which is defined as the Castrian product of $d$ copies of the interval $[0,1]$ with the identification of the end points. To avoid confusion, we use the notation $\IId$ to denote the standard unit $d$-cube $[0,1]^d$. 
Notice that all the results presented in this section for  functions on $\TTd$ also hold true for functions on $\IId$ in appropriate forms.

Let us give a notion of periodic Sobolev space of mixed smoothness. We define the univariate Bernoulli kernel
\[ 
F_r(x)
:= \
1 + 2\sum_{k=1}^\infty k^{-r}\cos (\pi kx - r \pi/2), \quad x \in \TT,
\]
and the multivariate Bernoulli kernels as the corresponding tensor products 
\begin{equation} \label{F_r}
F_r(\bx)
:= \
\prod_{j=1}^d F_r(x_j),\quad \bx \in
\TTd.
\end{equation}
Let $r > 0$ and $1 \le p \le \infty$. 
Denote by  
$L_p(\TTd)$ the normed space 
of functions on $\TTd$ with the $p$th integral norm 
$\|\cdot\|_p$ for $1 \le p < \infty,$ and 
the ess sup-norm $\|\cdot\|_p$ for $p = \infty$, where we make use of the abbreviation $\|\cdot\|_p:= \|\cdot\|_{L_p(\TTd)}$ .
If $r > 0$ and $1 \le p \le \infty$, we define the Sobolev space $\Wrp$ of mixed smoothness $r$ by
\begin{equation} \label{def[Wrp]}
\Wrp
:= \
\Big\{f \in L_p(\TTd):\, 
f \, = \, F_r\ast \varphi:= \, \int_{\TTd}F_r(\bx-\by)\varphi(\by){d}\by ,\quad
\|\varphi\|_p < \infty \Big\},
\end{equation}
and the norm in this space by
$$
\|f\|_{\Wrp}:= \|\varphi\|_p
$$ 
for $f$ represented as in \eqref{def[Wrp]}.
For $1 < p < \infty$, the  space $\Wrp$ coincides with the
set of all $f\in L_p(\TTd)$ such that the norm
\[
\Big\| \sum_{\bs \in \ZZd} \, \hat{f}(\bs) (1+|s_1|^2)^{r/2} \, \ldots \,
(1+|s_d|^2)^{r/2}
\, e^{\pi i(\bs,\cdot)} \Big\|_p
\]
is finite, where $\hat f(\bs)$ denotes the usual $\bs$th Fourier coefficient of $f$. 

We next give a definition of Besov spaces of mixed smoothness $r$.
For univariate functions $f$ on $\TT$ the $\ell$th difference operator $\Delta_h^\ell$ is defined by 
\begin{equation} \label{def[Delta_h^ell]}
	\Delta_h^\ell(f,x) := \
	\sum_{j =0}^\ell (-1)^{\ell - j} \binom{\ell}{j} f(x + jh).
\end{equation}
Denote by $[d]$  the set of all natural numbers from $1$ to $d$.  If $u$ is any subset of $[d]$, for multivariate functions on $\TTd$
the mixed $(\ell,u)$th difference operator $\Delta_\bh^{\ell,u}$ is defined by 
\begin{equation*}
	\Delta_\bh^{\ell,u} := \
	\prod_{i \in u} \Delta_{h_i}^\ell, \quad \Delta_\bh^{\ell,\varnothing} := \ I,
\end{equation*}
where the univariate operator
$\Delta_{h_i}^\ell$ is applied to the univariate function $f$ by considering $f$ as a 
function of  variable $x_i$ with the other variables held fixed, and $I(f):= f$ for functions $f$ on $\TTd$. 
We also use the abbreviation $\Delta_\bh^\ell:= \Delta_\bh^{\ell,[d]}$.

For $u \subset [d]$, let
\begin{equation*}
\omega_\ell^u(f,\bt)_p:= \sup_{h_i < t_i, i \in u}\|\Delta_\bh^{\ell,u}(f)\|_p, \ \bt \in {\II}^d,
\end{equation*} 
be the mixed $(\ell,u)$th modulus of smoothness of $f$ (in particular,
$\omega_l^{\varnothing}(f,t)_p = \|f\|_p$).

If $0 <  p, \theta \le \infty$, 
$r> 0$ and $\ell > r$, 
we introduce the quasi-semi-norm 
$|f|_{B_{p, \theta}^{r,u}}$ for functions $f \in L_p(\TTd)$ by
\begin{equation*} 
|f|_{B_{p, \theta}^{r,u}(\TTd)}:= 
\begin{cases}
\ \left(\int_{{\II}^d} \{ \prod_{i \in u} t_i^{- r}
\omega_\ell^u(f,\bt)_p \}^ \theta \prod_{i \in u} t_i^{-1} {d} \bt \right)^{1/\theta}, 
& \theta < \infty, \\
\sup_{\bt \in {\II}^d} \ \prod_{i \in u} t_i^{-r}\omega_\ell^u(f,\bt)_p,  & \theta = \infty
\end{cases}
\end{equation*}
(in particular, $|f|_{B_{p, \theta}^{r,\emptyset}(\TTd)} = \|f\|_p$).

For $0 <  p, \theta \le \infty$ and $0 < r< l,$ the Besov space 
$B_{p, \theta}^r(\TTd)$ is defined as the set of  functions $f \in L_p(\TTd)$ 
for which the Besov quasi-norm $\|f\|_{B_{p, \theta}^r(\TTd)}$ is finite. 
The  Besov quasi-norm is defined by
\begin{equation*} 
\|f\|_{B_{p, \theta}^r(\TTd)}
:= \ 
\sum_{u \subset [d]} |f|_{B_{p, \theta}^{r,u}(\TTd) }.
\end{equation*}
We also call the particular space $H_p^r(\TTd): = B_{p, \infty}^r(\TTd)$ H\"older-Nikol'skii space.

\subsection{B-spline quasi-interpolation representations} 
\label{B-spline Q-I}

A central role in sparse-grid linear sampling recovery on Smolyak sparse  grids of functions having a mixed smoothness play sampling representations which are based on dyadic scaled B-splines with integer knots or trigonometric kernels and constructed from function values at dyadic lattices. These representations are in the form of a B-spline or trigonometric polynomial series provided with discrete equivalent norm for functions in  Sobolev and Besov spaces of a mixed smoothness. By employing them we can construct sampling algorithms for recovery on Smolyak sparse grids of functions from the corresponding spaces which in some cases give the asymptotically optimal rate of the approximation error. We refer the reader to \cite[Section 5]{DTU18B} for detail and bibliography on this topic.
In this subsection, we present a Littlewood-Paley-type theorem from \cite{DD2018} on B-spline quasi-interpolation sampling representation for periodic unweighted Sobolev and Besov spaces of a mixed smoothness. 

In order to construct B-spline quasi-interpolation sampling representations on $\TTd$, we auxiliarily  introduce quasi-interpolation operators for functions on $\RRd$. For a given natural number $\ell,$ denote by  $M_\ell$ the cardinal B-spline of order $\ell$ with support $[0,\ell]$ and 
knots at the points $0, 1,...,\ell$. 
We fixed an even number $\ell \in \NN$ and take the cardinal B-spline $M= M_{\ell}$ of order $\ell$.
Let $\Lambda = \{\lambda(j)\}_{|j| \le \mu}$ be a given finite even sequence, i.e., 
$\lambda(-j) = \lambda(j)$ for some $\mu \ge \frac{\ell}{2} - 1$. 
We define the linear operator $Q$ for functions $f$ on $\RR$ by  
\begin{equation} \label{def:Q}
	Q(f,x):= \ \sum_{s \in \ZZ} \Lambda (f,s)M(x-s), 
\end{equation} 
where
\begin{equation} \label{def:Lambda}
	\Lambda (f,s):= \ \sum_{|j| \le \mu} \lambda (j) f(s-j + \ell/2).
\end{equation}
The operator $Q$ is local and bounded in $C(\RR)$  (see \cite[p. 100--109]{Chui92}).
An operator $Q$ of the form \eqref{def:Q}--\eqref{def:Lambda} is called a 
{\it quasi-interpolation operator in} $C(\RR)$ if  it reproduces 
$\Pp_{\ell-1}$, i.e., $Q(f) = f$ for every $f \in \Pp_{\ell-1}$, where $\Pp_{\ell-1}$ denotes the set of $d$-variate polynomials of degree at most $\ell-1$ in each variable.

We present some well-known examples of quasi-interpolation operators. 
A piecewise linear quasi-interpolation operator is defined as
\begin{equation} \nonumber
Q(f,x):= \ \sum_{s \in \ZZ} f(s) M(x-s), 
\end{equation} 
where $M$ is the   
symmetric piecewise linear B-spline with support $[-1,1]$ and 
knots at the integer points $-1, 0, 1$.
It is related to the classical Faber-Schauder basis of the hat functions 
(see, e.g., \cite{DD2011a, Tri2010B}, for details). 
A quadratic quasi-interpolation operator is defined by
\begin{equation*} 
Q(f,x):= \ \sum_{s \in \ZZ} \frac {1}{8} \{- f(s-1) + 10f(s) - f(s+1)\} M(x-s), 
\end{equation*} 
where $M$ is the symmetric quadratic B-spline with support $[-3/2,3/2]$ and 
knots at the half integer points $-3/2, -1/2, 1/2, 3/2$.
Another example is the cubic quasi-interpolation operator 
\begin{equation*} 
Q(f,x):= \ \sum_{s \in \ZZ} \frac {1}{6} \{- f(s-1) + 8f(s) - f(s+1)\} M(x-s), 
\end{equation*} 
where $M$ is the symmetric cubic B-spline with support $[-2,2]$ and 
knots at the integer points $-2, -1, 0, 1, 2$.

If $Q$ is a quasi-interpolation operator of the form 
\eqref{def:Q}--\eqref{def:Lambda}, for $h > 0$ and a function $f$ on $\RR$, 
we define the operator $Q(\cdot;h)$ by
$$
Q(f;h) 
:= \ 
\sigma_h \circ Q \circ \sigma_{1/h}(f),
$$
where $\sigma_h(f,x) = \ f(x/h)$.
Let $Q$ be a quasi-interpolation operator of the form \eqref{def:Q}--\eqref{def:Lambda} in $C({\RR}).$ 
If $k \in \NN_0 $, we introduce the operator $Q_k$  by  
\begin{equation*}
	Q_k(f,x) := \ Q(f,x;h^{(k)}), \  x \in \RR, \quad h^{(k)}:= \ \ell^{-1}2^{-k}. 
\end{equation*}
We define the integer translated dilation $M_{k,s}$ of $M$ by   
\begin{equation*}
	M_{k,s}(x):= \ M(\ell2^k x - s), \ k \in {\ZZ}_+, \ s \in \ZZ. 
\end{equation*}
Then we have for $k \in {\ZZ}_+$,
\begin{equation} \nonumber
	Q_k(f)(x)  \ = \ 
	\sum_{s \in \ZZ} a_{k,s}(f)M_{k,s}(x), \ \forall x \in \RR, 
\end{equation}
where the coefficient functional $a_{k,s}$ is defined by
\begin{equation} \label{def[a_{k,s}(f)]}
	a_{k,s}(f):= \ \Lambda(f,s;h^{(k)}) 
	= \   
	\sum_{|j| \le \mu} \lambda (j) f(h^{(k)}(s-j+r)).
\end{equation}

For $\bk \in \ZZdp$, let the mixed operator $Q_\bk$ be defined by
\begin{equation} \label{def:Mixed[Q_bk]} 
	Q_\bk:= \prod_{i=1}^d  Q_{k_i},
\end{equation}
where the univariate operator
$Q_{k_i}$ is applied to the univariate function $f$ by considering $f$ as a 
function of  variable $x_i$ with the other variables held fixed.
We define the $d$-variable B-spline $M_{k,s}$ by
\begin{equation} \label{def:Mixed[M_{k,s}]}
	M_{\bk,\bs}(\bx):=  \ \prod_{i=1}^d M_{k_i, s_i}( x_i),  
	\ \bk \in \ZZdp, \ \bs \in {\ZZ}^d,
\end{equation}
where $\ZZdp:= \{\bs \in {\ZZ}^d: s_i \ge 0, \ i \in [d] \}$.
Then we have
\begin{equation*} 
	Q_\bk(f,\bx)  \ = \ 
	\sum_{s \in \ZZ^d} a_{\bk,\bs}(f)M_{\bk,\bs}(\bx), \quad \forall \bx \in \RRd, 
\end{equation*}
where $M_{\bk,\bs}$ is the mixed B-spline  defined in \eqref{def:Mixed[M_{k,s}]}, 
and
\begin{equation} \label{def:Mixed[a_{k,s}(f)]}
	a_{\bk,\bs}(f) 
	\ = \   
	\biggl(\prod_{j=1}^d a_{k_j,s_j}\biggl)(f),
\end{equation}
and the univariate coefficient functional
$a_{k_i,s_i}$ is applied to the univariate function $f$ by considering $f$ as a 
function of  variable $x_i$ with the other variables held fixed. 

Since $M(\ell\,2^k x)=0$ for every $k \in \NN_0$ and $x \notin (0,1)$, we can extend  the restriction to the interval $[0,1]$ of the  B-spline  
$M(\ell\,2^k\cdot)$ to an $1$-periodic function on the whole $\RR$. Denote this periodic extension by $N_k$ and define 
\begin{equation*}
	N_{k,s}(x):= \ N_k(x - h^{(k)}s), \ k \in {\ZZ}_+, \ s \in I(k), 
\end{equation*}
where $$I(k) := \{0,1,..., \ell2^k - 1\}.$$
We define the $d$-variable B-spline $N_{\bk,\bs}$ by
\begin{equation} \nonumber
	N_{\bk,\bs}(\bx):=  \ \bigotimes_{i=1}^d N_{k_i, s_i}( x_i),  \ \bk \in \ZZdp, \ \bs \in I(\bk),
\end{equation}
where $$ I(\bk):=\prod_{i=1}^d I(k_i).$$
Then we have for functions $f$ on $\TTd$,
\begin{equation} \label{def[periodicQI]}
	Q_\bk(f,\bx)  \ = \ 
	\sum_{\bs \in I(\bk)} a_{\bk,\bs}(f) N_{\bk,\bs}(x), \quad \forall x \in \TTd. 
\end{equation}
There holds the inequality
\begin{equation} \nonumber
	\|Q_\bk(f)\|_{C(\TTd)} 
	\ \ll \ 
\|f\|_{C(\TT)} , \ \ f \in C(\TTd). 
\end{equation}

For $\bk \in \ZZdp$, we write  $\bk  \to  \infty$ if $k_i  \to  \infty$ for $i \in [d]$). We have the convergence for every $f \in C(\TTd)$,
\begin{equation} \label{ConvergenceMixedQ_bk(f)}
	\|f - Q_\bk(f)\|_{C(\TTd)} \to 0 , \ \bk  \to  \infty.
\end{equation}

For convenience we define the univariate operator $Q_{-1}$ by putting $Q_{-1}(f)=0$ for all $f$ on $\II$. Let  the operators $q_\bk$ be defined 
in the manner of the definition \eqref{def:Mixed[Q_bk]} by
\begin{equation} \nonumber
	q_\bk := \ \prod_{i=1}^d \biggl(Q_{k_i}- Q_{k_i-1}\biggl), \ \bk \in \ZZdp. 
\end{equation}

From the equation 
$$Q_\bk = \sum_{\bk' \le \bk}q_{\bk'}$$
and \eqref{ConvergenceMixedQ_bk(f)} it is easy to see that 
a continuous function  $f$ has the decomposition
$$
f \ =  \ \sum_{\bk \in \ZZdp} q_\bk(f)
$$
with the convergence in the norm of $C(\TTd)$.  
From  the refinement equation for the B-spline $M$, in the univariate case, we can represent the component functions $q_\bk(f)$ as 
\begin{equation} \label{eq:RepresentationMixedQ_bk(f)}
	q_\bk(f) 
	= \ \sum_{\bs \in I(\bk)}c_{\bk,\bs}(f) N_{\bk,\bs},
\end{equation}
where $c_{\bk,\bs}$ are certain coefficient functionals of 
$f$. In the multivariate case, the representation  \eqref{eq:RepresentationMixedQ_bk(f)} holds true 
with the $c_{k,s}$ which are defined in the manner of the definition  
\eqref{def:Mixed[a_{k,s}(f)]} by
\begin{equation} \nonumber
	c_{\bk,\bs}(f) 
	\ = \   
	\biggl(\prod_{j=1}^d c_{k_j,s_j}\biggl)(f).
\end{equation}
 The following periodic B-spline quasi-interpolation representation for continuous functions on $\TTd$ was proven in \cite[Lemma 2.1]{DD2018}.
\begin{lemma} \label{lemma[representation]}
	Every continuous function $f$ on $\TTd$ is represented as B-spline series 
	\begin{equation} \label{eq:B-splineRepresentation}
		f \ = \sum_{\bk \in \ZZdp} \ q_\bk(f) = 
		\sum_{\bk \in \ZZdp} \sum_{\bs \in I(\bk)} c_{\bk,\bs}(f)N_{\bk,\bs}, 
	\end{equation}  
	converging in the norm of $C(\TTd)$, where the coefficient functionals $c_{\bk,\bs}(f)$ are explicitly constructed as
	linear combinations of at most $m_0$ of function
	values of $f$ for some $m_0 \in \NN$ which is independent of $\bk,\bs$ and $f$.
\end{lemma}

\begin{theorem} \label{DirectThm}
	Let $1 < p < \infty$ and $1/p < r < \ell$. Then every function $f \in \Wrp$ can be represented as the series \eqref{eq:B-splineRepresentation} converging in the norm of $\Wrp$, and there holds
	the inequality
	\begin{equation} \label{DirectIneq}
		\biggl\| \biggl(\sum_{\bk \in \ZZdp} 
		\Big|2^{r|\bk|_1} q_\bk(f)\Big|^2 \biggl)^{1/2}\biggl\|_p
		\ \ll \
		\|f\|_{\Wrp}. 
	\end{equation}
\end{theorem}

\begin{theorem} \label{InverseThm}
	Let $1 < p < \infty$ and $0 < r < \ell - 1$.
	Then for every function $f$ on $\TTd$  represented as a B-spline series 
	\begin{equation} \label{InverseThm:Representation}
		f \ = \sum_{\bk \in \ZZdp} \ q_\bk = 
		\sum_{\bk \in \ZZdp} \sum_{\bs \in I(\bk)} c_{\bk,\bs}N_{\bk,\bs}, 
	\end{equation}  
	we have $f \in \Wrp$ and
	\begin{equation} \nonumber
		\|f\|_{\Wrp}
		\ \ll \
		\biggl\| \biggl(\sum_{\bk \in \ZZdp} 
		\Big|2^{r|\bk|_1} q_\bk\Big|^2 \biggl)^{1/2}\biggl\|_p,
	\end{equation}
	whenever the right hand side is finite.
\end{theorem}

\begin{theorem}\label{EqNormThm}
	Let $1 < p < \infty$ and $1/p < r < \ell - 1$.
	Then we have
	\begin{equation} \nonumber
		\biggl\| \biggl(\sum_{\bk \in \ZZdp} 
		\Big|2^{r|\bk|_1} q_\bk(f)\Big|^2 \biggl)^{1/2}\biggl\|_p
		\ \asymp \
		\|f\|_{\Wrp},
		\quad  \forall f \in \Wrp. 
	\end{equation}
\end{theorem}
Theorems \ref{DirectThm}-- \ref{EqNormThm} were proven in \cite[Theorems 3.1--3.3]{DD2018}.

Theorems on B-spline quasi-interpolation sampling representations with discrete equivalent quasi-norm in terms of coefficient functionals have been proved in \cite{DD2009}--\cite{DD2016}, \cite{DU2015} for various non-periodic  Besov spaces and periodic Besov spaces (see also \cite[Section 5]{DTU18B} for comments and bibliography). We recall the periodic version from \cite[Corollary 3.1]{DD2018}.

\begin{theorem} 
	Let $0 < p, \theta \le \infty$ and  $1/p < r < \min\{2\ell, 2\ell - 1 + 1/p\}$. 
	Then a  periodic function $f \in B_{p, \theta}^r$ can be represented
	by the B-spline series \eqref{eq:B-splineRepresentation}
	satisfying the relation
	\begin{equation} \nonumber
		\biggl( \sum_{\bk \in \ZZdp} \
		2^{r|\bk|_1\theta}\|q_\bk(f)\|^{\theta}_p\biggl)^{1/\theta} 
		\ \asymp \ 
		\|f\|_{B_{p, \theta}^r}
	\end{equation}
	with the sum over $\bk$ changing to the supremum when $\theta = \infty$.
\end{theorem}

\subsection{Sampling recovery  by Smolyak sparse-grid algorithms}
\label{Sampling recovery}
Based on the B-spline quasi-interpolation representation \eqref{eq:B-splineRepresentation}, we construct linear sampling algorithms on  Smolyak sparse grids induced by partial sums of the series in \eqref{eq:B-splineRepresentation} as follows.
For $m \in \NN$, the well known periodic Smolyak grid of points $G^d(m) \subset \TTd$ is defined as
\begin{equation}  \nonumber
	G^d(m)
	:= \
	\{\bx = 2^{-\bk} \bs: \bk \in \NNd, \ |\bk|_1 = m, \ \bs \in I(\bk)\}.
\end{equation}
Here and in what follows, we use the notations: 
$\bx \by := (x_1y_1,..., x_dy_d)$; 
$2^\bx := (2^{x_1},...,2^{x_d})$
 for $\bx, \by \in {\RR}^d$;
$x_i$ denotes the $i$th coordinate of $\bx \in \RR^d$, i.e., $\bx =: (x_1,..., x_d)$. 

For $m \in \NN_0$, we define the operator $R_m$ by 
\begin{equation*}
	R_m(f) 
	:= \ 
	\sum_{|\bk|_1 \le m} q_\bk(f)
	\ = \
	\sum_{|\bk|_1 \le m} \ \sum_{\bs \in I(\bk)} c_{\bk,\bs}(f)\, N_{\bk,\bs}.
\end{equation*} 
For functions $f$ on $\TTd$, $R_m$ defines the linear sampling algorithm 
on the Smolyak grid $G^d(m)$ 
\begin{equation*} 
	R_m(f) 
	\ = \ 
	S_n(\bX_n,\bPhi_n,f) 
	\ = \ 
	\sum_{\by \in G^d(m)} f(\bx) \varphi_{\bx}, 
\end{equation*} 
where $n := \ |G^d(m)|$, $\bX_n:= \{\bx \in G^d(m)\}$, $\bPhi_n:= \{\varphi_{\bx}\}_{\bx \in G^d(m)}$  and for 
$\bx =  2^{-\bk} \bs$,  $\varphi_\bx$ are explicitly constructed as linear combinations of at most 
at most $m_0$ B-splines $N_{\bk,\bj}$ for some $m_0 \in \NN$ which is independent of $\bk,\bs,m$ and $f$.  The operator $R_m$ is also called Smolyak (sparse-gird) sampling algorithm initiated and used by him \cite{Smo63} for quadrature and interpolation for functions having mixed smoothness. It plays an important role in sampling recovery of multivariate functions and its applications (see \cite{BuGr04}, \cite{DTU18B} for comments and bibliography). 

\begin{theorem} \label{thm[|f -  R_m(f)|_p<]}
	Let $1 < p,q < \infty$ and $1/p < r < \ell$.
	Then we have
	\begin{equation} \nonumber
		\big\|f -  R_m(f) \big\|_q
		\ \ll \
		\|f\|_{\Wrp} \times
		\begin{cases}
			2^{-rm} m^{(d-1)/2}, \ & p \ge q, \\
			2^{-(r - 1/p + 1/q)m} \, \ & p < q,
		\end{cases}
		\quad  \forall f \in \Wrp. 
	\end{equation}
\end{theorem}

\begin{proof} This theorem was proven in \cite[Theorem 4.1]{DD2018}. To understand the basic role of Theorems \ref{DirectThm}--\ref{EqNormThm} playing in its proof, let us recall that short proof. Let $f$ be a function in $\Wrp$.  Since $r > 1/p$, $f$ is continuous on $\TTd$ and consequently, we obtain by Lemma~\ref{lemma[representation]}
	\begin{equation} \label{remainder}
		f -  R_m(f) \ = \sum_{|\bk|_1 > m} \ q_\bk(f) 
	\end{equation}
	with uniform convergence.
	
	We first consider the case $p \ge q$. Due to the inequality  $\|f\|_q  \le \|f\|_p $, it is sufficient prove this case of the theorem for $p=q$. From \eqref{remainder} and Theorem  \ref{DirectThm} we have
	\begin{equation*}
		\begin{split}
			\big\|f -  R_m(f) \big\|_p
			\ &= \
			\biggl\|\sum_{|\bk|_1 > m} \ q_\bk(f) \biggl\|_p
			\ \le \
			\biggl\|\biggl(\sum_{|\bk|_1 > m} 2^{-2r|\bk|_1} \biggl)^{1/2} 
			\biggl(\sum_{|\bk|_1 > m} \big|2^{r|\bk|_1}  q_\bk(f) \big|^2 \biggl)^{1/2}\biggl\|_p
			\\[1ex] 
			\ &\le \
			\biggl(\sum_{|\bk|_1 > m} 2^{-2r|\bk|_1} \biggl)^{1/2} 
			\biggl\|\biggl(\sum_{\bk \in \ZZdp} \big|2^{r|\bk|_1}  q_\bk(f) \big|^2 \biggl)^{1/2}\biggl\|_p
				\\[1ex] 
				\ &\ll \
			2^{-rm} m^{(d-1)/2}\,\|f\|_{\Wrp}.
		\end{split}  
	\end{equation*}
	We next consider the case $p < q$. From \cite[Lemma 3]{Be1974}  one can prove the inequality
	\begin{equation} \nonumber
		\|f\|_q
		\ \ll \
		\|f\|_{W^{1/p - 1/q}_p}(\TTd), \quad \forall f \in W^{1/p - 1/q}_p(\TTd).
	\end{equation}
	Hence, by \eqref{remainder}, Theorems  \ref{InverseThm} and \ref{DirectThm} we derive that
	\begin{equation*}
		\begin{split}
			\big\|f -  R_m(f) \big\|_q
			\ &\ll \
			\biggl\|\sum_{|\bk|_1 > m} \ q_\bk(f) \biggl\|_{W^{1/p - 1/q}_p(\TTd)}
			\ \ll \
			\biggl\|\biggl(\sum_{|\bk|_1 > m} \big|2^{(1/p - 1/q)|\bk|_1}  q_\bk(f) \big|^2 \biggl)^{1/2}\biggl\|_p
			\\[1ex] 
			\ &\le \
			2^{-(r - 1/p + 1/q)m}
			\biggl\|\biggl(\sum_{\bk \in \ZZdp} \big|2^{r|\bk|_1}  q_\bk(f) \big|^2 \biggl)^{1/2}\biggl\|_p
				\\[1ex] 
				\ &\ll \
			2^{-(r - 1/p + 1/q)m}\,\|f\|_{\Wrp}.
		\end{split}  
	\end{equation*}
	The theorem is completely proven.
	\hfill
\end{proof}

Throughout the present paper, for a normed space $X$, we denote by boldface letter $\bX$ in unit ball in $X$.

From Theorem~\ref{thm[|f -  R_m(f)|_p<]} and  the lower bounds in Theorem~\ref{theorem[r_n^s><]} below
we also obtain

\begin{corollary} 
	Let $1 < p,q < \infty$ and $1/p< r < \ell$.
	Then we have
	\begin{equation}  \nonumber
		\sup_{f \in \Urp} \big\|f -  R_m(f) \big\|_q
		\ \asymp \
		\begin{cases}
			2^{-rm} m^{(d-1)/2}, \ & p \ge q, \\
			2^{-(r - 1/p + 1/q)m} \, \ & p < q.
		\end{cases}
	\end{equation}
\end{corollary}

If the approximation error is measured in the norm of $L_\infty(\TTd)$, we have \cite[Theorem~4.3]{DD2018}
\begin{theorem} \label{theorem[|f -  R_m(f)|_infty><]}
	Let $1 < p < \infty$ and $1/p < r < \ell$.
	Then we have
	\begin{equation}  \nonumber
		\sup_{f \in \Urp} \big\|f -  R_m(f) \big\|_\infty
		\ \asymp \
		2^{-(r-1/p)m} m^{(d-1)(1-1/p)}.
	\end{equation}
\end{theorem}


It worths to notice the following. For approximation of functions from $\Urp$,	we could take the sampling operator 
$$
R_m^{\square}(f):= \sum_{|\bk|_\infty\le m} q_\bk(f)
$$ 
on the  traditional standard grid 
$$
G^d_\square(m):= \big\{ 2^{-\bk}\bs: |\bk|_\infty  = m, \  \bs \in I(\bk)\big\}.
$$ 
However, it is easy to verify that the error of approximation of $f \in \Urp$ by  $R_m^{\square}(f)$ is the same as  by $R_m(f)$. On the other hand, the sparsity of   the grid $G^d(m)$ in the operator $R_m(f)$, is much higher than the sparsity of the grids $G^d_\square(m)$ comparing 
$|G^d(m)| \asymp 2^m m^{d-1}$  with $|G^d_\square(m)| \asymp 2^{dm}$.  

A natural question arises are there  sampling algorithms that use the Smolyak sparse grids $G^d(m)$ which give  bounds better than $R_m(f)$ for the Sobolev class $\Urp$. To answer this question, 
let us introduce the Smolyak sampling $n$-width $\varrho^{{s}}_n(\bW)_q$ characterizing  optimality of sampling recovery  on Smolyak grids $G^d(m)$ with respect to the function class $\bW$ by
\begin{equation} \nonumber
\varrho^{\rm{s}}_n(\bW,L_q(\TTd))
\ := \ \inf_{|G^d(m)| \le n, \, \Phi_m} \  \sup_{f \in \bW} \, \|f - S_m(\bPhi_m,f)\|_q,
\end{equation}
where for a family $\bPhi_m = \{\varphi_\bx\}_{\bx \in G^d(m)}$ of functions we define the linear sampling algorithm
$S_m(\bPhi_m,\cdot)$ on Smolyak grids $G^d(m)$ by
\begin{equation*}
S_m(\bPhi_m,f)
\ = \
\sum_{\by \in G^d(m)} f(\bx) \varphi_\bx.
\end{equation*}
The upper letter $\rm{s}$ indicates that we restrict to Smolyak grids here. From the definition it follows that 
\begin{equation} \label{rho_n < rho_n^s}
	\varrho_n(\bW,L_q(\TTd))
	\ \le \
\varrho^{\rm{s}}_n(\bW,L_q(\TTd)).
\end{equation}

The following theorem \cite[Theorem 4.2]{DD2018} confirms the asymptotic optimality of the Smolyak sampling algorithms  $R_m$  for the sampling recovery by using  the sparse Smolyak grids $G^d(m)$ of the Sobolev class $\Urp$.

\begin{theorem} \label{theorem[r_n^s><]}
	Let $1 < p,q < \infty$ and $r > 1/p$. For $n \in \NN$, let $m_n$ be the largest integer number  such that 
	$|G^d(m_n)| \le n$.
	Then we have
	\begin{equation} \label{r_n^s><]}
	\sup_{f \in \Urp} \big\|f -  R_{m_n}(f) \big\|_q
	\asymp 
	\varrho^{\rm{s}}_n(\Urp,L_q(\TTd)) 
	\asymp 
	\begin{cases}
	\biggl(\frac{(\log n)^{d-1}}{n}\biggl)^{r} (\log n)^{(d-1)/2}, \ & p \ge q, 
	\\[1.5ex]
	\biggl(\frac{(\log n)^{d-1}}{n}\biggl)^{(r-1/p+1/q)},   \ & p < q.
	\end{cases}
	\end{equation}
\end{theorem}

Another interesting question is when the  Smolyak sampling algorithms $R_{m_n}$ given as in Theorem~\ref{theorem[r_n^s><]}, are asymptotically optimal for the linear sampling $n$-widths $\varrho_n(\Urp,L_q(\TTd))$. So far there are known only a few cases when the asymptotic optimality  of the Smolyak sampling algorithms $R_{m_n}$ is confirmed as in the following theorem. Moreover, to our knowledge, excepting Smolyak sparse-grid sampling algorithms  by using $B$-spine quasi-interpolation or de la Vall\'ee Poussin kernels, there is no asymptotically optimal linear sampling algorithms  of another type for $\varrho_n(\Urp,L_q(\TTd))$ with $1\le p,q \le \infty$.

\begin{theorem} \label{theorem[r_n,p<q]} Let $r > 1/p$ and 
	for $n \in \NN$, let $m_n$ be the largest integer number  such that 
	$|G^d(m_n)| \le n$.
	Then we have that
		\begin{itemize}
		\item[{\rm(i)}] for $1 < p < q \le 2$  or $2\le p < q < \infty$, 
	\begin{equation}  \label{[r_n,p<q]}
		\sup_{f \in \Urp} \big\|f -  R_{m_n}(f) \big\|_q
		\asymp 
	\varrho_n(\Urp,L_q(\TTd)) 
	\ \asymp \
	\biggl(\frac{(\log n)^{d-1}}{n}\biggl)^{(r-1/p + 1/q)}, and
	\end{equation}
\item[{\rm(ii)}] 
\begin{equation*} 
		\sup_{f \in \bW^r_2(\TTd)} \big\|f -  R_{m_n}(f) \big\|_\infty
	\asymp 
	\varrho_n(\bW^r_2(\TTd), L_\infty(\TTd))	
	\ \asymp \ 
\biggl(\frac{(\log n)^{d-1}}{n}\biggl)^{(r-1/2)}(\log n)^{(d-1)/2} \,.
\end{equation*}
\end{itemize}
\end{theorem}
In this theorem,	the claim(i) was proven in   
	\cite{DD2018} for $r > \max\{1/p,1/2\}$ and in 
	\cite{BU2017} for $r > 1/p$. The special case $p=2<q$ was proven in \cite{BDSU2016}.  The claim (ii) was proven in \cite{Tem1993}. 

For the H\"older-Nikol'skii function  classes $\bH^r_p(\TTd)$, we have the following right asymptotic of order of sampling $n$-widths.
\begin{theorem} \label{theorem[r_n,p<q-H]} Let $r > 1/p$ and 
	for $n \in \NN$, let $m_n$ be the largest integer number  such that 
	$|G^d(m_n)| \le n$.
	Then we have that
	\begin{itemize}
		\item[{\rm(i)}] for $1 < p < q \le 2$, 
		\begin{equation*}  \label{[r_n,p<q]}
			\sup_{f \in \bH^r_p(\TTd)} \big\|f -  R_{m_n}(f) \big\|_q
			\asymp 
			\varrho_n(\bH^r_p(\TTd),L_q(\TTd)) 
			\ \asymp \
			\biggl(\frac{(\log n)^{d-1}}{n}\biggl)^{(r-1/p + 1/q)}(\log n)^{(d-1)/q}, and
		\end{equation*}
		\item[{\rm(ii)}] 
		\begin{equation*} 
			\sup_{f \in \bH^r_\infty(\TT^2)} \big\|f -  R_{m_n}(f) \big\|_\infty
			\asymp 
			\varrho_n(\bH^r_\infty(\TT^2), L_\infty(\TT^2))	
			\ \asymp \ 
			\biggl(\frac{\log n}{n}\biggl)^r(\log n) \,.
		\end{equation*}
	\end{itemize}
\end{theorem}

In this theorem,	the claim (i) was proven in   \cite{DD1991}, and  the claim (ii) in \cite{Tem2015}. Moreover, the claim (i) was the first result on the wright asymptotic order of sampling $n$-widths for classes of functions having a mixed smoothness. We refer the reader to \cite[Section 5]{DTU18B} for results on sampling recovery and sampling $n$-widths for functions from  Besov spaces $B_{p, \theta}^r(\TTd)$.


\subsection{Sampling recovery in reproducing kernel Hilbert spaces}
\label{Sampling recovery in RK Hilbert spaces}

In the previous section, we presented various aspects of sampling recovery  by using Smolyak algorithms $R_m(f)$ on sparse grids $G^d(m)$, in particular, their asymptotic optimalities in terms of linear sampling $n$-widths  $\varrho^{\rm{s}}_n(\Urp,L_q(\TTd))$ and $\varrho_n(\Urp,L_q(\TTd))$. The right asymptotic order of  $\varrho_n(\Urp,L_q(\TTd))$ ($r > 1/p$) can be  achieved by  the Smolyak sampling algorithms  $R_{m_n}(f)$  in the cases $1 < p < q \le 2$  or $2\le p < q < \infty$ or $p = 2, \ q = \infty$. It is interesting to notice that all these cases are restricted by the strict inequality $p < q$. 

It is a dilemma that the problem of right asymptotic order of the sampling $n$-widths $\varrho_n(\Urp,L_p(\TTd))$ for $1 \le p \le \infty$ has been open for long time, see Open Problem 4.1 in \cite{DTU18B}. From very recent results of \cite{DKU2023} on inequality between the linear sampling  and Kolmogorov $n$-widths of the unit ball of a reproducing kernel Hilbert subspace of the space $L_2(\Omega; \nu)$ one can immediately deduce the right asymptotic order of  $\varrho_n(\bW^r_2,L_2(\TTd))$  which solved the outstanding Open Problem 4.1 in \cite{DTU18B} for the particular case $p=2$. This open problem is still  not solved for the case $p\not= 2$.  Unfortunately, even in the solved case $p=2$, we do not know any explicit asymptotically optimal linear sampling algorithm since its proof is based on an inequality between the linear sampling  and Kolmogorov $n$-widths. The problem of  construction of  asymptotically optimal linear sampling algorithms for this case is still open.

In this section, we present some results of \cite{DKU2023} and their consequences on right asymptotic order of  $\varrho_n(\Urp,L_q(\TTd))$ for the case 
$1 < q \le 2 \le  p < \infty$. 

Let $H$ be a separable reproducing kernel (RK) Hilbert space on a $\Omega$ and  $\nu$   a positive measure on $\Omega$ such that $H$ is compactly embedded into the space $L_2(\Omega; \nu)$. We say that $H$ satisfy the finite trace assumption if there holds the condition
\begin{equation}\label{trace assumption2}
	\int_{\Omega} K(\bx,\bx)  \nu(\rd \bx)  \ < \ \infty,
\end{equation}
where $K(\cdot,\cdot)$ is the RK of $H$. 

Recall that $\bH$ denotes the unit ball of $H$. From the definitions we already know the inequality 
$\varrho_n(\bH, L_2(\Omega; \nu)) 
\ge  d_n(\bH, L_2(\Omega; \nu)).
$
The following theorem claims an inverse inequality \cite[Theorem 1]{DKU2023}.

\begin{theorem} \label{thm: rho_n < sum d_n}
	Let $H$ be a separable RK  Hilbert space on the domain $\Omega$ and  $\mu$   a positive measure such that $H$ is compactly embedded into $L_2(\Omega; \nu)$. Assume that $H$ satisfy the finite trace assumption \eqref{trace assumption2}. Then there is a constant $\lambda \in \NN$ such that  
	\begin{equation}  \label{[r_n,p<q]}
		\varrho_{\lambda n}(\bH, L_2(\Omega; \nu))^2 
		\ \le \
		\frac{1}{n} \sum_{m \ge n} d_m (\bH, L_2(\Omega; \nu))^2.
	\end{equation}
\end{theorem}

From this theorem we can deduce an important consequence \cite[Corollary 2]{DKU2023}.

\begin{corollary} \label{corollary: rho_n <  d_n}
	Under the assumptions of Theorem \ref{thm: rho_n < sum d_n} assume that  
	\begin{equation}  \label{[d_n<]}
		d_n(\bH, L_2(\Omega; \nu))
		\ \ll \
		n^{-\alpha} log^{-\beta} n
	\end{equation}
for some $\alpha \ge 1/2$ and $\beta \in \RR$. Then
	\begin{equation}  \label{[r_n,p<q]}
		\varrho_n(\bH, L_2(\Omega; \nu))
		\ \ll \
		\begin{cases}
				n^{-\alpha} log^{-\beta} n  & \ \ \text{if} \ \ \alpha > 1/2, \\
				n^{-\alpha} log^{-\beta + 1/2} n & \ \ \text{if} \ \ \alpha = 1/2 \ \text{and} \ \beta > 1/2.
		\end{cases}
	\end{equation}
Moreover, there exists $\bH$ such that these bounds are sharp.
\end{corollary}

\begin{theorem} \label{theorem: rho_n><-p=2}
	Let $r > 1/2$. Then
	\begin{equation}  \label{[rho_n-p=2]}
		\varrho_n(\bW^r_2(\TTd), L_2(\TTd))
		\ \asymp \
		\biggl(\frac{(\log n)^{d-1}}{n}\biggl)^{- r}.
	\end{equation}
\end{theorem}

\begin{proof}
	This theorem can be considered a particular case of Theorem \ref{theorem: rho_n><,p<2} below. It is a consequence of Corollary  \ref{corollary: rho_n <  d_n}. However, for completeness, let us give an independent proof based on Corollary \ref{corollary: rho_n <  d_n}.
	Notice that $W^r_2(\TTd)$ is an RK Hilbert space. The RK of  $W^r_2(\TTd)$ is the Bernoulli kernel $K =F_r$ defined as in \eqref{F_r}. For $r > 1/2$, the space $W^r_2(\TTd)$ satisfies the finite trace assumption \eqref{trace assumption2}.
	Hence, this corollary directly follows from Corollary \ref{corollary: rho_n <  d_n} and the known result 
	\begin{equation}  \label{[rho_n-p=2]}
		d_n(\bW^r_2(\TTd), L_2(\TTd))
		\ \asymp \
		\biggl(\frac{(\log n)^{d-1}}{n}\biggl)^{- r},
	\end{equation}
see, e.g., \cite[Theorem 4.3.1]{DTU18B}.
\hfill
\end{proof}	

Theorem \ref{thm: rho_n < sum d_n} can be extended to another situation for $\varrho_n(\bW, L_2(\Omega; \nu))$ when $\bW$ is not the unit ball of an RK Hilbert space. This extension allows to obtain a result more general than Theorem \ref{theorem: rho_n><-p=2}.

{\bf Assumption A.} Let $\bW$ be a class of complex-valued functions on the set $\Omega$. We say that $\bF$ satisfies Assumption A, if there is a metric on $\bW$ such that $\bW$ is continuously embedded into 
$L_2(\Omega; \nu)$, separable, and for each $\bx \in \Omega$, the function evaluation $f \mapsto f(\bx)$ is continuous on $\bW$.

The following theorem and corollary have been proven also  in \cite[Theorem~3 and Corollary~4]{DKU2023} (see also \cite[Proposition 11]{KPUU2023} for a slight  improvement).

\begin{theorem} \label{thm: rho_n < sum d_n,p<2}
	Assume that $F$ satisfies Assumption A. Then for every $0 < \tau < 2$, there is a constant $\lambda \in \NN$ such that  
	\begin{equation}  \label{[r_n,p<q]}
		\varrho_{\lambda n}(\bW, L_2(\Omega; \nu))^\tau
		\ \le \
		\frac{1}{n} \sum_{m \ge n} d_m (\bW, L_2(\Omega; \nu))^\tau.
	\end{equation}
\end{theorem}

\begin{corollary} \label{corollary: rho_n <  d_n,p<2}
	Assume that $\bW$ satisfies Assumption A and that  
	\begin{equation}  \label{[d_n<]}
		d_n(\bW, L_2(\Omega; \nu))
		\ \ll \
		n^{-\alpha} log^{-\beta} n
	\end{equation}
	for some $\alpha > 0$ and $\beta \in \RR$. Then
	\begin{equation}  \label{[r_n,p<q]}
		\varrho_n(\bW, L_2(\Omega; \nu))
		\ \ll \
		\begin{cases}
			n^{-\alpha} log^{-\beta} n & \ \ \text{if} \ \ \alpha > 1/2, \\
				n^{-\alpha} log^{-\beta + 1/2} n & \ \ \text{if} \ \ \alpha = 1/2 \ \text{and} \ \beta > 1, \\
			1 & \ \ \text{otherwise}.
		\end{cases}
	\end{equation}
	Moreover, there exists $H$ such that these bounds are sharp.
\end{corollary}

From these results we obtain

\begin{theorem} \label{theorem: rho_n><,p<2}
	Let $r > 1/2$ and $1 < q \le 2 \le  p < \infty$. Then
	\begin{equation}  \label{[rho_n-p=2]}
		\varrho_n(\bW^r_p(\TTd), L_q(\TTd))
		\ \asymp \
		\biggl(\frac{(\log n)^{d-1}}{n}\biggl)^{- r}.
	\end{equation}
\end{theorem}

\begin{proof}
	Notice that for $r > 1/2$ and $q=2$, the set $\bW^r_p(\TTd)$ satisfies Assumption A.
	Hence, the asymptotic order \eqref{[rho_n-p=2]} for $q=2$ directly follows from Corollary \ref{corollary: rho_n <  d_n,p<2} and the known result 
	\begin{equation}  \label{d_n-p >q}
		d_n(\bW^r_p(\TTd), L_q(\TTd))
		\ \asymp \
	\biggl(\frac{(\log n)^{d-1}}{n}\biggl)^{- r}
	\end{equation}
for $1< q \le p < \infty$,	see, e.g., \cite[Theorem 4.3.1]{DTU18B}. The case $1 < q < 2$ is implied from the case $q =2$, \eqref{d_n-p >q} and the inequalities 
	\begin{equation*} 
	\varrho_n(\bW^r_p(\TTd), L_q(\TTd))
	\ \ll \
	\varrho_n(\bW^r_p(\TTd), L_2(\TTd)),
\end{equation*}
and
	\begin{equation*}  
	\varrho_n(\bW^r_p(\TTd), L_q(\TTd))
	\ \gg \
	d_n(\bW^r_p(\TTd), L_q(\TTd)).
\end{equation*}
	\hfill 
\end{proof}	

We have a similar result for the H\"older-Nikol'skii  classes. 

\begin{theorem} \label{theorem: rho_n><,p<2,H}
	Let $r > 1/2$ and $1 < q \le 2 \le  p < \infty$. Then
	\begin{equation}  \label{[rho_n-p=2-H]}
		\varrho_n(\bH^r_p(\TTd), L_q(\TTd))
		\ \asymp \
		\biggl(\frac{(\log n)^{d-1}}{n}\biggl)^{- r}(\log n)^{(d-1)/2}.
	\end{equation}
\end{theorem}

\begin{proof}
	This corollary can be proven in a way similar to the proof of Theorem \ref{theorem: rho_n><,p<2} with a certain modification, based on Corollary \ref{corollary: rho_n <  d_n,p<2}. In particular, \eqref{d_n-p >q} is replaced with 
		\begin{equation*}  \label{d_n-p >q-H}
		d_n(\bH^r_p(\TTd), L_q(\TTd))
		\ \asymp \
		\biggl(\frac{(\log n)^{d-1}}{n}\biggl)^{- r}(\log n)^{(d-1)/2}
	\end{equation*}
	for $1 < q \le 2 \le  p < \infty$,	see, e.g., \cite[Theorem 4.3.10]{DTU18B}.
	\hfill 
\end{proof}	

We conjecture that the right asymptotic orders \eqref{[rho_n-p=2]} and  \eqref{[rho_n-p=2-H]} are still hold true for $r > 1/p$ and $1 < q \le p < \infty$.

\subsection{Different concepts of optimality in sampling recovery}
\label{Optimalities in sampling recovery}
Optimality in sampling recovery can be understood in different ways in terms of sampling widths. It depends on the restriction of sampling algorithms we consider. In the previous sections, optimalities of sampling recovery based on $n$ function values, are treated in terms of linear sampling $n$-widths $\varrho_n(\bW,X)$ and Smolyak sampling $n$-widths $\varrho^{\rm{s}}_n(\bW,X)$. For the first sampling $n$-widths, we are restricted with linear sampling algorithms, and for the second $n$-widths, with linear sampling algorithms on the Smolyak grids.  Theorems  \ref{theorem[r_n,p<q]}, {\ref{theorem: rho_n><,p<2} and \ref{theorem[r_n^s><]} show that for the Sobolev
class of mixed smoothness $\Urp$, the asymptotic orders of these sampling $n$-widths coincide in some cases and differ in other cases. In this subsection, we consider two more sampling widths which characterize other optimalities of sampling recovery.

The Kolmogorov sampling $n$-width  of the set $\bW$ in $X$ as
$$
\varrho_n^{\rm{k}}(\bW,X):=\inf_{\bx_1,\ldots,\bx_n\in \Omega,\atop
	A_n: \RR^n \to L_n, \ \dim L_n \le n} \ \sup_{f\in \bW}
\|f- A_n(f(\bx_1),\ldots,f(\bx_n))\|_X,
$$
where the inf is taken over all collections $(\bx_k)_{k=1}^n$ of $n$ points in $\Omega$, all mappings $A_n: \RR^n \to L_n$ and all linear subspaces $L_n \subset X$ of dimension at most $n$.

The absolute sampling $n$-width  of the set $\bW$ in $X$ as
$$
\varrho_n^{\rm{abs}}(\bW,X):=\inf_{\bx_1,\ldots,\bx_n\in \Omega,\atop
	A_n: \RR^n \to  X} \ \sup_{f\in \bW}
\|f- A_n(f(\bx_1),\ldots,f(\bx_n))\|_X,
$$
where the inf is taken over all collections $(\bx_k)_{k=1}^n$ of $n$ points in $\Omega$, and all mappings $A_n: \RR^n \to X$.

The sampling $n$-widths $\varrho_n^{\rm{k}}$  and $\varrho_n$   are  inspired by the concepts of the Kolmogorov  $n$-widths and the linear  $n$-widths. 
The sampling $n$-widths $\varrho_n^{\rm{k}}$ was introduced in \cite{DD1990}, $\varrho_n$  in \cite{Tem1993}, $\varrho_n^{\rm{abs}}$  in  \cite{TWW1988B}, and $\varrho_n^{\rm{s}}$  in \cite{DU2015}.

 From the definitions we can see that
\begin{equation} \label{sampling width inequalities}
\varrho_n^{\rm{abs}}(\bW,X)
\le \varrho_n^{\rm{k}}(\bW,X)
\le
\varrho_n(\bW,X)
\le 	\varrho^{\rm{s}}_n(\bW,X).
\end{equation}

For given $d$, $r$ and $p,q$, we  temporarily use the abbreviation:
$$
\rho_n := \rho_n(\bW^r_p(\TTd), L_q(\TTd)),
$$
where $\rho_n$ denotes one of $\varrho_n^{\rm{abs}}$, $\varrho_n^{\rm{k}}$, $\varrho_n$ and $\varrho_n^{\rm{s}}$.

For the univariate Sobolev class $\bW^r_p(\TT)$, it is known that  if $d=1$, $1 \le p,q \le \infty$ and $r > 1/p$, then we have that
\begin{equation} \nonumber
	\varrho_n^{\rm{abs}}
	\ \asymp \
	\varrho_n^{\rm{k}}
	\asymp \
	\varrho_n
	\asymp
	\varrho^{\rm{s}}_n
	\ \asymp \
	n^{- r - (1/p - 1/q)_+},
\end{equation}
showing these sampling $n$-widths have the same asymptotic order (see, e.g., \cite{DTU18B}, \cite{Tem18B}, \cite{NT2006}). The picture may be changed for the multivariate Sobolev class of mixed smoothness $\bW^r_p(\TTd)$ with $d>1$. We list some known particular cases when the asymptotic orders of some of these sampling $n$-widths coincide.

\begin{theorem} \label{theorem[r_n,p<q-Opt]} 
Let $d >1$. 	Then we have that
	\begin{itemize}
		\item[{\rm(i)}] for  $r> 1/p$ and $1 < p < q \le 2$, 
		\begin{equation}  \label{[r_n,p<q-Opt]}
			\varrho_n^{\rm{k}}
			\asymp \
			\varrho_n
			\asymp
			\varrho^{\rm{s}}_n 
			\ \asymp \
			\biggl(\frac{(\log n)^{d-1}}{n}\biggl)^{(r-1/p + 1/q)},
		\end{equation}
		\item[{\rm(ii)}] for $r > 1/2$ and $p=2$ and $q = \infty$,
		\begin{equation*} 
				\varrho_n^{\rm{abs}}
			\ \asymp \
			\varrho_n^{\rm{k}}
		\asymp \
		\varrho_n
		\asymp
		\varrho^{\rm{s}}_n 
			\ \asymp \ 
			\biggl(\frac{(\log n)^{d-1}}{n}\biggl)^{(r-1/2)}(\log n)^{(d-1)/2},
		\end{equation*}
		\item[{\rm(iii)}] for $r > 1/2$ and $1 < q \le 2 \le  p < \infty$,
	\begin{equation*}  \label{[r_n,p<q]}
			\varrho_n^{\rm{abs}}
		\ \asymp \
		\varrho_n^{\rm{k}}
		\asymp \
		\varrho_n
		\asymp
		\biggl(\frac{(\log n)^{d-1}}{n}\biggl)^{- r}.
	\end{equation*}
	\end{itemize}
\end{theorem}
\begin{proof}
The equality  	$\varrho_n^{\rm{abs}} =  \varrho_n$ in the claim (ii) follow from \cite[Proposition 13]{CW2004}, and 
the inequality $\varrho_n \ll	\varrho_n^{\rm{abs}}$ for $r > 1/2$ and $1 < q \le 2 \le  p < \infty$  from \cite[Proposition 14]{CW2004}. Hence this theorem immediately follows from the inequalities \eqref{eq-relations} and \eqref{sampling width inequalities}, Theorem \ref{theorem[r_n^s><]} and known results on Kolmogorov $n$-widths $d_n(\bW^r_p(\TTd), L_q(\TTd))$ (see \cite[Section 4.3]{DTU18B}).
	\hfill 
\end{proof}	

Results of the very recent papers \cite{JUV2023, DT2024} show that in some cases the absolute sampling $n$-widths may decay faster than the Kolmogorov sampling $n$-widths if the dimension $d$ is large sufficiently. In particular, it has been proven \cite{DT2024} the following result.

\begin{theorem} \label{theorem[r_n,p<2-Opt]} 
	Let $d >1$, $r> 1/p$ and $1 < p <  2$, $q=2$. Then we have that
		\begin{equation}  \label{[r_n,p<2]}
			\varrho_n^{\rm{abs}}
			\ \ll \
			\biggl(\frac{(\log n)^{d-1}}{n}\biggl)^{(r-1/p + 1/2)}(\log n)^{-(d-1)(1/p - 1/2) + 3r}.
				\end{equation}
\end{theorem}
If $d > \frac{3r}{1/p-1/2} + 1$, then from \eqref{[r_n,p<q-Opt]} and \eqref{[r_n,p<2]}, one can see that 
\begin{equation*} 
	\varrho_n^{\rm{abs}}
	\ \ll \
(\log n)^{-\alpha}	\varrho_n^{\rm{k}}
\end{equation*}
with $\alpha:= (d-1)(1/p - 1/2) + 3r>0$. 

 Excepting the cases in the claims (ii) and (iii) of   Theorem \ref{theorem[r_n,p<q-Opt]}, the problem of right asymptotic order of the absolute sampling $n$-widths $\varrho_n^{\rm{abs}}$ is still open for $d>1$, $1 \le p,q \le \infty$ and $r > 1/p$.

\section{Weighted sampling recovery}
\label{Weighted sampling recovery}

\subsection{Introducing remarks}
\label{Introducing remarks Sec 3}
In the previous section, we  considered the problem of linear sampling  recovery for functions in unweighted Sobolev spaces $W^r_p(\TTd)$.
In this section, we consider the problems of linear sampling recovery and linear approximation of functions in weighted Sobolev spaces $W^r_p(\RRd;\mu)$. The approximation error is measured by the norm of the weighted space $L_q(\RRd;\mu)$. The optimality of of sampling recovery and linear approximation is treated in terms of linear sampling, linear and Kolmogorov $n$-widths. 

We first introduce weighted  Sobolev spaces of  mixed smoothness.  
Let  $\Omega \subset \RRd$ be a Lebesgue measurable set. Let $v$ be a nonnegative Lebesgue measurable  function on $\Omega$. Denote by $\mu_v$ the measure on $\Omega$ defined via the density function $v$, i.e., for every Lebesgue measurable set $A \subset \Omega$,
$$
\mu_v(A) = \int_A v(\bx) \rd \bx.
$$

For $1 \le p < \infty$, let $L_p(\Omega;\mu_v)$ be the weighted space  of all Lebesgue measurable functions $f$ on $\Omega$ such that the norm
\begin{align} \label{L-Omega}
\|f\|_{L_p(\Omega;\mu_v)} : = 
\bigg( \int_\Omega |f(\bx)|^p  \mu_v(\rd \bx)\bigg)^{1/p} 
=
\bigg( \int_\Omega |f(\bx)|^p v(\bx) \rd \bx\bigg)^{1/p} 
\end{align}
is finite. Due to \eqref{L-Omega}, the density function $v$ is called also a weight and the space $L_p(\Omega;\mu_v)$  a weighted space. For $r \in \NN$, we define the weighted  Sobolev space $W^r_p(\Omega;\mu_v)$ of mixed smoothness $r$  as the normed space of all functions $f\in L_p(\Omega;\mu_v)$ such that the weak (generalized) partial derivative $D^{\bk} f$ belongs to $L_p(\Omega;\mu_v)$ for  every $\bk \in \ZZdp$ satisfying the inequality $|\bk|_\infty \le r$. The norm of a  function $f$ in this space is defined by
\begin{align} \label{W-Omega}
\|f\|_{W^r_p(\Omega;\mu_v)}: = \Bigg(\sum_{|\bk|_\infty \le r} \|D^{\bk} f\|_{L_p(\Omega;\mu_v)}^p\Bigg)^{1/p}.
\end{align}
We write $L_p(\Omega):= L_p(\Omega;\mu_v)$ and $W^r_p(\Omega):=W^r_p(\Omega;\mu_v)$ if  $v(\bx) = 1$ for every $\bx \in \Omega$.

In this section, we are interested in  approximation of functions from  $W^r_p(\Omega;\mu_w)$, where the density function or equivalently, the weight   
\begin{equation} \label{w(bx)}
w(\bx):= w_{\lambda,a,b}(\bx) := \prod_{i=1}^d w(x_i),
\end{equation}
is the tensor product of univariate Freud-type weights
\begin{equation} \label{w(x)}
w(x):= 	w_{\lambda,a,b}(x) := \exp (- a|x|^\lambda + b), \ \ \lambda > 1, \ \ a > 0, \ \ b \in \RR.
\end{equation} 
The approximation error is measured by the norm of the weighted space $L_q(\Omega;\mu_w)$ for $1\le q < \infty$.   The most important parameter in the weight $w_{\lambda,a,b}$ is $\lambda$. The parameter $b$ which produces only a positive constant in the weight $w_{\lambda,a,b}$  is introduced for a certain normalization for instance, for the standard Gaussian weight (see \eqref{Gassian weight} below) which is one of the most important weights. In what follows,  we fix the  parameters $\lambda,a,b$  and  for simplicity, drop them from the notation. As the weight  $w=w_{\lambda,a,b}$ is fixed, we also use the abbreviations: 
$$
L_p(\Omega;\mu):= L_p(\Omega;\mu_w), \ \ \ 
W^r_p(\Omega;\mu):=W^r_p(\Omega;\mu_w).
$$

The standard $d$-dimensional Gaussian measure $\gamma$ with the density function 
\begin{equation} \label{Gassian weight}
g(\bx) = (2\pi)^{-d/2}\exp (- |\bx|_2^2/2).
\end{equation}
is a particular case of measure $\mu_w$.
The well-known  spaces $L_p(\RRd;\gamma)$ and $\Wap(\RRd; \gamma)$ 
are used in many applications.

Notice that any function $f \in W^r_p(\RRd;\mu)$ is equivalent  in the sense of the Lesbegue measure to a continuous (not necessarily bounded) function on $\RRd$ (see \cite[Lemma 3.1]{DD2023}). Hence in what follows   we always assume that the functions $f \in W^r_p(\RRd;\mu)$ are  continuous. We need this assumption for  well defining  sampling algorithms and quadratures for functions $f \in W^r_p(\RRd;\mu)$.

The problems of linear sampling recovery and linear approximation of functions in Gaussian-weighted Sobolev spaces $\Wap(\RRd; \gamma)$ have been studied in \cite{DK2023}.  For the set $\bW^r_p(\RRd;\gamma)$, $r \in \NN$,  and Gaussian-weighted space $L_q(\RRd;\gamma)$, it was proven   in \cite{DK2023}  the following.
 \begin{itemize}
 	\item[(i)] 
 	\begin{equation}\label{rho_n:p=q=2}
 	\varrho_n(\bW^r_2(\RRd;\gamma),L_2(\RRd;\gamma))
 	\asymp 
 	n^{-\frac{r}{2}} (\log n)^{\frac{(d-1)r}{2}} \ \ (r \ge 2),
 	\end{equation}
 	and   
 \begin{equation}\label{lambda_n:p=q=2}
 \lambda_n(\bW^r_2(\RRd;\gamma),L_2(\RRd;\gamma)) 
 =
 d_n(\bW^r_2(\RRd;\gamma),L_2(\RRd;\gamma))
 \asymp 
 n^{-\frac{r}{2}} (\log n)^{\frac{(d-1)r}{2}}, 
 \end{equation}	
 	\item[(ii)] 
 	For $2 < p<\infty$,  
 	\begin{align}\label{rho_n:q=2}
 	\varrho_n(\bW^r_p(\RRd;\gamma),L_2(\RRd;\gamma))
 	\asymp 
 	n^{-r} (\log n)^{(d-1)r},
 	\end{align}
and for  $1\le q < p<\infty$,
\begin{equation}\label{lambda_n:p>q}
\lambda_n(\bW^r_p(\RRd;\gamma),L_q(\RRd;\gamma))
\asymp		
d_n(\bW^r_p(\RRd;\gamma),L_q(\RRd;\gamma))
\asymp 
n^{-r} (\log n)^{(d-1)r}.
\end{equation}			
\end{itemize}	

The proofs of the asymptotic orders in \eqref{rho_n:p=q=2}--\eqref{lambda_n:p>q} require quite different techniques. 

The proof of \eqref{lambda_n:p>q} for the case $1 \leq q <p <\infty$ is based on a linear approximation in the space $L_q(\IId)$ of functions from the classical unweighted  Sobolev spaces $W^r_p(\IId)$. By assembling of integer shifts of this approximation we construct  a linear approximation   in the space $L_q(\IId; \gamma)$ of functions from Gaussian-weighted Sobolev spaces $W^r_p(\RRd;\gamma)$ which preserves   the  convergence rate. 

The proof of \eqref{lambda_n:p=q=2}  is completely different from the proof of the case $1 \leq q <p <\infty$. It is similar to the hyperbolic cross trigonometric approximation  in the Hilbert space $\tilde{L}_2(\IId)$ of periodic functions from the Sobolev space $\tilde{W}_2^\alpha(\IId)$ (see, e.g., \cite{DTU18B} for detail). Here, the approximation is based on finite truncations of the Hermite polynomial expansion of functions to be approximated. 

The results \eqref{rho_n:p=q=2} and \eqref{rho_n:q=2}  are deduced from \eqref{lambda_n:p=q=2} and
\eqref{lambda_n:p>q} and the inequalities beween sampling and Kolmogorov $n$-widths in Corollaries \ref{corollary: rho_n <  d_n}  and \ref{corollary: rho_n <  d_n,p<2}, respectively.

It is interesting to compare the results \eqref{rho_n:p=q=2}--\eqref{lambda_n:p>q} with the known results for the corresponding unweighted cases. For the unweighted class $\bW^r_p(\IId)$, $r \in \NN$,  and unweighted space $L_q(\IId)$, we have the following.
 \begin{itemize}
 
 	\item[(iii)] 
 	\begin{equation}\label{rho_n:p=q=2-unweighted}
 	\varrho_n(\bW^r_2(\IId),L_2(\IId))
 	\asymp 
 	n^{-r } (\log n)^{(d-1)r},
 	\end{equation}
 	and 
 	\begin{equation}\label{lambda_n:p=q=2-unweighted}
 	\lambda_n(\bW^r_2(\IId),L_2(\IId)) 
 	=
 	d_n(\bW^r_2(\IId),L_2(\IId)
 	\asymp 
 	n^{-r} (\log n)^{(d-1)r}, 
 	\end{equation}	
 	
	\item[(iv)] For $2 <  p < \infty$,
	\begin{align}\label{rho_n:q=2-unweighted}
	\varrho_n(\bW^r_p(\IId),L_2(\IId))
	\asymp 
	n^{-r} (\log n)^{(d-1)r}.
	\end{align}
	and for  $1\le q < p<\infty$,
	\begin{equation}\label{lambda_n:p>q-unweighted}
	\lambda_n(\bW^r_p(\IId),L_q(\IId))
	\asymp		
	d_n(\bW^r_p(\IId),L_q(\IId))
	\asymp 
	n^{-r} (\log n)^{(d-1)r}.
	\end{equation}			
\end{itemize}


Inspecting \eqref{rho_n:p=q=2}--\eqref{lambda_n:p>q} and \eqref{rho_n:p=q=2-unweighted}--\eqref{lambda_n:p>q-unweighted}, we can see that there is a significant difference between the cases $p = q = 2$ and the cases $p > q$. If $p > q$, the asymptotic orders of the $n$-widths $\varrho_n$, $\lambda_n$ and $d_n$ are the same for the weighted and unweighted cases. If $p = q = 2$, the  asymptotic orders of these $n$-widths
in the weighted case are twice worse than in the unweighted case. 

In this section,  we extend  the results \eqref{rho_n:q=2} and \eqref{lambda_n:p>q} to the more general case -- the weighted class $\Wpgamma$ and weighted space $L_q(\RRd;\mu)$ for the measure $\mu$ associated with Freud-type weights as the density function given by \eqref{w(bx)}.

For technical  convenience we use the conventions  for $n \in \RR_1$:
  $$
  \varrho_n(F,X) := \varrho_{\lfloor n \rfloor}(F,X),
  $$
$$
d_n(F,X) := d_{\lfloor n \rfloor}(F,X), \ \ \lambda_n(F,X) := \lambda_{\lfloor n \rfloor}(F,X),
$$
and
 $$
A_n := A_{\lfloor n \rfloor}, \ \ R_n := R_{\lfloor n \rfloor}.
$$

\subsection{Linear sampling recovery and approximation}
\label{The case 1 le q < p < infty}
In this subsection, we give shorten proofs of  the results  \eqref{rho_n:p=q=2} and \eqref{lambda_n:p=q=2}  by using of the Hilbert space structure of the Gaussian-weighted  spaces $\bW^r_2(\RRd;\gamma)$ and $L_2(\RRd;\gamma )$ and Corollary \ref{corollary: rho_n <  d_n}, emphasizing possibility of application of the results in \cite{DKU2023} to different situations. We also  extend the results  \eqref{rho_n:q=2} and \eqref{lambda_n:p>q}  to the weighted function class $\BWpgamma$ and weighted space $L_q(\RRd;\mu)$ by using a modification of the  technique from \cite{DK2023}. 

%


We first give shorten proofs of  the results  \eqref{rho_n:p=q=2} and \eqref{lambda_n:p=q=2}. Let $(p_m)_{m \in \NN_0}$ be the sequence of orthonormal polynomials with respect to the univariate Gaussian weight $g(x) = (2\pi)^{-1/2}\exp (- |x|^2/2)$ .
For every multi-degree $\bk\in \ZZdp$, the $d$-variate 
polynomial $H_\bk$ is defined by
\begin{equation*}\label{H_bk}
	p_\bk(\bx) :=\prod_{j=1}^d p_{k_j}(x_j),
	\;\; \bx\in \RRd.
\end{equation*}

It is well-known that the  polynomials $\brab{p_\bk}_{\bk \in \ZZdp}$ constitute an orthonormal basis of the Hilbert space $L_2(\RRd;\gamma)$.
In particular,  every $f \in L_2(\RRd;\gamma)$ can be represented by the series 
\begin{equation}\label{H-series}
	f = \sum_{\bk \in \ZZdp} \hat{f}(\bk) p_\bk \ \ {\rm with} \ \ \hat{f}(\bk) := \int_{\RRd} f(\bx)\, p_\bk(\bx)w(\bx)\rd \bx
\end{equation}
converging in the norm of $L_2(\RRd;\gamma)$, and in addition, there holds  Parseval's identity
\begin{equation}\label{P-id}
	\norm{f}{L_2(\RRd;\mu)}^2= \sum_{\bk \in \ZZdp} |\hat{f}(\bk)|^2.
\end{equation}

For $r \in \NN_0$ and $\bk \in \ZZdp$, we define
\begin{equation*}\label{rho_bk}
	\rho_{r,\bk}: = \prod_{j=1}^d \brac{k_j + 1}^r.
\end{equation*}
For any $r >0$. Denote by $\Hh^r_w$ the space of all   functions $f \in L_2(\RRd;\gamma)$ represented by the  series \eqref{H-series} for which  the norm
\begin{equation}\label{Hh-norm}
	\norm{f}{\Hh^r_g} := \brac{\sum_{\bk \in \ZZdp} \rho_{r,\bk}|\hat{f}(\bk)|^2}^{1/2}
\end{equation}
is finite.

For the proof of the following  norm equivalence  see \cite[Lemma 3.4]{DK2023} (cf. also \cite[pages 687--689]{DILP18}). 
\begin{lemma}\label{lemma:N-eq}
	Let $r \in \NN_0$. Then we have the norm equivalence
	\begin{equation}\label{N-eq}
		\norm{f}{\Wa}
		\asymp 
		\norm{f}{\Hh^r_g}.
	\end{equation}
\end{lemma}

Due to this norm equivalence, we identify the space  $W^r_2(\RRd,\gamma)$ with the space $\Hh^r_g$ for $r \in \NN$.

\begin{theorem} 	\label{theorem:widths:p=q=2}
Let $r >0$. Then we have  the right asymptotic orders
\begin{align}\label{widths:p=q=2}
	\lambda_n(\boldsymbol{\Hh}_g^r, L_2(\RRd;\gamma)) 
	=
	d_n(\boldsymbol{\Hh}_g^r, L_2(\RRd;\gamma))
	\asymp 
	n^{-\frac{r}{2}} (\log n)^{\frac{(d-1)r}{2}}.
\end{align}
\end{theorem}
The proof of  this theorem which is given in \cite[Theorem 3.5]{DK2023},  is similar the proof of the right asymptotic order of the Kolmogorov $n$-widths $d_n(\bW^r_2(\IId),L_2(\IId)$ in the unweighted case. 

From Theorem \ref{theorem:widths:p=q=2} and Corollary \ref{corollary: rho_n <  d_n} we prove

\begin{theorem} 	\label{theorem:rho_n:p=q=2}
Let $r >1$. Then  there holds the right asymptotic order
\begin{align}\label{sampling-widths:p=q=2}
	\varrho_n(\boldsymbol{\Hh}_g^r, L_2(\RRd;\gamma))
	\asymp 
	n^{-\frac{r}{2}} (\log n)^{\frac{(d-1)r}{2}}.
\end{align}		
\end{theorem}
\begin{proof}
The lower bound of \eqref{sampling-widths:p=q=2} follows from \eqref{widths:p=q=2} and the inequality 
$$
\varrho_n(\boldsymbol{\Hh}_g^r, L_2(\RRd;\gamma)) 
\ge
\lambda_n(\boldsymbol{\Hh}_g^r, L_2(\RRd;\gamma)).
$$
We verify the upper one. By \eqref{widths:p=q=2}, 
\begin{align} \label{d_n<2}
	d_n(\boldsymbol{\Hh}_g^r, L_2(\RRd;\gamma))
	\ll
	n^{-\frac{r}{2}} (\log n)^{\frac{(d-1)r}{2}}.
\end{align} 
Notice that for $r >1$, $\Hh^r_g$ is a separable reproducing kernel Hilbert space with the reproducing kernel
\begin{equation}\label{RK}
	K(\bx,\by)= \sum_{\bk \in \ZZdp} \rho_{r,\bk}^{-1} \, p_\bk(\bx)p_\bk(\by).
\end{equation}
From the orthonormality of the system $\brab{p_\bk}_{\bk \in \ZZdp}$ it is easily seen that $	K(\bx,\by)$ satisfies the finite trace assumption
\begin{equation}\label{trace assumption}
	\int_{\RRd} K(\bx,\bx) g(\bx) \rd \bx  \ < \ \infty.
\end{equation}
Hence  by Corollary \ref{corollary: rho_n <  d_n} and \eqref{d_n<2} we obtain 
$$
\varrho_n(\boldsymbol{\Hh}_g^r, L_2(\RRd;\gamma)) 
\ll
d_n(\boldsymbol{\Hh}_g^r, L_2(\RRd;\gamma)).
$$
This and \eqref{d_n<2} prove the upper bound of \eqref{sampling-widths:p=q=2}.
\hfill	
\end{proof}

For the Sobolev function  class $\bW^r_2(\RRd;\mu)$, Theorems \ref{theorem:widths:p=q=2} and \ref{theorem:rho_n:p=q=2} are the following results on sampling, linear and Kolmogorov $n$-widths.

\begin{theorem} 	\label{theorem:lambda_n:p=q=2,d>1}
Let $r\in \NN$. Then   there hold the right asymptotic orders
\begin{align}\label{Cor-widths:p=q=2}
	\lambda_n (\bW^r_2(\RRd;\gamma), L_2(\RRd;\gamma))
	=
	 d_n(\bW^r_2(\RRd;\gamma), L_2(\RRd;\gamma))
	\asymp 
	n^{-\frac{r}{2}} (\log n)^{\frac{(d-1)r}{2}}.
\end{align}		
\end{theorem}

\begin{theorem} 	\label{theorem:rho_n:p=q=2,d>1}
Let $r\in \NN$ and $r \ge 2$. Then  there holds the right asymptotic order
\begin{align}\label{W-sampling-widths:p=q=2,W}
	\varrho_n(\bW^r_2(\RRd;\gamma), L_2(\RRd;\gamma))
	\asymp 
	n^{-\frac{r}{2}} (\log n)^{\frac{(d-1)r}{2}},
\end{align}		
\end{theorem}

We stress that the assumptions $r > 1$ in Theorem \ref{theorem:rho_n:p=q=2} for \eqref{sampling-widths:p=q=2}  is vital since  it is a necessary and sufficient condition for $\Hh^r_g$ to be  a separable reproducing kernel Hilbert space 
with the finite trace condition \eqref{trace assumption} and therefore, Corollary \ref{corollary: rho_n <  d_n} can be applied. We conjecture that the consequent right asymptotic order \eqref{W-sampling-widths:p=q=2,W}  still holds true for $r = 1$. Here it may require a different technique.

We now extend the results  \eqref{rho_n:q=2} and \eqref{lambda_n:p>q}  to the weighted function class $\BWpgamma$ and weighted space $L_q(\RRd;\mu)$ for the measure $\mu$ associated with Freud-type weights as the density function given by \eqref{w(bx)}. 

For given $r$ and $p,q$, we  make use of the abbreviations:
$$
\lambda_n := \lambda_n(\BWpgamma,\Lqgamma), \quad d_n := d_n(\BWpgamma,\Lqgamma) , 
$$
$$
\varrho_n := \varrho_n(\BWpgamma,\Lqgamma).
$$
We prove the right asymptotic orders
\begin{align}\label{rho_n:q le 2-w}
\varrho_n
\asymp 
n^{-r} (\log n)^{(d-1)r}
\end{align}
for $1 \le q \le 2 < p<\infty$,  and 
\begin{equation}\label{lambda_n:p>q-w}
\lambda_n
\asymp		
d_n
\asymp 
n^{-r} (\log n)^{(d-1)r}.
\end{equation}	
for  $1\le q < p<\infty$.		
Moreover, we explicitly construct asymptotically optimal  approximation methods for  $\lambda_n$ and $d_n$ in the case $1\le q < p <\infty$.

Let us describe our strategy to prove the upper and lower bounds of \eqref{lambda_n:p>q-w} for linear and Kolmogorov $n$-widths, and construct asymptotically optimal linear approximation methods when $1 \le q < p < \infty$.
For the sampling $n$-widths, the result \eqref{rho_n:q le 2-w} can be deduced from \eqref{lambda_n:p>q-w} and the results from \cite{DKU2023} presented in Subsection \ref{Sampling recovery in RK Hilbert spaces}. 

Let $r\in \NN$, $1 \le q < p < \infty$ and $\alpha >0$, $\beta \ge 0$. Denote by $\tilde{L}_q(\IId)$ and $\tilde{W}^r_p(\IId)$ the subspaces of  $L_q(\IId) $ and $\Wpmix$, respectively,  of all functions $f$ which can be extended to the whole $\RRd$ as $1$-periodic  functions in each variable (denoted again by $f$).
Let $A_m$  be a linear operator  in $\tilde{L}_q(\IId)$ of rank $\leq m$. Assume it holds that
\begin{equation}\label{A_m-Error-a,b-theta}
	\| f - A_m(f) \|_{\tilde{L}_q(\IId)}\leq C m^{-\alpha} (\log m)^\beta \|f\|_{\tilde{W}^r_p(\IId)}, 
	\ \  f\in \tilde{W}^r_p(\IId).
\end{equation}
Then based on  $A_m$, we will construct  a linear operator  
$A_m^\mu$ in $L_q(\RRd;\mu)$ which approximates  $f \in \Wpgamma$  with the same bound as in \eqref{A_m-Error-a,b-theta}. Therefore, with $\alpha = r$ and $\beta = (d-1)r$ we prove the upper bound of \eqref{lambda_n:p>q}.

Fix a number $\theta > 1$ and put $\IId_\theta := \big[-\frac{\theta}{2}, \frac{\theta}{2}\big]^d$.
Denote by $\tilde{L}_q(\IId_\theta)$ and $\tilde{W}^r_p(\IId_\theta)$ the subspaces of  $L_q(\IId_\theta) $ and $W^r_p(\IId_\theta)$, respectively,  of all functions $f$ which can be extended to the whole $\RRd$ as $\theta$-periodic  functions in each variable (denoted again by $f$). 
A linear operator $A_m$  induces the linear operator $A_{\theta,m}$ in $\tilde{L}_q(\IId_\theta)$, defined 
for $f \in \tilde{L}_q(\IId_\theta)$ by 
$$
A_{\theta,m}(f):= A_m(f(\cdot/\theta)).
$$
From  \eqref{A_m-Error-a,b-theta} it follows that
\begin{equation*}\label{A_m-Error-a,b-theta2}
	\| f - A_{\theta,m}(f) \|_{\tilde{L}_q(\IId_\theta)}\leq C m^{-\alpha} (\log m)^\beta \|f\|_{\tilde{W}^r_p(\IId_\theta)}, 
	\ \  f\in \tilde{W}^r_p(\IId_\theta).
\end{equation*}

Since $q<p$,  we can choose a fixed $\delta>0$ such that
\begin{equation}  \label{<e^{-delta k}}
	{e^{\frac{a|\bk + (\theta \sign \bk)/2 |^\lambda}{p}-
			\frac{a|\bk + (\theta\sign \bk)/2 |^\lambda}{q}}}
	\leq 
	C e^{- \delta |\bk|}, \ \  \bk\in \ZZd.
\end{equation}
(Here, $a$ and $\lambda$ are the parameters in the definition of the weigh $w$ in \eqref{w(bx)}--\eqref{w(x)}.)

We define	for $n\in \RR_1$,
\begin{align} \label{xi-int}	
\xi_n :=  \brac{\delta^{-1} \lambda \alpha \log n}^{1/\lambda},
\end{align}
and for $\bk \in \ZZd$,
\begin{align} \label{n_bk}
n_{\bk}=
\begin{cases}
\varrho n  e^{-\frac{a\delta}{\alpha}|\bk|^\lambda} &\text{if} \ |\bk|< \xi_n,
\\
0&\text{if}\  |\bk|\geq \xi_n,
\end{cases}
\end{align}
where $
	\varrho^{-1} := \frac{2(2\pi)^{(d-1)/2}}{d!!}\sum_{s=0}^{\infty} s^d e^{-\frac{a \delta}{\alpha}s^\lambda} < \infty.$
We write
$\IId_{\theta,\bk}:=\bk+\IId_\theta$  for $\bk \in \ZZd$, and denote by
$ f_{\theta,\bk} $ the restriction of $f$ on $\IId_{\theta,\bk}$  for a function $f$ on $\RRd$.  

It is  well-known 
that one can constructively define a unit partition $\brab{\varphi_\bk}_{\bk \in \ZZd}$ such that
\begin{itemize}
	\item[\rm{(i)}] $\varphi_\bk \in C^\infty_0(\RRd)$ and 
	$0 \le \varphi_\bk (\bx)\le 1$, \ \ $\bx \in \RRd$, \ \ $\bk \in \ZZd$;
	\item[\rm{(ii)}]  $\supp \varphi_\bk$ are contained in the interior of  $\IId_{\theta,\bk}$, $\bk \in \ZZd$;
	\item[\rm{(iii)}]  $\sum_{\bk \in \ZZd}\varphi_\bk (\bx)= 1$, \ \ $\bx \in \RRd$;
	\item[\rm{(iv)}]  $\norm{\varphi_\bk }{W^r_p(\IId_{\theta,\bk})} \le C_{r,d,\theta}$, \ \ 
	$\bk \in \ZZd$,
\end{itemize}
(see, e.g., \cite[Chapter VI, 1.3]{Stein1970}). 

By using the items (ii)  and (iv)  we have that if $f \in \Wpgamma$, then
\begin{equation*}\label{f_theta,bk}
	f_{\theta,\bk}(\cdot+\bk)\varphi_\bk (\cdot+\bk)\in \tilde{W}^r_p(\IId_\theta),
\end{equation*}
and it holds that
\begin{equation}  \label{eq:norm-fwid2}
	\|f_{\theta,\bk}(\cdot+\bk)\varphi_\bk(\cdot+\bk)\|_{\tilde{W}^r_p(\IId_\theta)} 
	\ll 
	{e^{\frac{a|\bk + (\theta \sign \bk)/2 |^\lambda}{p}}}\|f\|_{\Wpgamma}.
\end{equation}
We define the linear operator $	A_{\theta,n}^w$ in $L_q(\RRd;\mu)$ by
\begin{equation}  \label{A_n^gamma}
	\brac{A_{\theta,n}^w f}(\bx): = 
	\sum_{|\bk|< \xi_n} 
	\brac{A_{\theta, n_\bk}\tilde{f}_{\theta,\bk}}(\bx-\bk),
\end{equation} 
where $\tilde{f}_{\theta,\bk}(\bx)=f_{\theta,\bk}(\bx+\bk)\varphi_\bk (\bx+\bk)$. The rank of this operator is not greater than $n$.
Indeed, by the definition of  $A_{\theta,n}^w$ and \eqref{n_bk}:
\begin{align} \label{sum n_k}
{\rm rank}\, A_{\theta,n}^w
\le 
\sum_{|\bk|< \xi_n}  {\rm rank}\, A_{\theta, n_\bk} 
\le	\sum_{|\bk|< \xi_n}n_\bk
\le n.
\end{align}

\begin{theorem} \label{thm:approx-general-theta}
	Let $r\in \NN$, $1\le q < p <\infty$ and $\alpha>0$, $\beta \ge 0$, $1 < \theta < 2$.  
	Assume that for any $m \in \RR_1$, there is  a linear operator
	$A_m$  in $\tilde L_q(\IId)$ of rank $\leq m$ 
	such that the convergence rate \eqref{A_m-Error-a,b-theta} holds. Then for any $n \in \RR_1$, based on this linear operator one can construct the linear operator
	$A_{\theta,n}^\mu$ in $L_q(\RRd;\mu)$ of rank $\leq n$ as in \eqref{A_n^gamma}  so that
	\begin{equation}\label{A_m-Error-a,b}
		\| f - A_{\theta,n}^\mu(f) \|_{\Lqgamma}\leq C n^{-\alpha} (\log n)^\beta \|f\|_{\Wpgamma}, 
		\ \  f\in \Wpgamma.
	\end{equation}
\end{theorem}

\begin{proof} 
	We preliminarily  decompose a function in $W^r_p(\RRd;\mu)$ into a sum of functions on  $\RR^d$ having support contained in integer translations of the $d$-cube $\IId_\theta := \big[-\frac{\theta}{2}, \frac{\theta}{2}\big]$. Then a desired linear operator for $W^r_p(\RRd;\mu)$ will be the sum of integer-translated dilations of $A_m$. Details of this construction are presented below.  
	
	First observe that 
	\begin{align*} 
		\RRd = \bigcup_{\bk \in \ZZd}	\IId_{\theta,\bk},
	\end{align*}	
	where $\IId_{\theta,\bk}:= \IId_\theta + \bk$. 
	From the items (ii) and (iii)  it is implied that
	\begin{align}  \label{theta-decomposition}
		f
		=  \sum_{\bk \in \ZZd}f_{\theta,\bk}\varphi_\bk,
	\end{align}
	where  $f_{\theta,\bk}$ denotes the restriction to $\IId_{\theta,\bk}$.
	Hence we have
	\begin{equation} \label{f-A_n'^gf}
	\begin{aligned} 
		\|f- A_{\theta,n}^w(f)\|^q_{L_q(\IId_{\theta,\bk};\mu)}
		&=
		\sum_{|\bk|< \xi_n} 
		\norm{f_{\theta,\bk}\varphi_\bk - \brac{A_{\theta, n_\bk}\tilde{f}_{\theta,\bk}}(\cdot -\bk)}
		{L_q(\IId_{\theta,\bk};\mu)}^q 
		\\
		&
		+ \sum_{|\bk|\geq \xi_n} 	\norm{f_{\theta,\bk}\varphi_\bk}{L_q(\IId_{\theta,\bk};\mu)}^q. 
	\end{aligned}
	\end{equation}
	By	\eqref{n_bk}, \eqref{A_m-Error-a,b-theta} and \eqref{eq:norm-fwid2} we derive the estimates
	\begin{align*}
		&	\norm{f_{\theta,\bk}\varphi_\bk - \brac{A_{\theta, n_\bk}\tilde{f}_{\theta,\bk}}(\cdot -\bk)}
		{L_q(\IId_{\theta,\bk};\mu)}
		\\
		& \ll e^{- \frac{a|\bk - (\theta\sign \bk)/2 |^\lambda}{q}}
		\norm{f_{\theta,\bk}\varphi_\bk - A_{\theta, n_\bk}\tilde{f}_{\theta,\bk}}
		{\tilde{L}_q(\IId_{\theta})} 
		\\
		& \ll  e^{- \frac{a|\bk - (\theta\sign \bk)/2 |^\lambda}{q}}  n_{\bk}^{-\alpha} (\log n_{\bk})^\beta  \|f(\cdot+\bk)\|_{\tilde{W}^r_p(\IId_\theta)}.
		\\
		&
		\ll   e^{\frac{a|\bk + (\theta \sign \bk)/2 |^\lambda}{p}-
			\frac{a|\bk - (\theta\sign \bk)/2 |^\lambda}{q}}
		\Big( n e^{-\frac{a\delta}{\alpha}|\bk|^\lambda} \Big)^{-\alpha} (\log n)^\beta  \|f\|_{\Wpgamma}.
	\end{align*}
	Using  \eqref{<e^{-delta k}} we get
	\begin{align*}
		\norm{f_{\theta,\bk}\varphi_\bk - \brac{A_{\theta, n_\bk}\tilde{f}_{\theta,\bk}}(\cdot -\bk)}
		{L_q(\IId_{\theta,\bk};\mu)} 
		&
		\ll  e^{- a \delta |\bk|^\lambda}  n^{-\alpha}  (\log n)^\beta   \|f\|_{\Wpgamma},
	\end{align*}
	which implies
	\begin{align*}
		\sum_{|\bk|< \xi_n} \norm{f_{\theta,\bk}\varphi_\bk - \brac{A_{\theta, n_\bk}\tilde{f}_{\theta,\bk}}(\cdot -\bk)}
		{L_q(\IId_{\theta,\bk};\mu)}^q 
		&\ll \sum_{|\bk|< \xi_n}   e^{- q a \delta |\bk|^\lambda} 
		\big(n^{-\alpha}  (\log n)^\beta  \big)^q  \|f\|_{\Wpgamma}^q
		\\&
		\ll  \big(n^{-\alpha}  (\log n)^\beta  \big)^q  \|f\|_{\Wpgamma}^q.
	\end{align*}
	We have for a fixed $\varepsilon \in (0,1/2)$,
	\begin{align*}
		\sum_{|\bk|\geq \xi_n} 	\norm{f_{\theta,\bk}\varphi_\bk}{L_q(\IId_{\theta,\bk};\mu)}^q
		& \ll  \sum_{|\bk|\geq \xi_n} 
		e^{- q\big(\frac{a|\bk - (\theta \sign \bk)/2 |^\lambda}{q}-
			\frac{a|\bk + (\theta\sign \bk)/2 |^\lambda}{p}\big)}
		\|f\|_{\Wpgamma}^q
		\\
		&	\ll  \sum_{|\bk|\geq \xi_n}  e^{- q\delta |\bk|^\lambda}  \|f\|_{\Wpgamma}^q
		\ll e^{-q \delta (1-\varepsilon) \xi_n^2}  \|f\|_{\Wpgamma}^q 
		\\
		&	\ll e^{-2q a (1-\varepsilon) \log n}  \|f\|_{\Wpgamma}^q 
		\ll  \big(n^{-\alpha}  (\log n)^\beta  \big)^q \|f\|_{\Wpgamma}^q.
	\end{align*}
	From the last two estimates and \eqref{f-A_n'^gf} we obtain \eqref{A_m-Error-a,b}.
	\hfill
\end{proof}

The lower bound of \eqref{lambda_n:p>q} relies on the following known results on Kolmogorov $n$-widths in the space $\tilde{L}_q(\IId)$  the unweighted Sobolev class $\tilde{\bW}^r_p(\IId)$.

\begin{lemma} \label{lemma: widths}
	Let $r\in \NN$ and $1\le q<p<\infty$.
	Then we have the right asymptotic order
	\begin{equation*}
		d_m(\tilde{\bW}^r_p(\IId),\tilde{L}_q(\IId)) \asymp m^{-r} (\log m)^{(d-1)r}.
	\end{equation*}
	Moreover,   truncations  on certain hyperbolic crosses of the Fourier series form an asymptotically optimal  linear operator $A_n$ in $\tilde{L}_q(\IId)$  of rank  $\le n$ such that 
	\begin{equation}\label{f-A_m f}
		\|f-A_m f\|_{\tilde{L}_q(\IId)} \ll m^{-r} (\log m)^{(d-1)r} \|f\|_{\tilde{W}^r_p(\IId)}, 
		\ \  f\in \tilde{W}^r_p(\IId).
	\end{equation}		
\end{lemma}
For details on this lemma see, e.g., in \cite[Theorems 4.2.5, 4.3.1 \& 4.3.7]{DTU18B} and related comments on the asymptotic optimality of the hyperbolic cross approximation.

We are now in the position to prove the main result in this section.

\begin{theorem} \label{theorem:widths: q<p}	
	Let $r\in \NN$ and $1\le q<p<\infty$. Then for any $n \in \RR_1$, based on the linear operator $A_m$ in Lemma \ref{lemma: widths} one can construct the linear operator
	$A_n^\mu$ in  $L_q(\RRd;\mu)$ of rank $\leq n$ as in \eqref{A_n^gamma}  so that there hold the right asymptotic orders
	\begin{equation}\label{eq:dnWg}
		\sup_{f\in \BWpgamma}	\| f - A_n^\mu(f) \|_{\Lqgamma}
		\asymp
		\lambda_n\asymp	d_n \asymp
		n^{-r} (\log n)^{(d-1)r}.
	\end{equation}	
\end{theorem}
\begin{proof}
	For a fixed $1 < \theta < 2$, we define $A_n^\mu:=  A_{\theta,n}^\mu$  as the linear operator described in  Theorem~\ref{thm:approx-general-theta}.
	The upper bounds  in \eqref{eq:dnWg} follow from \eqref{f-A_m f} and Theorem \ref{thm:approx-general-theta} with $\alpha = r$, $\beta =(d-1)r$.
	
	If $f$ is a $1$-periodic function on $\RRd$ and $f\in  \tilde{W}^r_p(\IId)$, then
	\begin{align*}
		\|f\|_{\Wpgamma}&= \Bigg(e^{db}\sum_{|\br|_\infty \leq r} \int_{\RRd} |D^\br f|^p e^{-a|\bx|^\lambda}\rd \bx\Bigg)^{1/p}
		\\
		& \ll \Bigg(\sum_{|\br|_\infty \leq r} \sum_{\bk\in \ZZd} \int_{\IId} |D^\br f(\bx+\bk)|^p e^{-a|\bx+\bk|^\lambda}\rd \bx\Bigg)^{1/p}
		\\
		& \ll\Bigg(\sum_{|\br|_\infty \leq r}  \int_{\IId} |D^\br f(\bx)|^p \rd \bx \sum_{\bk\in \ZZd}e^{-a|\bk - (\sign \bk)/2 |^\lambda}\Bigg)^{1/p}
		\\
		&
		\ll 
		\|f\|_{\tilde{W}^r_p(\IId)},
	\end{align*}
	and
	$$
	\|f\|_{\tilde{L}_q(\IId)} \leq e^{\frac{db}{q}}e^{\frac{d}{8q}}\|f\|_{\Lqgamma}.
	$$
	Hence we get
	\begin{align*}  
		\lambda_n \ge	d_n \gg	d_n(\tilde{\bW}^r_p(\IId),\tilde{L}_q(\IId)).
	\end{align*}	
	Now Lemma \ref{lemma: widths} implies the lower bounds in \eqref{eq:dnWg}. 
	\hfill 
\end{proof}

\begin{theorem} \label{theorem:rho_n: q le 2 <p}	
	Let $r\in \NN$ and $1\le q \le 2< p<\infty$. Then  there holds the right asymptotic order
	\begin{align}\label{sampling-widths:p>2}
	\varrho_n
	\asymp 
	n^{-r} (\log n)^{(d-1)r}.
	\end{align}		
\end{theorem}
\begin{proof}
 The lower bound of \eqref{sampling-widths:p>2} follows from \eqref{eq-relations} and \eqref{eq:dnWg}. By the norm inequality $\|\cdot\|_{L_q(\RRd;\mu)} \ll \|\cdot\|_{L_2(\RRd;\mu)}$ for $1\le q \le 2$,  it is sufficient to prove the upper bound of \eqref{sampling-widths:p>2} for $q=2$.  By \eqref{eq:dnWg} we have that 
\begin{align}\label{d_n<}
d_n
\ll 
n^{-r} (\log n)^{(d-1)r}.
\end{align}		
Notice that 
the separable normed space $\Wpgamma$ is continuously embedded into  $L_2(\RRd;\mu)$, and the evaluation functional $f \mapsto f(\bx)$ is continuous on the space $\Wpgamma$ for each $\bx \in \RRd$. This means that $\BWpgamma$ satisfies Assumption A in Subsection \ref{Sampling recovery in RK Hilbert spaces}. By Corollary \ref{corollary: rho_n <  d_n,p<2} and \eqref{d_n<} we prove the upper bound for $q=2$: 
$$
\varrho_n \ll d_n
\ll
n^{-r} (\log n)^{(d-1)r}.
$$
\hfill 
\end{proof}
	

	\section{Numerical weighted integration   in the space $\Wpgamma$}
\label{Numerical weighted integration p}

\subsection{Introducing remarks}
\label{Introducing remarks sec 4}
The aim of this section  is to present and extend some recent results of \cite{DK2023} on numerical weighted integration over $\mathbb{R}^d$ for functions from weighted  Sobolev spaces $W^r_p(\mathbb{R}^d;\mu)$ of  mixed smoothness $r \in \mathbb{N}$ for $1<p<\infty$ and the measure $\mu$ with the density function given by \eqref{w(bx)}. The next section is devoted to this problem in the case $p=1$ which requires a different approach.
Extending the results from \cite{DK2023} for the Gaussian-weighted case, we prove the right asymptotic order of the quantity of  optimal quadrature  for the Sobolev function class  $\bW^r_p(\mathbb{R}^d;\mu)$ and construct asymptotically optimal quadratures. 

There is a large number of works on high-dimensional unweighted integration over the unit $d$-cube $\IId:=[0,1]^d$ for functions having a mixed smoothness (see  \cite{DTU18B, DKS13, Tem18B} for results and bibliography). However, there are  only a few works  on high-dimensional weighted integration for functions having a mixed smoothness. The problem of  optimal  weighted integration \eqref{If}--\eqref{Int_n} has been studied in \cite{IKLP2015, IL2015,DILP18}  for functions in certain Hermite spaces, in particular, the space $\Hh_{d,r}$ which coincides with $W^r_2(\mathbb{R}^d;\gamma)$ in terms of norm equivalence.
It has been proven in \cite{DILP18} that
\begin{equation*}\label{DILP18}
n^{-r}  (\log n)^{(d-1)/2}
\ll	
\Int_n\big( \bW^r_2(\RRd; \gamma))  
\ll 
n^{-r} (\log n)^{d(2r + 3)/4 - 1/2}.
\end{equation*}
 Moreover,  the upper bound is achieved by  a translated and scaled quasi-Monte Carlo (QMC) quadrature based on Dick's high order digital nets. 

Recently, in \cite[Theorem 2.3]{DK2023} for the space $W^r_p(\mathbb{R}^d, \gamma)$ with  $r\in \NN$ and $1<p<\infty$, we have constructed  an  asymptotically optimal quadrature  $Q_n^\gamma$ of the form \eqref{Q_nf-introduction} which  gives the  asymptotic order  
\begin{equation} 	\label{AsympQuadrature}
\sup_{f\in \bW^r_p(\RRd; \gamma)} \bigg|\int_{\RRd}f(\bx) \gamma(\rd\bx) - Q_n^\gamma f\bigg| 
\asymp
\Int_n\big(\bW^r_p(\RRd; \gamma) \big) 
\asymp
n^{-r} (\log n)^{(d-1)/2}.
\end{equation}

In constructing $Q_n^\gamma$, we  proposed a novel method assembling an asymptotically optimal quadrature for the related Sobolev spaces of the same mixed smoothness $r$ on the unit $d$-cube to the  integer-shifted $d$-cubes which cover $\RRd$. The asymptotically optimal quadrature  $Q_n^\gamma$ 
is based on very sparse  integration nodes contained in a $d$-ball of radius
$\sqrt{\log n}$.

 In this section,  we  show that  the result \eqref{AsympQuadrature} can be extended by a modification of this technique to the  weighted Sobolev spaces $\Wpgamma$ for $1 < p < \infty$.

\subsection{Assembling quadratures}
\label{Assembling quadratures}
In this subsection, based on a quadrature on  the  $d$-cube $\IId:=\big[-\frac{1}{2}, \frac{1}{2}\big]^d$  for numerical integration of functions from classical Sobolev spaces of mixed smoothness $r$ on $\IId$, by assembling we construct  a quadrature on $\RRd$ for numerical integration of functions from weighted Sobolev spaces $\Wpgamma$ which preserves   the  convergence rate.  In the next subsection, as a consequence, we prove the right asymptotic order of 
$\Int_n\big(\BWpgamma\big)$.

Let $r\in \NN$, $1<p<\infty$ and $\alpha>0$, $\beta \ge 0$.  	
Assume that for the quadrature
\begin{equation}\label{Q_m(f)}
	Q_m(f): = \sum_{i=1}^m \lambda_i f(\bx_i), \ \ \{\bx_1,\ldots,\bx_m\}\subset \IId,
\end{equation}
holds the convergence rate
\begin{equation}\label{IntError-a,b}
	\bigg|\int_{\IId} f(\bx) \rd \bx  - Q_m(f)\bigg| \leq C m^{-\alpha} (\log m)^\beta \|f\|_{\Wpmix}, 
	\ \  f\in \Wpmix.
\end{equation}
Then by using  $Q_m$, we will construct  a quadrature  on $\RRd$  which approximates the integral \eqref{If} with the same bound of the error approximation as in \eqref{IntError-a,b}  for 
$f \in \Wpgamma$. 

Our strategy is as follows.  
The weighted integral 
\begin{equation} \label{integral}
\int_{\RRd}f(\bx) w(\bx) \rd\bx
\end{equation}
 can be represented as the sum of component integrals over the  integer-shifted $d$-cubes $\IId_{\bk}$ as 
\begin{align} \label{Int_RRd}
	\int_{\RRd}f(\bx) w(\bx) \rd\bx 
	=  \sum_{\bk \in \ZZd}\int_{\IId_\bk}f_\bk(\bx)w_\bk(\bx)\rd \bx,
\end{align}
where  for $\bk \in \ZZd$, $\IId_\bk:=\bk+\IId$ 
and for a function $h$  on $\RRd$,  
$h_\bk $ denotes the restriction of $h$ to $\IId_\bk$.
For a given $n \in \RR_1$, we take ``shifted" quadratures $Q_{n_\bk}$  of the form \eqref{Q_m(f)} for  approximating the component integrals in the sum in \eqref{Int_RRd}.  
	The integration nodes  in  
	$Q_{n_\bk}$, $\bk \in \ZZd$, are taken so that
	$$
	\sum_{\bk \in \ZZd} \lfloor n_\bk \rfloor \le n.
	$$
In the next step, we  ``assemble" these shifted integration nodes to form a quadrature  $Q_n(f)$ for approximating the integral \eqref{integral}. Let us describe this construction in detail.

It is clear that if $f\in \Wpgamma$, then 
$
f_\bk(\cdot+\bk) \in \Wpmix,
$
and
\begin{equation}\label{eq:norm-fwid}
	\|f_\bk(\cdot+\bk)\|_{\Wpmix} \leq e^{\frac{a|\bk + (\sign \bk)/2 |^\lambda}{p}}\|f\|_{\Wpgamma},
\end{equation}
where $\sign \bk:= \brac{\sign k_1, \ldots , \sign k_d}$ and $\sign  x := 1$ if $x \ge 0$, and $\sign  x := -1$ otherwise for $x \in \RR$.
We have
$$\|w_\bk(\cdot+\bk)\|_{\Wpmix}=\Bigg(\sum_{|\br|_\infty \leq r} \|D^\br w\|_{L_p(\IId_\bk)}^p\Bigg)^{1/p}.$$
A direct computation shows that for any $\bs \in \NN_0^d$ with $|\bs|_\infty \le r$, we have 
$$D^\bs w(\bx)=F_\bs(\bx)w(\bx),$$ 
where 
$$
F_\bs(\bx):= \prod_{i=1}^d F_{s_i} (x_i),
$$ 
and for $s \in \NN_0$, the univariate function $F_s$ is defined by
	\begin{equation}\nonumber
	F_s(x):=  = (\sign (x))^s\sum_{j=1}^s c_{s,j}(\lambda,a) |x|^{\lambda_{s,j}},
\end{equation}
where  $\sign (x):= 1$ if $x \ge 0$, and $\sign (x):= -1$ if  $x < 0$, 
\begin{equation}\label{lambda_{s,s}}
	\lambda_{s,s} = s(\lambda - 1) > \lambda_{s,s-1} > \cdots > \lambda_{s,1} = \lambda - s,
\end{equation}
and  $c_{s,j}(\lambda,a)$ are polynomials in the variables $\lambda$ and $a$ of degree at most $s$ with respect to each variable.
 Therefore, for $\bx \in \IId_\bk$ we get
$$
|D^\bs w(\bx)| \leq  Ce^{-\frac{a|\bk - (\sign \bk)/2 |^\lambda}{\tau'}}\leq Ce^{-\frac{a|\bk|^\lambda}{\tau}}
$$
for some $\tau'$ and $\tau$ such that $1<\tau<\tau'<p<\infty $.  This implies
\begin{align} \label{g_bk}
	\|w_\bk(\cdot+\bk)\|_{\Wpmix}  \leq  Ce^{-\frac{a|\bk|^\lambda}{\tau}}
\end{align}
with $C$ independent of $\bk \in \ZZd$. Since $\Wpmix$ is a multiplication algebra (see \cite[Theorem~3.16]{NgS17}), from \eqref{eq:norm-fwid} and \eqref{g_bk}  we have that 
\begin{align} \label{multipl-algebra1}
	f_{\bk}(\cdot+\bk)w_\bk(\cdot+\bk)\in \Wpmix,
\end{align}
and 
\begin{equation}  
	\begin{aligned} \label{multipl-algebra2}
		\|f_{\bk}(\cdot+\bk)w_\bk(\cdot+\bk)\|_{\Wpmix} 
		& 
		\leq C \|f_{\bk}(\cdot+\bk)\|_{\Wpmix}  \cdot \|w_\bk(\cdot+\bk)\|_{\Wpmix}
		\\
		& \leq C e^{\frac{a|\bk + (\sign \bk)/2 |^\lambda}{p}-\frac{a|\bk|^\lambda}{\tau}}\|f\|_{\Wpgamma}.
	\end{aligned}
\end{equation}
For $1<\tau<p<\infty $,  we choose $\delta>0$ such that
\begin{equation}  \label{[tau]<e^{-delta k}1}
	\max \bigg\{e^{-a|\bk - (\sign \bk)/2 |^\lambda\big(1- \frac{1}{p}\big)},
	e^{\frac{a|\bk + (\sign \bk)/2 |^\lambda}{p}-\frac{a|\bk|^\lambda}{\tau}}\bigg\}
	\leq 
	C e^{-\delta |\bk|^\lambda}
\end{equation}
for $\bk\in \ZZd$, and therefore,
\begin{align}  \label{f_{bk}<}
	\|f_{\bk}(\cdot+\bk)w_\bk(\cdot+\bk)\|_{\Wpmix} 
	& \leq C e^{-\delta |\bk|^\lambda}\|f\|_{\Wpgamma}, \quad  \bk\in \ZZd.
\end{align}
For $n\in \RR_1$, let $\xi_n$ and $n_\bk$ be given as in \eqref{xi-int} and \eqref{n_bk}. We have by \eqref{sum n_k}
\begin{align} 	\label{<n2}
	\sum_{|\bk|< \xi_n}  n_\bk  \le n.
\end{align} 
We define
\begin{equation} \label{I_n^gamma}
	Q_n(f):=\sum_{|\bk|< \xi_n}Q_{n_\bk}(f_{\bk}(\cdot+\bk)w_\bk(\cdot+\bk))
	= \sum_{|\bk|< \xi_n}\sum_{j=1}^{\lfloor n_\bk \rfloor} \lambda_j  f_{\bk}(\bx_j+\bk)w_\bk(\bx_j+\bk),
\end{equation}
or equivalently,
\begin{equation} \label{Q_n^gamma2}
	Q_n(f):=	\sum_{|\bk|< \xi_n}\sum_{j=1}^{\lfloor n_\bk \rfloor} \lambda_{\bk,j} f(\bx_{\bk,j})
\end{equation}
as a quadrature for the approximate weighted integration of  functions $f$ on $\RRd$,
where 
$$\bx_{\bk,j}:= \bx_j+\bk, \ \ \ \lambda_{\bk,j}:= \lambda_j w_\bk(\bx_j+\bk)
$$
 (here for simplicity, with an abuse the dependence of integration nodes and weights on the quadratures  $Q_{n_\bk}$ is omitted).  The  integration nodes of the quadrature $Q_n$  are
\begin{equation} \label{int-nodes}
	\{\bx_{\bk,j}: |\bk|< \xi_n, \, j=1,\ldots,\lfloor n_\bk \rfloor\}
	\subset \RRd,
\end{equation}
and the  integration weights 
$$
(\lambda_{\bk,j}: |\bk|< \xi_n, \, j=1,\ldots,\lfloor n_\bk \rfloor).
$$
Due to \eqref{<n2}, the number of integration nodes  is not greater than $n$.  From the definition we can see that the  integration nodes are contained in the ball of radius  $\xi_n^*:=\sqrt{d}/2 + \xi_n$, i.e., 
$$
\{\bx_{\bk,j}: |\bk|< \xi_n, \, j=1,\ldots,\lfloor n_\bk \rfloor\}
\subset B(\xi_n^*):= \brab{\bx \in \RRd:\, |\bx| \le \xi_n^*}.
$$
 The density of the integration nodes is exponentially decreasing in $|\bk|$ to zero from the origin of $\RRd$ to the boundary of the ball $B(\xi_n^*)$, and  the set of integration nodes is very sparse because of the choice of $n_{\bk}$ as in \eqref{n_bk}.

\begin{theorem} \label{thm:int-general}
	Let $r\in \NN$, $1<p<\infty$ and $\alpha >0$, $\beta \ge 0$.  
	Assume that for any $m \in \RR_1$, there is  a   quadrature  $Q_m$  of the form \eqref{Q_m(f)} satisfying \eqref{IntError-a,b}. Then for  the   quadrature  $Q_{n}$ defined as in \eqref{Q_n^gamma2}, 
	we have
	\begin{equation} 	\label{IntError0}
		\bigg|\int_{\RRd}f(\bx) \mu(\rd\bx) - Q_n(f)\bigg| 
		\ll n^{- \alpha}  (\log n)^\beta \|f\|_{\Wpgamma}, \ \  f \in \Wpgamma.
	\end{equation}
\end{theorem}
\begin{proof} 
	Let $f\in \Wpgamma$ and $m\in \RR_1$. 	For the  quadrature $I_m$ for functions on $\IId$ in the assumption, by \eqref{IntError-a,b} and \eqref{f_{bk}<},  we have 
	\begin{equation} \label{IntError1}
		\begin{aligned}
			&\bigg|\int_{\IId} f_\bk(\bx+\bk)w_\bk(\bx+\bk)\rd \bx- I_{m}(f_{\bk}(\cdot+\bk)w_\bk(\cdot+\bk))\bigg| 
							\\&
						\ll m^{-r} (\log m)^{\frac{d-1}{2}} 	\|f_{\bk}(\cdot+\bk)w_\bk(\cdot+\bk)\|_{\Wpmix} 
			\\&
			\ll  m^{-\alpha} (\log m)^\beta e^{-\delta |\bk|^\lambda}\|f\|_{\Wpgamma}.
		\end{aligned}
	\end{equation}
	From \eqref{Int_RRd} and \eqref{I_n^gamma}
	it follows that
	\begin{align*}
		\bigg|\int_{\RRd}f(\bx) w(\bx) \rd\bx - Q_n(f)\bigg| 
		&\leq \sum_{|\bk|< \xi_n} \bigg|\int_{\IId_\bk} f_{\bk}(\bx)w_\bk(\bx)\rd \bx - 
		Q_{n_\bk}(f_{\bk}(\cdot+\bk)w_\bk(\cdot+\bk))\bigg| 
		\\ &
		+ \sum_{|\bk|\geq \xi_n}\bigg|\int_{\IId_\bk}f_\bk(\bx)w_\bk(\bx)\rd \bx\bigg|.
	\end{align*}
	For each term in the first sum, by   \eqref{IntError1} we derive  the estimates
	\begin{align*}
		\bigg|\int_{\IId_\bk} f_{\bk}(\bx)w_\bk(\bx)\rd \bx 
		& - Q_{n_\bk}(f_{\bk}(\cdot+\bk)w_\bk(\cdot+\bk))\bigg| 
		\\& = \bigg|\int_{\IId} f_\bk(\bx+\bk)w_\bk(\bx+\bk)\rd \bx- Q_{n_\bk}(f_{\bk}(\cdot+\bk)w_\bk(\cdot+\bk))\bigg|
		\\
		& \ll  n_{\bk}^{- \alpha} (\log n_{\bk})^b e^{-\delta |\bk|^\lambda} \|f\|_{\Wpgamma}
		\\
		& \ll ( n e^{-\frac{\delta}{2\alpha}|\bk|^\lambda} )^{- \alpha} (\log n)^b
		e^{-\delta |\bk|^\lambda}\|f\|_{\Wpgamma}
		\\
		& \ll   n^{- \alpha} (\log n)^be^{-\frac{a|\bk|^\lambda\delta}{2}}\|f\|_{\Wpgamma}.
	\end{align*}
	Hence,
	\begin{align*}
		&\sum_{|\bk|< \xi_n} \bigg|\int_{\IId_\bk} f_{\bk}(\bx)w_\bk(\bx)\rd \bx 
		- Q_{n_\bk}(f_{\bk}(\cdot+\bk)w_\bk(\cdot+\bk))\bigg|
		\\
		& 
	 \ll \sum_{|\bk|< \xi_n} n^{-\alpha}  (\log n)^be^{-\frac{a|\bk|^\lambda\delta}{2}}\|f\|_{\Wpgamma}
		\ll n^{-\alpha}  (\log n)^b\|f\|_{\Wpgamma}.
	\end{align*}
	For each term in the second sum  we get by H\"older's inequality and \eqref{[tau]<e^{-delta k}1},
	\begin{align*}
		\bigg|\int_{\IId_\bk} f_\bk(\bx) w_\bk(\bx)\rd \bx \bigg| 
		&\leq \bigg(\int_{\IId_\bk}|f_\bk(\bx)|^p w_\bk(\bx)\rd \bx \bigg)^{\frac{1}{p}} \bigg(\int_{\IId_\bk}w_\bk(\bx) \rd \bx \bigg)^{1-\frac{1}{p}} 
		\\
		& \ll e^{-a|\bk - (\sign \bk)/2 |^\lambda(1-\frac{1}{p})}\|f\|_{\Wpgamma}
		\\
		&
	\ll e^{-\delta |\bk|^\lambda}\|f\|_{\Wpgamma},
	\end{align*}
		which implies
	\begin{equation}\label{eq-epsilon01}
		\begin{aligned}
			\sum_{|\bk|\geq \xi_n}\bigg|\int_{\IId_\bk} f_\bk(\bx)w_\bk(\bx)\rd \bx\bigg|
			&
			\leq \sum_{|\bk|\geq \xi_n}e^{-\delta |\bk|^\lambda} \|f\|_{\Wpgamma}
			\\
			&  \ll \|f\|_{\Wpgamma}\sum_{s=\lceil\xi_n\rceil}^\infty    e^{-s^\lambda \delta} s^d
			\\
			&
			\le  \|f\|_{\Wpgamma}e^{-\xi_n^\lambda\delta(1-\varepsilon)}\sum_{s=\lceil\xi_n\rceil}^\infty   e^{-s^\lambda \varepsilon \delta} s^d
			\\
			&
		\ll \|f\|_{\Wpgamma}e^{-\xi_n^\lambda\delta(1-\varepsilon)}\sum_{s=0}^\infty   e^{-s \varepsilon \delta} s^d
		\\
		&  \ll  e^{-2\alpha(1-\varepsilon)\log n}\|f\|_{\Wpgamma}		
		\ll n^{-\alpha}  (\log n)^b\|f\|_{\Wpgamma},
		\end{aligned}
	\end{equation}
	with $\varepsilon \in (0,1/2)$. 
	
	Summing up, we prove \eqref{IntError0}.
	\hfill
\end{proof}

Some important quadratures such as the Frolov  quadrature and the QMC quadrature based on Fibonacci lattice rules ($d=2$) are constructively designed for functions  on $\RRd$ with support contained in  the unit  $d$-cube or for $1$-periodic functions. To employ them for  constructing  a quadrature for functions on $\RRd$ we need to modify   those constructions.

Assume that  there is a quadrature $Q_m$ of the form \eqref{Q_m(f)} with the integration nodes
$\{\bx_1,\ldots,\bx_m\}\subset \brac{- \frac{1}{2}, \frac{1}{2}}^d$ and weights $(\lambda_1,\ldots,\lambda_m)$ such that  the convergence rate
\begin{equation}\label{IntError-a,b,F}
	\bigg|\int_{\IId} f(\bx) \rd \bx  - Q_m(f)\bigg| \leq C m^{-\alpha} (\log m)^\beta \|f\|_{\Wpmix}, 
	\ \  f\in \mathring{W}^r_p(\IId)
\end{equation}
holds where $ \mathring{W}^r_p(\IId)$ denotes the space of functions in $W^r_p(\RRd)$ with support contained in $\IId$.
Then based on the quadrature $Q_m$, we propose two constructions of  quadratures  which approximate the integral \eqref{integral} with the same convergence rate for 
$f \in \Wpgamma$.

The first method is a preliminary change of variables  to transform the quadrature $Q_m$  into a quadrature for functions in $\Wpmix$  which gives the same convergence rate, and then apply the  construction as in \eqref{Q_n^gamma2}. Let us describe it.
Let $\psi_k$  be the function defined by
\begin{equation}\label{psin}
	\psi_k(t) = \left\{\begin{array}{rcl}
		C_k\int_0^t \xi^k(1-\xi)^k\,\rd\xi, & t\in [0,1];\\
		1,& t>1;\\
		0 ,& t<0\,,
	\end{array}\right.
\end{equation}
where 
$$
C_k=\big(\int_0^1 \xi^k(1-\xi)^k \,\rd\xi\big)^{-1}. 
$$
If $f \in \Wpmix$,  a change of variable yields that 
$$
\int_{\IId} f(\bx) \rd \bx =\int_{\IId} \big(T_{\psi_k}f\big)(\bx) \rd \bx,
$$
where the function $T_{\psi_k}f$ is defined by
\begin{equation} \label{Tf}
\big(T_{\psi_k}f\big) (\bx) :=\bigg(\prod_{i=1}^d\psi_k'(x_i)\bigg)f\big(\psi_k(x_1),\ldots,\psi_k(x_d)\big).
\end{equation}
 Observe that the support of $T_{\psi_k}f$ is contained in $\IId$. If  $T_{\psi_k} f$ belongs to $ \mathring{W}^\alpha_p(\IId)$,  then a quadrature with the integration nodes  $\{\tilde{\bx}_1,\ldots,\tilde{\bx}_m\}\subset \IId$ and weights $(\tilde{\lambda}_1,\ldots, \tilde{\lambda}_m)$ for the function $f$ can be defined  as
$$
\tilde{Q}_m(f):= Q_m(T_{\psi_k} f) = \sum_{j=1}^m \tilde{\lambda}_j f(\tilde{\bx}_j),
$$ 
where 
$$\tilde{\bx}_j=(\psi_k(x_{j,1}),\ldots,\psi_k(x_{j,d})) \ \ \text{and} \ \ 
\tilde{\lambda}_j=\lambda_j \psi_k'(x_{j,1})\cdot\ldots\cdot\psi_k'(x_{j,d}).
$$  Hence, our task is finding a condition on $k$ so that  the mapping
$$
	T_{\psi_k}: \Wpmix \to  \mathring{W}^r_p(\IId)
$$ 
defined by \eqref{Tf},
is a bounded operator from $\Wpmix$ to $ \mathring{W}^r_p(\IId)$, see, e.g., \cite[Theorem IV.4.1]{Tem93B}. 

The second method is based on the decomposition \eqref{theta-decomposition} of  functions in $W^r_p(\RRd;\mu)$.  Let  $\brab{\varphi_\bk}_{\bk \in \ZZd}$ be the unit partition satisfying items (i)--(iv), which is defined in Subsection \ref{The case 1 le q < p < infty}. 
By the items (ii) and (iii),
the integral \eqref{integral} can be represented as
\begin{align} \label{Int_RRd-F}
	\int_{\RRd}f(\bx) w(\bx) \rd\bx 
	=  \sum_{\bk \in \ZZd}\int_{\IId_{\theta,\bk}} f_{\theta,\bk}(\bx) w_{\theta,\bk}(\bx) \varphi_\bk (\bx) \rd \bx,
\end{align}
where  $f_{\theta,\bk}$ and $w_{\theta,\bk}$ denote the restrictions of $f$ and $w$ on $\IId_{\theta,\bk}$, respectively.
The quadrature \eqref{Q_m(f)} induces the quadrature 
\begin{equation}\label{I_theta,m(f)}
	Q_{\theta,m}(f): = \sum_{i=1}^m \lambda_{\theta,i} f(\bx_{\theta,i}), 
\end{equation}
for functions $f$ on $\IId_\theta$,  where 
$\bx_{\theta,i}:= \theta\bx_i$ and 
$\lambda_{\theta,i}:= \theta \lambda_i $.

Denote by $ \mathring{W}^r_p(\IId_\theta)$ the subspace of function in $W^r_p(\RRd)$ with support contained in $\IId_\theta$.
From  \eqref{IntError-a,b,F} the error bound
	\begin{equation*}\label{IntError-a,b,theta}
		\bigg|\int_{\IId_\theta} f(\bx) \rd \bx  - Q_{\theta,m}(f)\bigg| \ll m^{-\alpha} (\log m)^\beta \|f\|_{W^r_p(\IId_\theta)}
	\end{equation*}
	holds for every $f\in \mathring{W}^r_p(\IId_\theta)$.
Let $f\in \Wpgamma$. It is clear that 
$f_{\theta,\bk}(\cdot+\bk)\in W^r_p(\IId_\theta)$ and
\begin{equation*}\label{ineq-norms}
	\|f_{\theta,\bk}(\cdot+\bk)\|_{W^r_p(\IId_\theta)} 
	\ll e^{\frac{a|\bk + (\theta \sign \bk)/2 |^\lambda}{p}}\|f\|_{\Wpgamma}, \ \
	f\in \Wpgamma, \ \ \bk \in \ZZd.
\end{equation*}
Similarly to \eqref{multipl-algebra1} and \eqref{multipl-algebra2}, by additionally using the items (ii)  and (iv)  we have that 
\begin{equation*} 
	f_{\theta,\bk}(\cdot+\bk)w_{\theta,\bk}(\cdot+\bk)\varphi_\bk (\cdot+\bk)\in \mathring{W}^r_p(\IId_\theta),
\end{equation*}
and 
\begin{align*}
	\|f_{\theta,\bk}(\cdot+\bk)w_{\theta,\bk}(\cdot+\bk)\varphi_\bk(\cdot+\bk)\|_{W^r_p(\IId_\theta)} 
	\ll 
	e^{\frac{a|\bk + (\theta \sign \bk)/2 |^\lambda}{p}-\frac{a|\bk|^\lambda}{\tau}} \|f\|_{\Wpgamma},
\end{align*}
where $\tau$ is a fixed number satisfying the inequalities $1<\tau<p<\infty $.
We choose $\delta>0$ so that
\begin{equation*}  \label{[tau]<e^{-delta k}}
	\max \bigg\{e^{- a|\bk - (\theta \sign \bk)/2 |^\lambda\big(1-\frac{1}{p}\big)}, 
	e^{\frac{a|\bk + (\theta \sign \bk)/2 |^\lambda}{p}-\frac{a|\bk|^\lambda}{\tau}}\bigg\}
	\leq 
	C e^{-\delta |\bk|^\lambda}, \ \  \bk\in \ZZd.
\end{equation*}

For $n\in \RR_1$, let $\xi_n$ and $n_{\bk}$ be given as in \eqref{xi-int}	
and \eqref{n_bk}, respectively. 
Noting \eqref{Int_RRd-F} and \eqref{I_theta,m(f)}, we define
\begin{equation*} 
	Q_{\theta,n}(f):=\sum_{|\bk|< \xi_n}Q_{\theta,n_\bk}
	\big(f_{\theta,\bk}(\cdot+\bk)w_{\theta,\bk}(\cdot+\bk)\varphi_\bk(\cdot+\bk)\big),
\end{equation*}
or equivalently,
\begin{equation} \label{Q_n^gamma3}
	Q_{\theta,n}(f):=
	\sum_{|\bk|< \xi_n}\sum_{j=1}^{\lfloor n_\bk \rfloor} \lambda_{\theta,\bk,j} f(\bx_{\theta,\bk,j})
\end{equation}
as a linear quadrature for the approximate integration of weighted functions $f$ on $\RRd$,
where 
$$
\bx_{\theta,\bk,j}:= \bx_{\theta,i}+\bk, \ \ \
\lambda_{\theta, \bk,j}:= \lambda_{\theta,i} w_\bk(\bx_{\theta,\bk,j})\varphi_\bk(\bx_{\theta,\bk,j}).
$$   The integration nodes and the  weights of the quadrature  $Q_{\theta,n}$ are
\begin{equation} \label{int-nodes-theta}
	\{\bx_{\theta,\bk,j}: |\bk|< \xi_n, \, j=1,\ldots,\lfloor n_\bk \rfloor\}
	\subset \RRd,
\end{equation}
and
$$
(\lambda_{\theta,\bk,j}: |\bk|< \xi_n, \, j=1,\ldots,\lfloor n_\bk \rfloor).
$$
Due to \eqref{<n2}, the number of integration nodes  is not greater than $n$. Moreover, from the definition we can see that the integration nodes are contained in the ball of radius $\xi_{\theta,n}^*:= \theta\sqrt{d}/2 + \xi_n$, i.e., 
$$
\{\bx_{\theta,\bk,j}: |\bk|< \xi_{\theta,n}^*, \, j=1,\ldots,\lfloor n_\bk \rfloor\}
\subset B(\xi_{\theta,n}^*):= \brab{\bx \in \RRd:\, |\bx| \le \xi_{\theta,n}^*}.$$
Notice  that the set of integration nodes \eqref{int-nodes-theta} possesses  similar sparsity properties as  the set \eqref{int-nodes}.

In a way similar to the proof of Theorem \ref{thm:int-general} we derive

\begin{theorem} \label{thm:int-general2}
	Let $r\in \NN$, $1<p<\infty$ and $\alpha >0$, $\beta \ge 0$, $1 < \theta < 2$. 
	Assume that for any $m \in \RR_1$, there is  a   quadrature  $Q_m$  of the form \eqref{Q_m(f)} with 
	$\{\bx_1,\ldots,\bx_m\}\subset \brac{- \frac{1}{2}, \frac{1}{2}}^d$ satisfying \eqref{IntError-a,b,F}. Then for  the   quadrature  $Q_{\theta,n}$ defined as in \eqref{Q_n^gamma3} we have 
	\begin{equation} 	\label{IntError}
		\bigg|\int_{\RRd}f(\bx) \mu(\rd\bx) - 	Q_{\theta,n}(f)\bigg| 
		\ll 
		n^{-\alpha}  (\log n)^\beta  \|f\|_{\Wpgamma}, 
		\ \  f \in \Wpgamma.
	\end{equation}
\end{theorem}

\subsection{Optimal quadrature}
\label{Optimal numerical integration}
In this subsection, we prove the right asymptotic order of the quantity of optimal quadrature as formulated in \eqref{AsympQuadrature} based on Theorem \ref{thm:int-general2} and known results on  numerical integration for functions from $\Wpmix$.

\begin{theorem} \label{thm:main}
	Let $r\in \NN$ and $1<p<\infty$.  Then one can construct  an asymptotically optimal  family of quadratures  of the form  \eqref{Q_n^gamma3} $\big(Q_n^\mu (f)\big)_{n \in \RR_1}$ so that   there hold the right asymptotic orders
	\begin{equation}\label{eq:InWg}
		\sup_{f\in \BWpgamma} \bigg|\int_{\RRd}f(\bx) \mu(\rd\bx) - Q_n^\mu f)(f)\bigg| 
		\asymp
		\Int_n\big(\BWpgamma\big) 
		\asymp 
		n^{-r} (\log n)^{\frac{d-1}{2}}.
	\end{equation}
\end{theorem}

\begin{proof} 
	Let $Q_{{\rm F},m}$ be the  Frolov quadrature for functions in $\mathring{W}^r_p(\IId)$ (see, e.g.,  \cite[Chapter 8]{DTU18B}  for the definition)    in
	the form \eqref{Q_m(f)} with 
	$\{\bx_1,\ldots,\bx_m\}\subset \brac{- \frac{1}{2}, \frac{1}{2}}^d$. 
	It was proven   in \cite{Florov1976} for $p=2$, and in \cite{Skriganov1994} for $1 < p < \infty$  that
	\begin{equation}\label{OptIntegration}
		\bigg|\int_{\IId} f(\bx) \rd \bx  - {I_{{\rm F},m}}\bigg| \leq C n^{-r} (\log n)^{\frac{d-1}{2}} \|f\|_{\mathring{W}^r_p(\IId)}, 
		\ \  f\in \mathring{W}^r_p(\IId).
	\end{equation}		
	For a fixed $1 < \theta < 2$, we define $Q_n^w(f):=  Q_{\theta,n}$  as the quadrature described in  Theorem \ref{thm:int-general2} for $a= r$ and $b = {\frac{d-1}{2}}$, based on  $Q_m = {Q_{{\rm F},m}}$. By 
	Theorem \ref{thm:int-general2} and \eqref{OptIntegration} we prove the upper bound in \eqref{eq:InWg}.
	
	Since for $f\in \mathring{W}^r_p(\IId)$
	$$
	\|f\|_{\Wpgamma} \leq e^{\frac{db}{p}} \|f\|_{\mathring{W}^r_p(\IId)},
	$$
	we get
	\begin{align*}  
		\Int_n\big(\BWpgamma\big) \gg \Int_n(\mathring{\boldsymbol{W}}^r_p(\IId)).
	\end{align*}	
	Hence the lower bound in \eqref{eq:InWg} follows from the lower bound 
	$$
	\Int_n(\mathring{\boldsymbol{W}}^r_p(\IId)) \gg  n^{-r} (\log n)^{\frac{d-1}{2}}
	$$ proven in \cite{Tem1990}. 
	%
	\hfill
\end{proof}

Besides Frolov quadratures, there are many  quadratures for efficient numerical integration for functions on $\IId$ to list. We refer the reader to \cite[Chapter 8]{DTU18B} for  bibliography  and historical comments as well as related results, in particular,
the asymptotic order
\begin{equation*}\label{IntW(IId)}
	\Int_m\big(\BWpmix\big) \asymp m^{-r} (\log m)^{\frac{d-1}{2}}.
\end{equation*}
We recall only some of them, especially those which give asymptotic order of  optimal integration.

The QMC quadrature based on a set of integration nodes  $\{\bx_1,\ldots,\bx_m\}\subset \IId$  is defined as
\begin{equation*}\label{Q_m(f)-QMC}
	Q_m(f) = \frac{1}{m}\sum_{i=1}^m f(\bx_i). 
\end{equation*}
In  \cite{Dick2007, Dick2008} for a prime number $q$  the author introduced higher order digital nets over the finite field 
$\FF_q:=\brab{0, 1, \ldots, q-1}$ equipped with the arithmetic operations modulo $q$.  Such digital nets can achieve the convergence rate $m^{-r}(\log m)^{d r}$ with $m = q^s$ \cite{DP2010} for functions from $W^r_2(\IId)$.
In the recent paper \cite{GSY2018},  the authors have shown that   the asymptotic order  of
$\Int_m\big(\bW^r_2(\IId)\big)$  can be achieved  by Dick's digital nets $\{\bx_1^*,\ldots,\bx_{q^s}^*\}$ of order $(2r + 1)$. Namely, they proved that 
\begin{equation}\label{OptQMC}
	\bigg|\int_{\IId} f(\bx) \rd \bx  -  \frac{1}{m} \sum_{i=1}^m f(\bx_i^*)\bigg| \leq C 
	m^{-r} (\log m)^{\frac{d-1}{2}} \|f\|_{W^r_2(\IId)}, 
	\ \  f\in \ W^r_2(\IId), \ \ m=q^s.
\end{equation}		

In the case $d=2$ the QMC quadrature $Q_m=Q_{\Phi,m}$ based on Fibonacci lattice rules ($d=2$)  is also  asymptotically optimal for numerical integration of  periodic functions in   $\tilde{W}^r_p(\II^2)$, that is,
\begin{equation}\label{OptInt-Fibonacci}
	\bigg|\int_{\II^2} f(\bx) \rd \bx  - Q_{\Phi,m}(f)\bigg| \leq C m^{-r} (\log m)^{\frac{1}{2}} \|f\|_{W^r_p(\TT^2)}, 
	\ \  f\in \tilde{W}^r_p(\II^2),
\end{equation}
where $\tilde{W}^r_p(\II^2)$ denotes the subspace of $W^r_p(\II^2)$ of all functions which can be extended to the whole $\RR^2$ as $1$-periodic  functions in each variable. The estimate \eqref{OptInt-Fibonacci} was proven in \cite{Bakhvalov1963} for $p=2$ and in \cite{Tem1991}
for $1 < p < \infty$.
The QMC quadrature $Q_m=Q_{\Phi,m}$ based on Fibonacci lattice rules ($d=2$)  is defined by
\begin{equation*}\label{Phi_m(f)}
	Q_{\Phi,m}(f): = \frac{1}{b_m} \sum_{i=1}^{b_m} f\bigg(\Big\{\frac{i}{b_m}\Big\} - \frac{1}{2}, \Big\{\frac{ib_{m-1}}{b_m}\Big\} - \frac{1}{2}\bigg),
\end{equation*}
where $b_0 = b_1 = 1$,  $b_m := b_{m-1} + b_{m-2}$ are the Fibonacci numbers and $\brab{x}$ denotes the fractional part of the number $x$.

Therefore,  from Theorems \ref{thm:int-general}--\ref{thm:main} and  \eqref{OptQMC}, \eqref{OptInt-Fibonacci} it follows that the QMC quadratures based on Dick's digital nets  of order $(2r + 1)$  and  Fibonacci lattice rules ($d=2$) can be used for  assembling  asymptotically optimal quadratures  $Q_n^w(f)$ and $Q_{\theta,n}^\gamma$  of the forms \eqref{Q_n^gamma2} and \eqref{Q_n^gamma3} for 
$\Int_n\big(\BWpgamma\big)$, in the particular cases $p=2$, $d \ge 2$, and $1 < p < \infty$, $d=2$, respectively.

The sparse Smolyak grid $SG(\xi)$ in $\IId$ is defined as  the set of points:
\begin{equation*}\label{Smolyak-net}
	SG(\xi):= \brab{\bx_{\bk,\bs}:= 2^{-\bk}\bs \in \ZZd: \, |\bk|_1 \le \xi, \ \ | s_i| \le 2^{k_i - 1}, \ i=1,\ldots,d}, 
	\ \ \xi  \in \RR_1. 
\end{equation*}
For a given $m \in \RR_1$, let $\xi_m$ be the maximal number satisfying $|SG(\xi_m)| \le m$. Then we can constructively define a quadrature $Q_m = {Q_{{\rm S},m}}$ based on the integration nodes in $SG(\xi_m)$ so that
\begin{equation}\label{OptInt-Smolyak}
	\bigg|\int_{\IId} f(\bx) \rd \bx  - {Q_{{\rm S},m}}(f)\bigg| \leq C m^{-r} (\log m)^{(d-1)(r + 1/2)} \|f\|_{\Wpmix}, 
	\ \  f\in W^r_p(\IId).
\end{equation}
To understand this quadrature let us recall  a detailed construction from \cite[p. 760]{DU2015}.
Indeed, the well-known embedding of $\Wpmix$ into the Besov space  of mixed smoothness $B^r_{p,\max(p,2)}(\IId)$  (see, e.g., \cite[Lemma 3.4.1(iv)]{DTU18B}), and the result on B-spline sampling  recovery of functions from the last space   it follows  that one can constructively define a sampling recovery algorithm of the form 
\begin{equation*}\label{R_m(f)}
	R_m(f): = \sum_{\bx_{\bk,\bs}\in SG(\xi_m)} f(\bx_{\bk,\bs}) \phi_{\bk,\bs}
\end{equation*}
with certain B-splines $\phi_{\bk,\bs}$, such that
\begin{equation*}\label{Sampling-Smolyak}
	\norm{f - R_m(f)}{L_1(\IId)}	 \leq C m^{-r} (\log m)^{(d-1)(r + 1/2)} \|f\|_{\Wpmix}, 
	\ \  f\in W^r_p(\IId).
\end{equation*}
Then the quadrature ${Q_{{\rm S},m}}$ can be defined as
\begin{equation*}\label{S_m(f)}
	{Q_{{\rm S},m}}(f): = \sum_{\bx_{\bk,\bs} \in SG(\xi_m)} \lambda_{\bk,\bs} f(\bx_{\bk,\bs}) ,  \ \ 
	\lambda_{\bk,\bs}:= \int_{\IId} \phi_{\bk,\bs}(\bx) \rd \bx,
\end{equation*} 
and \eqref{OptInt-Smolyak} is implied by the obvious inequality 
$$
\big|\int_{\IId} f(\bx) \rd \bx  - S_m(f)\big| \le \norm{f - R_m(f)}{L_1(\IId)}.
$$
Therefore,  from Theorem \ref{thm:int-general} and  \eqref{OptInt-Smolyak} we can see that the  Smolyak quadrature ${Q_{{\rm S},m}}$  can be used for  assembling  a  quadrature  
	${Q_{{\rm S},n}^\mu}$ of the form \eqref{Q_n^gamma2} with ``double" sparse integration nodes which gives the convergence rate
\begin{equation*}
	\bigg|\int_{\RRd}f(\bx) \mu(\rd\bx) - {Q_{{\rm S},n}^\mu}(f)\bigg| 
	\ll
	n^{-r} (\log n)^{(d-1)(r + 1/2)}, \ \ 	f \in \BWpgamma.
\end{equation*}

We  illustrate the integration nodes of the quadratures constructed in this subsection, in comparison with the integration nodes used in \cite{DILP18}. Assume that $\{\bx_1,\ldots,\bx_n\}$ are the integration nodes for  an optimal quadrature $Q_n$  for functions in ${W}^r_p(\II^2)$. Then the integration nodes in \cite{DILP18} are just a dilation of these nodes to the cube $[-C\sqrt{\log n}, C\sqrt{\log n}]^2$. Hence these nodes are distributed similarly on this cube. Differently,  the integration nodes in our construction  are  formed from certain integer-shifted dilations of $\{\bx_1,\ldots,\bx_m\}$ and contained in the ball of radius $C\sqrt{\log n}$. These nodes are dense when they are near the origin and getting sparser as they are farther from the origin.  The illustration is given in Figure \ref{Fig-1}.
\begin{figure}
\begin{tabular}{cc}
	\includegraphics[height=7.5cm]{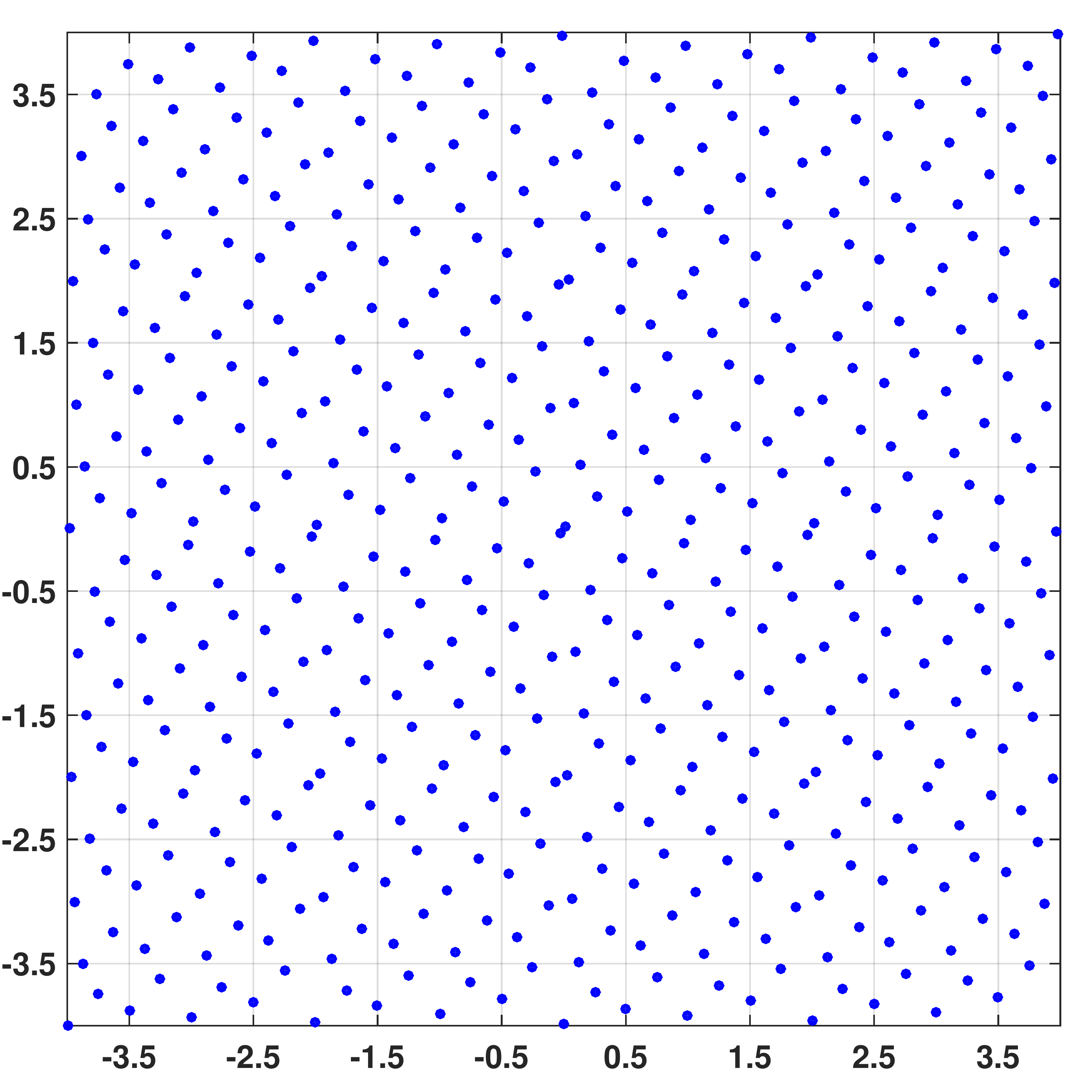}	 &  \includegraphics[height=7.5cm]{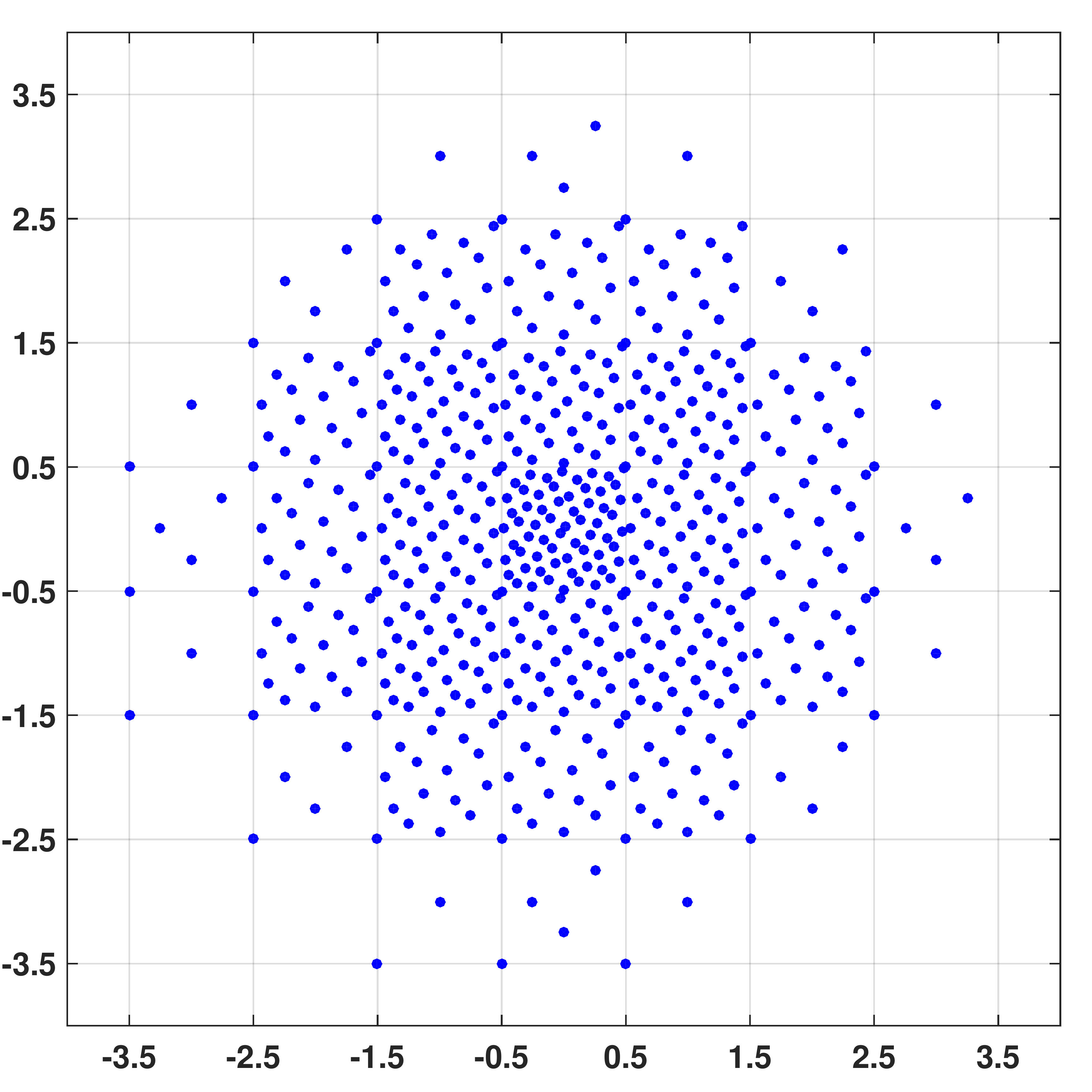}
	\\
	{Point set of Dick et al. \cite{DILP18} (512 points)} & {Point set of our construction (560 points)}
\end{tabular}
\caption{Pictures of integration nodes  from \cite{DK2023}}
\label{Fig-1}
\end{figure}

	\section{Numerical weighted integration in the space $W^r_1(\mathbb{R}^d;\mu)$}
		\label{Numerical weighted integration 1}
	\subsection{Introducing remarks}
\label{Introducing remarks sec 5}
	In this section, we present 
	some recent results of \cite{DD2023} on numerical weighted integration over $\mathbb{R}^d$ for functions from weighted  Sobolev spaces  $W^r_1(\mathbb{R}^d;\mu)$ of  mixed smoothness $r \in \mathbb{N}$, in particular, 
	 upper and lower bounds of the quantity of optimal quadrature of the functions class $W^r_1(\mathbb{R}^d;\mu)$. Here the measure $\mu$ is defined  by the Freud-type weight as the density function given by \eqref{w(bx)}.  
	We first briefly describe these results and then give comments on related works.
	
	 Throughout this section, for $\lambda >1$, we make use of  the notation
	$$
	r_\lambda:= (1 - 1/\lambda)r.
	$$ 
	For  the Sobolev class $\bW^r_1(\mathbb{R}^d;\mu)$, there hold the upper and lower bounds 
	\begin{equation}\label{Int_n(W^r_1)}
	n^{-r_\lambda} (\log n)^{r_\lambda(d-1)} 
	\ll	
	\Int_n(\bW^r_1(\mathbb{R}^d;\mu)) 
	\ll 
	n^{-r_\lambda} (\log n)^{(r_\lambda + 1)(d-1)},
	\end{equation}
	in particular, in the case of Gaussian measure
	\begin{equation}\label{Int_n(W^r_1)-tau=2}
	n^{-r/2} (\log n)^{r(d-1)/2} 
	\ll	
	\Int_n(\bW^r_1(\RRd; \gamma)) 
	\ll 
	n^{-r/2} (\log n)^{(r/2 + 1)(d-1)}.
	\end{equation}
	In  the one-dimensional  case, we have the right asymptotic order
	\begin{equation*}\label{Int_n(W)-introduction-d=1}
	\Int_n(\bW^r_1(\mathbb{R};\mu)) 
	\asymp 
	n^{-r_\lambda}.
	\end{equation*}
	The difference between the upper and lower bounds in \eqref{Int_n(W^r_1)} is the logarithmic factor 
	$(\log n)^{d-1}$.
	
	In Theorem \ref{thm:main}, for $1<p<\infty$, we have constructed  an  asymptotically optimal quadrature  $Q_n^\mu$ of the form \eqref{Q_nf-introduction} which  gives the right asymptotic order  
	\begin{equation} 	\label{w-quadrature-bounds}
	\sup_{f\in \bW^r_p(\mathbb{R}^d;\mu)} \bigg|\int_{\RRd}f(\bx) \mu(\rd\bx) - Q_n^\mu f\bigg| 
	\asymp
	\Int_n\big(\bW^r_p(\mathbb{R}^d;\mu)) \big) 
	\asymp
	n^{-r} (\log n)^{(d-1)/2}.
	\end{equation}
	The results \eqref{Int_n(W^r_1)-tau=2}  and \eqref{w-quadrature-bounds} show a substantial difference of the convergence rates between the cases $p=1$ and $1 < p < \infty$.
	In constructing the asymptotically optimal quadrature $Q_n^\mu$ in \eqref{w-quadrature-bounds}, we used a technique assembling a quadrature for the Sobolev spaces on the unit $d$-cube to the  integer-shifted $d$-cubes. Unfortunately, this technique is not suitable to constructing a quadrature   realizing the upper bound in \eqref{Int_n(W^r_1)} for the space  
	$W^r_1(\mathbb{R}^d;\mu)$ which is the largest among the spaces $W^r_p(\mathbb{R}^d;\mu)$ with $1\le p < \infty$. It requires  a different technique based on Smolyak algorithms \cite{Smo63}. Such a quadrature relies on  sparse grids of  integration nodes which are step hyperbolic crosses in the function domain $\RRd$,  and some generalization of the results on univariate numerical integration by truncated Gaussian quadratures from \cite{DM2003}. To prove the lower bound in \eqref{Int_n(W^r_1)} we adopt a traditional technique to construct for arbitrary $n$ integration nodes a fooling function vanishing at these nodes. 
	
	It is interesting to compare the results  \eqref{Int_n(W^r_1)-tau=2}  and \eqref{w-quadrature-bounds} on	$\Int_n\big(\bW^r_p(\mathbb{R}^d;\mu)\big)$ with known results  on
	$\Int_n\big(\bW^r_p(\IId)\big)$ for the unweighted Sobolev space $W^r_p(\IId)$ of mixed smoothness $r$. For  $1<p<\infty$, there holds the asymptotic order
	\begin{equation*}\label{Int_n-p>1-unweighted}
	\Int_n\big(\bW^r_p(\IId)\big) 
	\asymp 
	n^{-r} (\log n)^{(d-1)/2},
	\end{equation*}
	and for $p=1$ and $r > 1$, there hold the  bounds
	\begin{equation*}\label{Int_n(W)IId}
	n^{-r} (\log n)^{(d-1)/2} 
	\ll	
	\Int_n(\bW^r_1(\IId)) 
	\ll 
	n^{-r} (\log n)^{d-1}
	\end{equation*}
	which are so far  the best known (see, e.g., \cite[Chapter 8]{DTU18B}, for detail).	
	Hence we can see that $\Int_n\big(\bW^r_p(\mathbb{R}^d;\mu)\big)$ and $\Int_n\big(\bW^r_p(\IId)\big)$ have the same asymptotic order  in the case 	$1<p<\infty$, and very different lower and upper bounds in both power and   logarithmic terms in the case $p=1$.
	The right asymptotic orders of the both $\Int_n\big(\bW^r_1(\IId)\big)$ and 
	$\Int_n\big(\bW^r_1(\mathbb{R}^d;\mu)\big)$ are still open problems (cf. \cite[Open Problem 1.9]{DTU18B}).
	
	The problem of numerical integration considered in this section is  related to  the research direction of optimal approximation and integration for functions having mixed smoothness on one hand, and the other research  direction of univariate weighted polynomial approximation and integration on $\RR$, on the other hand.  For survey and bibliography, we refer the reader to the books  \cite{DTU18B,Tem18B} on the first direction, and \cite{Mha1996B,Lu07B,JMN2021} on the other  hand.

	\subsection{Univariate numerical integration}
\label{Univariate integration}

In this subsection, for one-dimensional numerical integration, we present the right asymptotic order of  the quantity of optimal quadrature $\Int_n\big(\bW^r_1(\mathbb{R};\mu)\big)$ and its shortened proof  from \cite{DD2023}.

The following lemma which is implied directly from the definition \eqref{Int_n}  is quite useful  for lower estimation of $\Int_n(\bW)$.
 
\begin{lemma} \label{lemma:Int_n>}
Let $\bW$ be a set of continuous functions on $\RRd$. 	Then we have
\begin{equation} \label{Int_n>}
	\Int_n(\bW) \ge \inf_{\brab{\bx_1,...,\bx_n} \subset \RRd} \ \sup_{f\in \bW: \ f(\bx_i)= 0,\ i =1,...,n}\bigg|\int_{\RRd} f(\bx) \, \mu(\rd\bx)\bigg|.
\end{equation}
\end{lemma}
	
We now consider the problem of approximation of  integral \eqref{If} for univariate functions from 
$W^r_1(\RR;\mu)$. Let $(p_m(w))_{m \in \NN_0}$ be the sequence of orthonormal polynomials with respect to the weight $w$. In the classical quadrature theory, a possible choice of integration nodes is to take the zeros of the polynomials $p_m(w)$.
Denote by $x_{m,k}$, $1 \le k \le \lfloor m/2 \rfloor$ the positive zeros of 	$p_m(w)$, and by $x_{m,-k} = - x_{m,k}$ the negative ones (if $m$ is odd, then $x_{m,0}= 0$ is also a zero of $p_m(w)$). These zeros are located as 
\begin{equation}\label{zeros-location}
-a_m + \frac{Ca_m}{m^{2/3}} < x_{m,- \lfloor m/2 \rfloor} < \cdots 
< x_{m,-1} < x_{m,1} < \cdots <  x_{m,\lfloor m/2 \rfloor} \le a_m - \frac{Ca_m}{m^{2/3}}, 
\end{equation}
with a positive constant $C$ independent of $m$ (see, e. g., \cite[(4.1.32)]{JMN2021}). Here $a_m$ is the Mhaskar-Rakhmanov-Saff number which is
\begin{equation}\label{a_m}
	a_m = a_m(w)= (\gamma_\lambda m)^{1/\lambda}, \ \ \gamma_\lambda:= \frac{2 \Gamma((1+\lambda)/2)}{\sqrt{\pi} \Gamma(\lambda/2)},
\end{equation}
and $\Gamma$ is the gamma function. Notice that the formula \eqref{a_m} is given in \cite[(4.1.4)]{JMN2021} for the particular case $a=1, \ b=0$ of the weight $w$, but inspecting the definition of Mhaskar-Rakhmanov-Saff number (see, e.g., \cite[Page 116]{JMN2021}), one easily verify that it still holds true for the general weight $w$ for any $a >0$ and  $b\in \RR$.

 For a continuous function on $\RR$, the classical Gaussian quadrature is defined as
\begin{equation} \label{G-Q_mf}
	Q^{\rm G}_mf: = \sum_{|k| \le m/2 } \lambda_{m,k}(w) f(x_{m,k}), 
\end{equation} 
where $\lambda_{m,k}(w)$ are the corresponding Cotes numbers. This quadrature is based on Lagrange interpolation (for details, see, e.g., \cite[1.2. Interpolation and quadrature]{Mha1996B}). Unfortunately, it does not give the optimal convergence rate for functions from $\bW^r_1(\mathbb{R};\mu)$, see comments in Remark \ref{Comment on G-quadrature} below. 

In \cite{DM2003}, for the weight $w(x)$ with $a=1, \ b=0$, the authors proposed truncated Gaussian quadratures which not only improve the convergence rate but also give the asymptotic order of  $\Int_n\big(\bW^r_1(\mathbb{R};\mu)\big)$ as shown in Theorem \ref{thm:Q_n-d=1} below. Let us introduce  in the same manner truncated Gaussian quadratures for the weight 
$w(x)$ with any $a >0$ and  $b\in \RR$.

In what follows, we fix a number $\theta$ with $0 < \theta < 1$, and denote by $j(m)$ the smallest integer satisfying $x_{m,j(m)} \ge \theta a_m$. 
It is useful to remark that
\begin{equation*} \label{x_k+1 - x_k}
d_{m,k} \, \asymp  \, \frac{a_m}{m} \asymp m^{1/\lambda -1}, \ \ |k| \le j(m); \quad x_{m,j(m)} \, \asymp \, m^{1/\lambda},
\end{equation*} 
where $d_{m,k}:= x_{m,k} - x_{m,k-1}$ is the distance between consecutive zeros of  the polynomial $p_m(w)$. These relations were proven  in \cite[(13)]{DM2003} for the case particular $w(x) = e^{-|x|^\lambda}$. From their proofs there, one can easily see that they are still hold true for the general case of the weight $w$. Hence by the definitions and \eqref{zeros-location}, for $m$ sufficiently large we have that 
\begin{equation} \label{<j(m)<}
Cm \le j(m) \le m/2
\end{equation} 
with a positive constant $C$ depending on $\lambda, a, b$ and $\theta$ only. 

For a continuous function on $\RR$, consider 
the truncated Gaussian quadrature 
\begin{equation} \label{G-Q_mf-TG}
	Q^{\rm{TG}}_{2 j(m)}f: = \sum_{|k| \le j(m)} \lambda_{m,k}(w) f(x_{m,k}). 
\end{equation} 
Notice that the number $2j(m)$ of samples in the quadrature $Q^{\rm{TG}}_{2 j(m)}f$ is strictly smalller than $m$ -- the number of samples in the quadrature  $Q^{\rm G}_mf$.  However, due to \eqref{<j(m)<} it has the same asymptotic order as $2 j(m)\asymp m$ when $m$ going to infinity.

\begin{theorem} \label{thm:Q_n-d=1}
	For any $n \in \NN$, let $m_n$ be the largest integer such that $2 j(m_n) \le n$. Then the family of quadratures 	$\big(Q^{\rm{TG}}_{2 j(m_n)}\big)_{n \in \RR_1}$ is  asymptotically optimal  for $\bW^r_1(\mathbb{R};\mu)$, and
	\begin{equation}\label{Q_n-d=1}
		\sup_{f\in \bW^r_1(\mathbb{R};\mu)} \bigg|\int_{\RR}f(x) \mu(\rd x) - Q^{\rm{TG}}_{2 j(m_n)}f\bigg| 
		\asymp
		\Int_n\big(\bW^r_1(\mathbb{R};\mu)\big) 
		\asymp 
		n^{- r_\lambda}.
	\end{equation}
\end{theorem}

\begin{proof} 
For $f \in W^r_1(\RR;\mu)$, there holds the inequality
	\begin{equation}\label{DellaVecchia2003}
	\bigg|\int_{\RR}f(x) \mu(\rd x) - Q^{\rm{TG}}_{2 j(m)}f\bigg| 
	\le 
C\brac{m^{-(1 - 1/\lambda)r} \norm{f^{(r)}}{L_1(\RR;\mu)} + e^{-Km}\norm{f}{L_1(\RR;\mu)} }
\end{equation}
with some constants $C$ and $K$ independent of $m$ and $f$.  The inequality \eqref{DellaVecchia2003} implies the upper bound in \eqref{Q_n-d=1}:
	\begin{equation*}\label{Q_n-d=1(2)}
\Int_n\big(\bW^r_1(\mathbb{R};\mu)\big)  \le 
	\sup_{f\in \bW^r_1(\mathbb{R};\mu)} \bigg|\int_{\RR}f(x) \mu(\rd x) - Q^{\rm{TG}}_{2 j(m_n)}f\bigg| 
	\ll
	n^{- r_\lambda}.
\end{equation*}

 In order to prove the lower bound in \eqref{Q_n-d=1} we  apply Lemma \ref{lemma:Int_n>}. Let $\brab{\xi_1,...,\xi_n} \subset \RR$ be arbitrary $n$ points. For a given $n \in \NN$, we put $\delta = n^{1/\lambda - 1}$ and $t_j = \delta j$, $j \in \NN_0$. Then there is $i \in \NN$ with 
$n + 1 \le  i \le 2n + 2$ such that the interval $(t_{i-1}, t_i)$ does not contain any point from the set $\brab{\xi_1,...,\xi_n}$.
Take a nonnegative function $\varphi \in C^\infty_0([0,1])$, $\varphi \not= 0$, and put
	\begin{equation*}\label{b_s}
b_0:= \int_0^1 \varphi(y) \rd y > 0, \quad b_s :=  \int_0^1 |\varphi^{(s)}(y)| \rd y, \ s = 1,...,r.
\end{equation*}		
Define the functions $g$ and $h$ on $\RR$ by
\begin{equation*}\label{g}
g(x):= 
\begin{cases}
\varphi(\delta^{-1}(x - t_{i-1})), & \ \ x \in (t_{i-1}, t_i), \\
0, & \ \ \text{otherwise},	
\end{cases}	
\end{equation*}
and
	\begin{equation*}\label{h}
		h(x):= (gw^{-1})(x).		
	\end{equation*}
By a direct computation we find that 
	\begin{equation*}\label{int-h^(s)w}
		\begin{aligned}
\norm{h}{W^r_1(\RR;\mu)}
	\le C n^{(1-1/\lambda)(k-1)} \le  C n^{(1-1/\lambda)(r-1)}.		
	\end{aligned}
\end{equation*}
If we define 
	\begin{equation*}\label{h-bar}
	\bar{h}:=  C^{-1} n^{-(1-1/\lambda)(r-1)}h,
\end{equation*}
then $\bar{h}$ is nonnegative, 
$\bar{h} \in \bW^r_1(\mathbb{R};\mu)$, $\sup \bar{h} \subset (t_{i-1},t_i)$ and
	\begin{equation*}\label{int-h^(s)w2}
	\begin{aligned}	
		\int_{\RR}(\bar{h}w)(x) \rd x 
 \gg n^{-(1-1/\lambda)r}
	\end{aligned}
\end{equation*}
 Since the interval $(t_{i-1}, t_i)$ does not contain any point from the set $\brab{\xi_1,...,\xi_n}$, we have $\bar{h}(\xi_k) = 0$, $k = 1,...,n$. Hence, by Lemma \ref{lemma:Int_n>},
	\begin{equation}\nonumber
		\Int_n\big(\bW^r_1(\mathbb{R};\mu)\big) 	\ge  \int_{\RR}\bar{h}(x) w(x) \rd x
	 \gg  n^{- r_\lambda}.
\end{equation}
	\hfill
\end{proof}

\begin{remark} \label{Comment on G-quadrature}
{\rm	
In the case of the Gaussian measure $\gamma$, the truncated Gaussian quadratures $Q^{\rm{TG}}_{2 j(m)}$ in Theorem \ref{thm:Q_n-d=1} give
\begin{equation}\label{Q_n-d=1-tau=2}
		\sup_{f\in \bW^1_1(\RR;\mu)} \bigg|\int_{\RR}f(x) \gamma(\rd x) - Q^{\rm{TG}}_{2 j(m_n)}f\bigg| 
\asymp
	\Int_n\big(\bW^1_1(\RR;\gamma)\big) 
	\asymp 
	n^{-r/2}.
\end{equation}
On the other hand, for the full Gaussian quadratures $Q^{\rm G}_n$,  it has been proven in \cite[Proposition 1]{DM2003}	the right asymptotic order
\begin{equation}\nonumber
	\sup_{f\in \bW^1_1(\RR;\gamma)}	\bigg|\int_{\RR}f(x) \gamma(\rd x)  - Q^{\rm G}_nf\bigg|
	\asymp
	n^{-1/6}
\end{equation}
which is much worse than the right asymptotic order of 
$	\Int_n\big(\bW^1_1(\RR;\gamma)\big) \asymp 	n^{-1/2}$ as 
in \eqref{Q_n-d=1-tau=2} for $r = 1$.
}
\end{remark}


	\subsection{Multivariate numerical  integration}
\label{Multivariate integration}

In this section, for multivariate numerical integration ($d \ge 2$), we present some results and their shortened proofs from \cite{DD2023} on upper and lower bounds of  the quantity of optimal quadrature $\Int_n\big(\bW^r_1(\mathbb{R}^d;\mu)\big)$ and  a construction of  quadratures based on  step-hyperbolic-cross grids of integration nodes which give the upper bounds. 

We need some auxiliary lemmata. 

For $\bx \in \RRd$ and $e \subset \brab{1,...,d}$, let $\bx^e \in \RRd$ be defined by $x^e_i := x_i$, $i \in e$, and $x^e_i := 0$ otherwise, and put $\bar{\bx}^e:= \bx - \bx^e$. With an abuse we write 
$(\bx^e,\bar{\bx}^e) = \bx$.

\begin{lemma} \label{lemma:g(bx^e}
	Let $1\le p \le \infty$,  $e \subset \brab{1,...,d}$ and $\br \in \ZZdp$. Assume that $f$ is a function on $\RRd$  such that for every $\bk \le \br$, $D^\bk f \in L_p(\RRd;\mu)$. 
	Put for  $\bk \le \br$ and $\bar{\bx}^e \in \RR^{d-|e|}$,
	\begin{equation*}\label{g(bx^e}
	g(\bx^e): =  D^{\bar{\bk}^e} f(\bx^e,\bar{\bx}^e).
	\end{equation*}
Then $D^\bs g \in L_p(\RR^{|e|};\mu)$ for every $\bs \le \bk^{e}$ and almost every 
$\bar{\bx}^e \in \RR^{d-|e|}$.
\end{lemma}

	Assume that there exists a sequence of quadratures $\brac{Q_{2^k}}_{k \in \NN_0}$ with
\begin{equation}\label{Q_2^kf}
	Q_{2^k}f: = \sum_{s=1}^{2^k} \lambda_{k,s} f(x_{k,s}), \ \ \{x_{k,1},\ldots,x_{k,2^k}\}\subset \RR,
\end{equation}
such that 
\begin{equation}\label{IntErrorR}
	\bigg|\int_{\RR} f(x) \mu(\rd x)  - Q_{2^k}f\bigg| 
	\leq C 2^{-\alpha k} \|f\|_{W^r_1(\RR;\mu)}, 
	\ \  \ k \in \NN_0,  \ \ f \in W^r_1(\RR;\mu), 
\end{equation}
for some  number $a>0$ and constant $C>0$. 
Based on a sequence $\brac{Q_{2^k}}_{k \in \NN_0}$ of the form \eqref{Q_2^kf} satisfying \eqref{IntErrorR}, we construct quadratures on $\RRd$ by using the well-known Smolyak algorithm. We define for $k \in \NN_0$, the one-dimensional operators
\begin{equation*}\label{DeltaI}
\Delta_k^Q:= Q_{2^k} - Q_{2^{k-1}}, \  k >0, \ \ \Delta_0^Q:= Q_1,  
\end{equation*}
and 
\begin{equation*}\label{E}
	E_k^Qf:= \int_{\RR} f(x) \mu(\rd x) - Q_{2^k}f.
\end{equation*}
For $\bk \in \NNd$, the $d$-dimensional operators $Q_{2^\bk}$,  $\Delta_\bk^Q$ and $E_\bk^Q$ are defined as the tensor  product of one-dimensional operators:
\begin{equation}\label{tensor-product}
	Q_{2^\bk}:= \bigotimes_{i=1}^d Q_{2^{k_i}}, \ \	
	\Delta_\bk^Q:= \bigotimes_{i=1}^d \Delta_{k_i}^Q, \ \ 	
	E_\bk^Q:= \bigotimes_{i=1}^d E_{k_i}^Q,  
\end{equation}
where $2^\bk:= (2^{k_1},\cdots, 2^{k_d})$ and 
the univariate operators $Q_{2^{k_j}}$, $\Delta_{k_j}^Q$ and $E_{k_j}^Q$ 
 are successively applied to the univariate functions $\bigotimes_{i<j} Q_{2^{k_i}}(f)$, $\bigotimes_{i<j} \Delta_{k_i}^Q(f)$ and $\bigotimes_{i<j} E_{k_i}^Q $, respectively, by considering them  as 
functions of  variable $x_j$ with the other variables held fixed. The operators $Q_{2^\bk}$, $\Delta_\bk^Q$ and $E_\bk^Q$ are well-defined for continuous functions on $\RRd$, in particular for ones from $W^r_1(\RRd;\mu)$.

If $f$ is a continuous function on $\RRd$, then  $Q_{2^\bk}(f)$ is a quadrature on $\RRd$ which is given by
\begin{equation}\label{Q_2^bkf}
	Q_{2^\bk}f = \sum_{\bs=\bone}^{2^\bk} \lambda_{\bk,\bs} f(\bx_{\bk,\bs}), 
	\ \ \{\bx_{\bk,\bs}\}_{\bone \le \bs \le 2^\bk}\subset \RRd,
\end{equation}
where $$
\bx_{\bk,\bs}:= \brac{x_{k_1,s_1},...,x_{k_d,s_d}}, \ \ \ 
\lambda_{\bk,\bs}:= \prod_{i=1}^d \lambda_{k_i,s_i} ,
$$
 and the summation $\sum_{\bs=\bone}^{2^\bk}$ means that the sum is taken over all $\bs$ such that $\bone \le \bs \le 2^\bk$. Hence we derive that 
\begin{equation}\label{Delta_bk}
	\Delta_\bk^Q f =  \sum_{e \subset \brab{1,...,d}} (-1)^{d - |e|}Q_{2^{\bk(e)}} f
	= \sum_{e \subset \brab{1,...,d}} (-1)^{d - |e|}\sum_{\bs=\bone}^{2^{\bk(e)}} \lambda_{\bk(e),\bs} f(\bx_{\bk(e),\bs}), 
\end{equation}
where  $\bk(e) \in \ZZdp$ is defined by $k(e)_i = k_i$, $i \in e$, and 	$k(e)_i = \max(k_i-1,0)$, $i \not\in e$.  We also have
\begin{equation}\label{E_bk}
	E_\bk^Q f =  \sum_{e \subset \brab{1,...,d}} (-1)^{|e|} 
	\int_{\RR^{d - |e|}}Q_{2^{\bk^e}}f(\cdot,\bar{\bx^e} ) w(\bar{\bx}^e) \rd \bar{\bx}^e,  
\end{equation}
where $w(\bar{\bx}^e):= \prod_{j \not\in e} w(x_j)$.

 Notice that as mappings from $C(\RRd)$ to $\RR$, the operators $Q_{2^\bk}$, $\Delta_\bk^Q$ and $E_\bk^Q$  possess commutative and associative properties  with respect to applying the component operators $Q_{2^{k_j}}$, $\Delta_{k_j}^Q$ and $E_{k_j}^Q$.  In particular, we have  for any $e \subset \brab{1,...,d}$,
$$
Q_{2^\bk}f= Q_{2^{\bk^e}}\brac{Q_{2^{\bar{\bk}^e}} f}, \ \ 
\Delta_\bk^Q f= \Delta_{\bk^e}^Q \brac{\Delta_{\bar{\bk}^e}^Q f}, \ \ 
E_\bk^Q f= E_{\bk^e}^Q \brac{E_{\bar{\bk}^e}^Q f},
$$
and for any reordered sequence $\brab{i(1),..., i(d)}$ of $\brab{1,...,d}$, 
\begin{equation}\label{commutative}
	Q_{2^\bk} = \bigotimes_{j=1}^d Q_{2^{k_{i(j)}}}, 
	\ \	\Delta_\bk^Q = \bigotimes_{j=1}^d \Delta_{k_{i(j)}}^Q, \ \ 	
	E_\bk^Q = \bigotimes_{j=1}^d E_{k_{i(j)}}^Q.
\end{equation}
  These properties directly follow from 
\eqref{Q_2^bkf}--\eqref{E_bk}.

\begin{lemma} \label{lemma:E_k}
	Under the assumption \eqref{Q_2^kf}--\eqref{IntErrorR}, we have
	\begin{equation*}\label{E_k}
		\big|E_\bk^Qf\big| 
		\leq C 2^{-\alpha|\bk|_1} \|f\|_{W^r_1(\RRd;\mu)}, 
		\ \  \bk \in \ZZdp, \ \ f \in W^r_1(\RRd;\mu), 
	\end{equation*}
\end{lemma}

We say that $\bk \to \infty$, $\bk \in \ZZdp$, if and only if $k_i \to \infty$ for every $i = 1,...,d$.
\begin{lemma} \label{lemma:Delta_k}
	Under the assumption \eqref{Q_2^kf}--\eqref{IntErrorR}, we have that
	for every $ f \in W^r_1(\RRd;\mu)$,
		\begin{equation}\label{Series1}
	\int_{\RRd} f(\bx) \mu(\rd \bx)
	= 	\sum_{\bk \in \ZZdp}\Delta_\bk^Qf 
	\end{equation}
with absolute convergence of the series, and
	\begin{equation}\label{Delta_k}
		\big|\Delta_\bk^Qf\big| 
		\leq C 2^{-\alpha|\bk|_1} \|f\|_{W^r_1(\RRd;\mu)}, 
		\ \  \bk \in \ZZdp.
	\end{equation}
\end{lemma}

We now define an  algorithm for quadrature on sparse grids adopted from the algorithm for sampling recovery initiated by Smolyak (for detail see \cite[Sections 4.2 and 5.3]{DTU18B}). For $\xi > 0$, we define the operator
\begin{equation}\nonumber
Q_\xi
:= 	\sum_{|\bk|_1 \le \xi } \Delta_{\bk}^Q.
\end{equation}	
From \eqref{Delta_bk} we can see that $Q_\xi$ is a quadrature on $\RRd$ of the form  \eqref{Q_nf-introduction}:
\begin{equation}\label{Q_xi}
Q_\xi f
= 	\sum_{|\bk|_1 \le \xi } \ \sum_{e \subset \brab{1,...,d}} (-1)^{d - |e|}\ \sum_{\bs=\bone}^{2^{\bk(e)}} \lambda_{\bk(e),\bs} f(\bx_{\bk(e),\bs})
= \sum_{(\bk,e,\bs) \in G(\xi)} \lambda_{\bk,e,\bs} f(\bx_{\bk,e,\bs}), 
\end{equation}
where 
$$
\bx_{\bk,e,\bs}:= \bx_{\bk(e),\bs}, \quad \lambda_{\bk,e,\bs}:= (-1)^{d - |e|}\lambda_{\bk(e),\bs}
$$ 
and 
\begin{equation}\nonumber
G(\xi)	:= \brab{(\bk,e,\bs): \ |\bk|_1 \le \xi, \,   e \subset \brab{1,...,d}, \,  \bone \le \bs \le \bk(e)}
\end{equation}
is a finite set.
The set of integration nodes in this quadrature 
\begin{equation}\nonumber
H(\xi):=\brab{\bx_{\bk,e,\bs}}_{(\bk,e,\bs) \in G(\xi)}
\end{equation}
is a step hyperbolic cross   in the function domain $\RRd$. 
The number of integration nodes in the quadrature $Q_\xi$ is 
\begin{equation}\nonumber
|G(\xi)|
= 	\sum_{|\bk|_1 \le \xi } \ \sum_{e \subset \brab{1,...,d}}2^{|\bk(e)|_1}
\end{equation}
which can be estimated as 
\begin{equation}\label{|G(xi)|}
|G(\xi)|
\asymp	\sum_{|\bk|_1 \le \xi }2^{|\bk|_1} \ \asymp \ 2^\xi \xi^{d - 1}, \ \ \xi \ge 1.
\end{equation}
This quadrature plays a crucial role in the proof of the upper bound in  \eqref{Int_n(W^r_1)}.
 
From Lemmata \ref{lemma:g(bx^e}--\ref{lemma:Delta_k},  by using a modification of a  technique  for establishing upper bounds of the error of unweighted sampling recovery by Smolyak algorithms of functions having mixed smoothness  on a bounded domain (see, e.g., \cite[Section 5.3]{DTU18B} and \cite[Section 6.9]{TB18} for detail of related techniques), we prove the following upper bound.		
\begin{lemma} \label{lemma:Q_xi-error}
	Under the assumption \eqref{Q_2^kf}--\eqref{IntErrorR}, we have that
\begin{equation} \label{upperbound}
	\bigg|\int_{\RRd} f(\bx) \mu(\rd \bx)  -  Q_\xi f\bigg| 
	\leq C 2^{-\alpha\xi} \xi^{d - 1} \|f\|_{W^r_1(\RRd;\mu)}, 
	\ \ \xi \ge 1, \ \ f \in W^r_1(\RRd;\mu). 
\end{equation}
\end{lemma}

\begin{proof}  From  Lemma \ref{lemma:Delta_k} we derive that for $\xi \ge 1$ and $f \in W^r_1(\RRd;\mu)$,
	\begin{equation}\nonumber
		\begin{aligned}
	\bigg|\int_{\RRd} f(\bx) w(\bx)\rd \bx  -  Q_\xi f\bigg| 
			& \le
				\sum_{|\bk|_1 > \xi } \big|\Delta_\bk^Qf\big| 
				\le 	C\sum_{|\bk|_1 > \xi } 2^{-\alpha|\bk|_1}  \|f\|_{W^r_1(\RRd;\mu)}
		\\
		&\leq 
	C \|f\|_{W^r_1(\RRd;\mu)} \sum_{|\bk|_1 > \xi } 2^{-\alpha|\bk|_1} 
	\leq C 2^{-\alpha\xi} \xi^{d - 1} \|f\|_{W^r_1(\RRd;\mu)}.
		\end{aligned}
	\end{equation}
	\hfill
\end{proof}
\begin{remark} \label{remark1}
	{\rm	
		From Theorem \ref{thm:Q_n-d=1} we can see that the truncated Gaussian quadratures $Q^{\rm{TG}}_{2 j(m)}$ 
		form  a sequence $\brac{Q_{2^k}}_{k \in \NN_0}$ of the form \eqref{Q_2^kf} satisfying  \eqref{IntErrorR} with $a = r_\lambda$. 
	}
\end{remark}

\begin{theorem} \label{theorem:Int_n(W)}
	We have that
	\begin{equation}\label{Int_n(W)}
	n^{-r_\lambda} (\log n)^{r_\lambda(d-1)} 
	\ll	
	\Int_n(\bW^r_1(\mathbb{R}^d;\mu)) 
	\ll 
	n^{-r_\lambda} (\log n)^{(r_\lambda + 1)(d-1)}.
	\end{equation}
\end{theorem}

\begin{proof}  
Let us first prove the upper bound in \eqref{Int_n(W)}. We  will construct a quadrature of the form \eqref{Q_xi} which realizes it.	
In order to do this, we take the truncated Gaussian quadrature 	$Q^{\rm{TG}}_{2 j(m)}f$ defined in \eqref{G-Q_mf}. For every $k \in \NN_0$, let $m_k$ be the largest number such that $2j(m_k) \le 2^k$.  Then we have 
$2j(m_k) \asymp 2^k$. 
For the sequence of quadratures  $\brac{Q_{2^k}}_{k \in \NN_0}$ with 
$$
Q_{2^k}:= Q^{\rm{TG}}_{2 j(m_k)} \in \Qq_{2^k},
$$
 from Theorem \ref{thm:Q_n-d=1} it follows that
\begin{equation*}\label{IntErrorR2}
	\bigg|\int_{\RR} f(x) w(x)\rd x - Q_{2^k}f\bigg| 
	\leq C 2^{- r_\lambda k} \|f\|_{W^r_1(\RR;\mu)}, 
	\ \  \ k \in \NN_0,  \ \ f \in W^r_1(\RR;\mu), 
\end{equation*}
This means that the assumption \eqref{Q_2^kf}--\eqref{IntErrorR} holds for $a = r_\lambda$. To prove the upper bound in \eqref{Int_n(W)} we approximate the integral \eqref{integral}
  by the quadrature $Q_{\xi}$ which is formed from the sequence $\brac{Q_{2^k}}_{k \in \NN_0}$.
For every $n \in \NN$, let $\xi_n$ be the largest number such that $|G(\xi_n)| \le n$.  Then the corresponding operator $Q_{\xi_n}$ defines a quadrature belonging to $\Qq_n$. From \eqref{|G(xi)|} it follows 
$$
\ 2^{\xi_n} \xi_n^{d - 1} \asymp |G(\xi_n)|  \asymp n.
$$ 
Hence we deduce the asymptotic equivalences
$$
\ 2^{-\xi_n}  \asymp n^{-1} (\log n)^{d-1}, \ \  \xi_n \asymp \log n,
$$
which together with Lemma \ref{lemma:Q_xi-error}  yield that
	\begin{equation*}\label{Q_xi-error}
		\begin{aligned}
		\Int_n(\bW^r_1(\mathbb{R}^d;\mu)) 
		&\le 
	\sup_{f\in \bW^r_1(\mathbb{R}^d;\mu)}	
	\bigg|\int_{\RRd} f(\bx) w(\bx)\rd \bx  -  Q_{\xi_n}f\bigg| 
		\\
		&
		\leq 
		C 2^{- r_\lambda \xi_n} \xi_n^{d - 1}
		\asymp  	n^{-r_\lambda} (\log n)^{(r_\lambda + 1)(d-1)}.
		\end{aligned}
	\end{equation*}
The upper bound in \eqref{Int_n(W)} is proven.

We now prove the lower bound  in \eqref{Int_n(W)} by using the inequality \eqref{Int_n>} in Lemma \ref{lemma:Int_n>}. For $M \ge 1$, we define the set 
	\begin{equation*}\label{Gamma-set}
\Gamma_d(M):= \brab{\bs \in \NNd: \, \prod_{i=1}^d s_i \le 2M, \ s_i \ge M^{1/d}, \ i=1,...,d}.
\end{equation*}
Then we have 
	\begin{equation}\label{|Gamma|}
	|\Gamma_d(M)| \asymp M (\log M)^{d-1}, \ \ M>1.
\end{equation}
For a given $n \in \NN$, let $\brab{\bxi_1,...,\bxi_n} \subset \RRd$ be arbitrary  $n$ points. Denote by $M_n$  the smallest number such that $|\Gamma_d(M_n)| \ge n + 1$. We define the $d$-cube $K_\bs$ for $\bs \in \ZZdp$ of size 
$$
\delta:= M_n^{\frac{1/\lambda - 1}{d}}
$$
 by
	\begin{equation*}\label{K_bs}
K_\bs:= \prod_{i=1}^d K_{s_i}, \ \ K_{s_i}:= (\delta s_i, \delta s_{i-1}).
\end{equation*}
 Since  $|\Gamma_d(M_n)| > n$, there exists a multi-index 
$\bs \in \Gamma_d(M_n)$ such that $K_\bs$ does not contain any point from $\brab{\bxi_1,...,\bxi_n}$.

We take a nonnegative function $\varphi \in C^\infty_0([0,1])$, $\varphi \not= 0$, and put
\begin{equation}\label{b_s-d}
	b_0:= \int_0^1 \varphi(y) \rd y > 0, \quad b_s :=  \int_0^1 |\varphi^{(s)}(y)| \rd y, \ s = 1,...,r.
\end{equation}		
For $i= 1,...,d$, we define the univariate functions $g_i$ in variable $x_i$ by
\begin{equation}\label{g_i}
	g_i(x_i):= 
	\begin{cases}
		\varphi(\delta^{-1}(x_i - \delta s_{i-1})), & \ \ x_i  \in K_{s_i}, \\
		0, & \ \ \text{otherwise}.	
	\end{cases}	
\end{equation}
Then the multivariate functions $g$ and $h$ on $\RRd$ are defined by
	\begin{equation*}\label{g(bx)}
	g(\bx):= \prod_{i=1}^d g_i(x_i),
\end{equation*}
and  
\begin{equation}\label{h(bx)}
	h(\bx):= (gw^{-1})(\bx)= \prod_{i=1}^d g_i(x_i)w^{-1}(x_i)=:
	\prod_{i=1}^d h_i(x_i).		
\end{equation}
Let us estimate the norm $\norm{h}{W^r_1(\RRd;\mu)}$. For every $\bk \in \ZZdp$ with 
$0 \le |\bk|_\infty \le r$,  we prove the inequality
\begin{equation}\label{int-D^br}
	\int_{\RRd}\big|(D^{\bk} h) w\big|(\bx) \rd \bx 
	\le  C   M_n^{(1 - 1/\lambda)(r - 1)}.
\end{equation}
This  inequality means that $h \in W^r_1(\RRd;\mu)$ and  
$$
\norm{h}{W^r_1(\RRd;\mu)} \le  C   M_n^{(1 - 1/\lambda)(r - 1)}.
$$
 If  we define 
\begin{equation*}\label{h-bar}
	\bar{h}:=  C^{-1} M_n^{-(1-1/\lambda)(r-1)}h,
\end{equation*}
then $\bar{h}$ is nonnegative, $\bar{h} \in \bW^r_1(\mathbb{R};\mu)$, $\sup \bar{h} \subset K_\bs$ and by \eqref{b_s-d}--\eqref{h(bx)},
\begin{equation}\nonumber
	\begin{aligned}	
		\int_{\RRd}(\bar{h}w)(\bx) \rd \bx 
		&=  C^{-1} M_n^{-(1-1/\lambda)(r-1)}\int_{\RRd}(h w)(\bx) \rd \bx = \prod_{i=1}^d \int_{K_{s_i}} g_i(x_i) \rd x_i
		\\
		&
		=  C^{-1} M_n^{-(1-1/\lambda)(r-1)} \brac{b_0 \delta}^d =  C' M_n^{- r_\lambda}.
	\end{aligned}
\end{equation}
From the definition of $M_n$ and \eqref{|Gamma|} it follows that
$$M_n(\log M_n)^{d-1} \asymp |\Gamma(M_n)| \asymp n,$$ 
 which implies that $M_n^{-1} \asymp n^{-1} (\log n)^{d-1}$. This allows to receive the estimate
\begin{equation}\label{int-h-bar2}
		\int_{\RRd}(\bar{h}w)(\bx) \rd \bx 
		=  C' M_n^{- r_\lambda}
		\gg
		n^{-r_\lambda} (\log n)^{r_\lambda(d-1)}.
\end{equation}
 Since the interval $K_\bs$ does not contain any point from the set $\brab{\bxi_1,...,\bxi_n}$ which has been arbitrarily chosen, we have 
 $$\bar{h}(\bxi_k) = 0, \ \ k = 1,...,n.
 $$
  Hence, by  Lemma \ref{lemma:Int_n>} and \eqref{int-h-bar2} we have that
\begin{equation*}\label{int-h-bar3}
	\Int_n(\bW^r_1(\mathbb{R}^d;\mu)) 
	\ge
	 \int_{\RRd}\bar{h}(\bx) w(\bx) \rd \bx
	\gg  n^{-r_\lambda} (\log n)^{r_\lambda(d-1)}.
\end{equation*}
The lower bound in \eqref{Int_n(W)} is proven.
	\hfill
\end{proof}

\begin{remark}
	{\rm	
		Let us analyse some properties of the quadratures $Q_\xi$ and their integration nodes $H(\xi)$ which give the upper bound in \eqref{Int_n(W)}. 
		\\
		\\
		1. The set of integration nodes $H(\xi)$ in the quadratures $Q_\xi$ which are formed from the non-equidistant zeros of the orthonornal polynomials $p_m(w)$,
		is a  step hyperbolic cross   \emph{on the function domain} $\RRd$. This is a contrast to the  classical theory of  approximation of multivariate periodic functions having mixed smoothness for which the classical step hyperbolic crosses of integer points are \emph{on the frequency domain} $\ZZd$ (see, e.g., \cite[Section 2.3]{DTU18B} for detail). The terminology 'step hyperbolic cross' of integration nodes is borrowed from there. In Figure \ref{Fig-2}, in particular, the step hyperbolic cross in the right picture is designed for the Hermite weight $w(\bx)= \exp(- x_1^2 - x_2^2)$ ($d=2$). The set  $H(\xi)$ also completely differs from the classical Smolyak grids of fractional dyadic points \emph{on the function domain} $[-1,1]^d$ (see Figure \ref{Fig-3} for $d=2$) which are used in sparse-grid sampling recovery and numerical integration for functions having a mixed smoothness (see, e.g., \cite{DTU18B} and \cite{BuGr04} for detail). 
		\\ 
		\\
		2. The set  $H(\xi)$ is  very sparsely distributed inside the $d$-cube 
		$$K(\xi): = \brab{\bx \in \RRd: \, |x_i| \le C2^{\xi/\lambda}, \ i =1,...,d},$$
		for some constant $C > 0$.
		Its diameter  which is the length of its symmetry axes is  $2C2^{\xi/\lambda}$, i.e.,  the size of $K(\xi)$. 
		The number of integration nodes in $H(\xi)$ is  $|G(\xi)|\asymp  2^\xi \xi^{d - 1}$. For the integration nodes $H(\xi)=\brab{\bx_{\bk,e,\bs}}_{(\bk,e,\bs) \in G(\xi)}$, we have that
		\begin{equation}\nonumber
		\min_{\substack{(\bk,e,\bs), (\bk',e',\bs') \in G(\xi) \\ (\bk,e,\bs) \not= (\bk',e',\bs')} } 
		\ \min_{1 \le i \le d} \left|\brac{x_{\bk,e,\bs}}_i - \brac{x_{\bk',e',\bs'}}_i \right|  \ \asymp \  2^{- (1 - 1/\lambda)\xi} \ \to 0, \ \text{when} \ \xi \to \infty.
		\end{equation}	
		On the other hand, the diameter of $H(\xi) $ is going to $\infty$ when $\xi \to \infty$.
	}
\end{remark}

\begin{figure}
	\centering{
		\begin{tabular}{cc}
			\includegraphics[height=7.0cm]{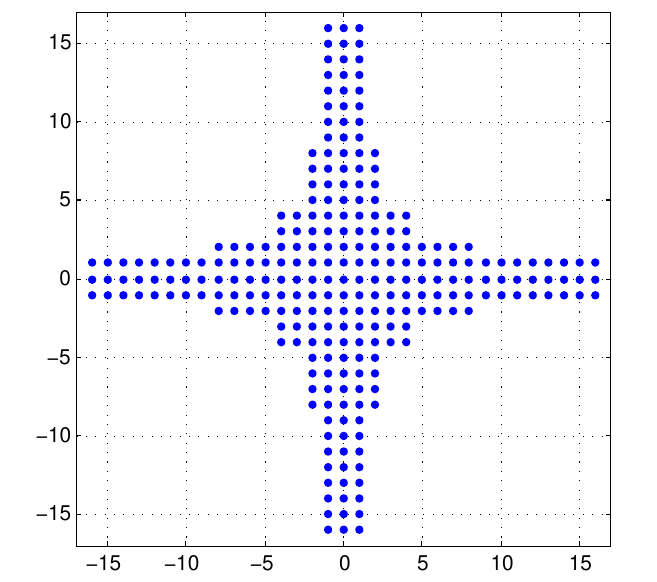}	 &  
			\includegraphics[height=7.0cm]{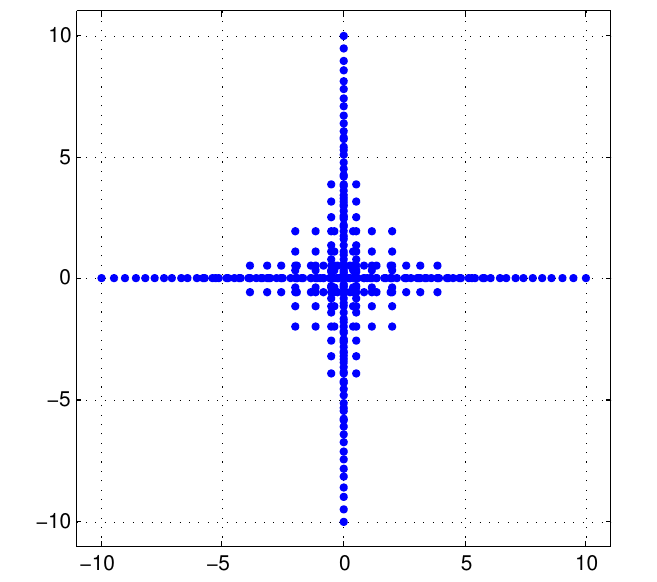}
			\\
			{A classical step hyperbolic cross} & { A Hermite step hyperbolic cross}
		\end{tabular}
	}
	\caption{Pictures of step hyperbolic crosses in $\RR^2$  from \cite{DD2023}}
	\label{Fig-2}
\end{figure}

\begin{figure}
	\centering{
		\begin{tabular}{cc}
			\includegraphics[height=7.0cm]{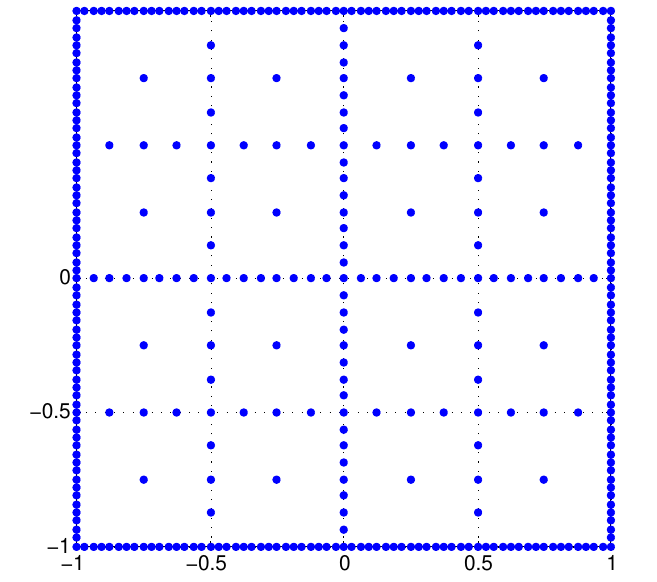}	
		\end{tabular}
	}
	\caption{A picture of Smolyak grid in $[-1,1]^2$ from \cite{DD2023}}
	\label{Fig-3}
\end{figure}

\subsection{Extension to Markov-Sonin weights}
\label{Extension}

In this subsection, we extend the results of the previous subsection to  Markov-Sonin weights. A univariate Markov-Sonin weight is a function of the form 
\begin{equation} \nonumber 
	w_\beta(x):= 	|x|^\beta \exp (- a|x|^2 + b), \ \ \beta > 0, \ \ a > 0, \ \ b \in \RR
\end{equation}
(here $\beta$ is indicated in the notation to distinguish Markov-Sonin weights $w_\beta$ and Freud-type weight $w$).  
A $d$-dimensional Markov-Sonin weight is defined as
\begin{equation} \nonumber
	w_\beta(\bx):= \prod_{i=1}^d w_\beta(x_i).
\end{equation}
Markov-Sonin weights are not of the form \eqref{w(x)} and have a singularity at $0$. 
We will keep all the notations  and definitions in Sections 
\ref{Introduction}--\ref{Multivariate integration}  with replacing $w$ by $w_\beta$ and $\mu$ by $\mu_\beta$, pointing some modifications.

Denote $\mathring{\RR}^d:= \brac{\RR \setminus \brab{0}}^d$ and 
$\mathring{\Omega}:= \Omega \cap \mathring{\RR}^d$. Besides the spaces $L_p(\Omega;\mu_\beta)$ and $W^r_{p}(\Omega;\mu_\beta)$ we consider also the spaces $L_p\big(\mathring{\Omega};\mu_\beta\big)$ and $W^r_{p}\big(\mathring{\Omega};\mu_\beta\big)$ which are defined in a similar manner. For the space $W^r_{p}\big(\mathring{\Omega};\mu_\beta\big)$, we require one of the following restrictions on $r$ and $\beta$ to be satisfied:
\begin{itemize}
	\item [{\rm (i)}]
	 $\beta > r - 1$;
	\item [{\rm (ii)}] 
	$0 <\beta < r - 1$ and $\beta$ is not an integer, for $f \in W^r_{p}\big(\mathring{\Omega};\mu_\beta\big)$, the derivative $D^\bk f$ can be extended to a continuous function on $\Omega$ for all $\bk \in \ZZdp$ such that 
	$|\bk|_\infty \le r - 1 -\lceil \beta \rceil$. 
	\end{itemize}

Let $(p_m(w_\beta))_{m \in \NN}$ be the sequence of orthonormal polynomials with respest to the weight $w_\beta$.
Denote again by $x_{m,k}$, $1 \le k \le \lfloor m/2 \rfloor$ the positive zeros of 	$p_m(w_\beta)$, and by $x_{m,-k} = - x_{m,k}$ the negative ones (if $m$ is odd, then $x_{m,0}= 0$ is also a zero of $p_m(w_\beta)$). 
If $m$ is even, we add $x_{m,0} := 0$.
These nodes are located as 
\begin{equation} \nonumber
	-\sqrt{m} + Cm^{-1/6} < x_{m,- \lfloor m/2 \rfloor} < \cdots 
	< x_{m,-1} < x_{m,0} < x_{m,1} < \cdots <  x_{m,\lfloor m/2 \rfloor} \le \sqrt{m} - Cm^{-1/6}, 
\end{equation}
with a positive constant $C$ independent of $m$ (the Mhaskar-Rakhmanov-Saff number is 
$a_m(w_\beta) = \sqrt{m}$).

In the case (i),  the truncated Gaussian quadrature is defined by
\begin{equation} \nonumber 
	Q^{\rm{TG}}_{2 j(m)}f: = \sum_{1 \le |k| \le j(m)} \lambda_{m,k}(w_\beta) f(x_{m,k}), 
\end{equation} 
and in  the case (ii) by
\begin{equation} \nonumber 
	Q^{\rm{TG}}_{2 j(m)}f: = \sum_{0 \le |k| \le j(m)} \lambda_{m,k}(w_\beta) f(x_{m,k}),
\end{equation} 
where $\lambda_{m,k}(w_\beta)$ are the corresponding Cotes numbers. 

In the same ways, by using related results in \cite{MO2004} we can prove the following counterparts of Theorems \ref{thm:Q_n-d=1} and \ref{theorem:Int_n(W)}  for the unit ball 
$\bW^r_{1,w_\beta}\big(\mathring{\RR}^d\big)$ of the Markov-Sonin weighted Sobolev space $W^r_{1,w_\beta}\big(\mathring{\RR}^d\big)$ of mixed smoothness $r \in \NN$. 
\begin{theorem} \label{thm:Q_n-d=1,SMW}
	For any $n \in \NN$, let $m_n$ be the largest integer such that $2 j(m_n) \le n$. Then the family of  quadratures 	$\big(Q^{\rm{TG}}_{2 j(m_n)}\big)_{n \in \RR_1}$ is  asymptotically optimal  for $\bW^r_1\big(\mathring{\RR};\mu_\beta\big)$ and
	\begin{equation*}
		\sup_{f\in \bW^r_1\big(\mathring{\RR};\mu_\beta\big)}
		\bigg|\int_{\RR} f(x) \mu_\beta(\rd x) - Q^{\rm{TG}}_{2 j(m_n)}f\bigg| 
		\asymp
		\Int_n\big(\bW^r_1\big(\mathring{\RR};\mu_\beta\big) \big)
		\asymp 
		n^{- r/2}.
	\end{equation*}
\end{theorem}

\begin{theorem} \label{theorem:Int_n(W)SMW}
	We have that
	\begin{equation*}\label{Int_n(W)}
		n^{-r/2} (\log n)^{(d-1)r/2} 
		\ll	
		\Int_n\bW^r_1\big(\mathring{\RR}^d;\mu_\beta\big)) 
		\ll 
		n^{-r/2} (\log n)^{(d-1)(r/2 + 1)}.
	\end{equation*}
\end{theorem}

\medskip
\noindent
{\bf Acknowledgments:}  
A part of this work was done when  the author was working at the Vietnam Institute for Advanced Study in Mathematics (VIASM). He would like to thank  the VIASM  for providing a fruitful research environment and working condition. He expresses  thanks to Dr. David Krieg for pointing out the papers \cite{KPUU2023} and \cite{CW2004} and for useful comments related to these papers.

%

\bibliographystyle{abbrv}
\bibliography{AllBib}

\begin{thebibliography}{10}

\bibitem{Bakhvalov1963}
N.~Bakhvalov.
\newblock {Optimal convergence bounds for quadrature processes and integration
  methods of Monte Carlo type for classes of functions}.
\newblock {\em Zh. Vychisl. Mat. i Mat. Fiz.}, 4(4):5--63, 1963.

\bibitem{Bel57B}
R.~Bellman.
\newblock {\em {Dynamic Programming}}.
\newblock Princeton University Press, Princeton, 1957.

\bibitem{Be1974}
O.~Besov.
\newblock {Multiplicative estimates for integral norms of differentiable
  functions of several variables}.
\newblock {\em Proc. Steklov Inst. Math.}, 131:1--14, 1974.

\bibitem{BuGr04}
H.-J. Bungartz and M.~Griebel.
\newblock Sparse grids.
\newblock {\em Acta Numer.}, 13:147--269, 2004.

\bibitem{BDSU2016}
G.~Byrenheid, D.~{D\~ung}, W.~Sickel, and T.~Ullrich.
\newblock {Sampling on energy-norm based sparse grids for the optimal recovery
  of Sobolev type functions in $H^{\gamma}$}.
\newblock {\em J. Approx. Theory}, 207:207--231, 2016.

\bibitem{BU2017}
G.~Byrenheid and T.~Ullrich.
\newblock { Optimal sampling recovery of mixed order Sobolev embeddings via
  discrete Littlewood--Paley type characterizations}.
\newblock {\em Anal Math.}, 43:133--191, 2017.

\bibitem{Chui92}
C.~K. Chui.
\newblock {\em {An Introduction to Wavelets}}.
\newblock Academic Press, 1992.

\bibitem{CW2004}
J.~Creutzig and P.Wojtaszczyk.
\newblock {Linear vs. nonlinear algorithms for linear problems}.
\newblock {\em J. Complexity}, 20:807--820, 2004.

\bibitem{DD1990}
D.~{D\~ung}.
\newblock {On recovery and one-sided approximation of periodicfunctions of
  several variables}.
\newblock {\em Dokl. Akad. Nauk SSSR, Ser. Mat.}, 313:787--790, 1990.

\bibitem{DD1991}
D.~{D\~ung}.
\newblock {On optimal recovery of multivariate periodic functions}.
\newblock {\em In: HarmonicAnalysis (Conference Proceedings, Ed. S. Igary),
  Springer-Verlag1991, Tokyo-Berlin}, pages 96--105, 1991.

\bibitem{DD2009}
D.~{D\~ung}.
\newblock Non-linear sampling recovery based on quasi-interpolant wavelet
  representations.
\newblock {\em Adv. Comput. Math.}, 30:375--401, 2009.

\bibitem{DD2011a}
D.~{D\~ung}.
\newblock {B-spline quasi-interpolant representations and sampling recovery of
  functions with mixed smoothness}.
\newblock {\em J. Complexity}, 27:541--567, 2011.

\bibitem{DD2016}
D.~{D\~ung}.
\newblock {Sampling and cubature on sparse grids based on a B-spline
  quasi-interpolation}.
\newblock {\em Found. Comp. Math.}, 16:1193--1240, 2016.

\bibitem{DD2018}
D.~{D\~ung}.
\newblock {B-spline quasi-interpolation sampling representation and sampling
  recovery in Sobolev spaces of mixed smoothness}.
\newblock {\em { Acta Math. Vietnamica }}, 43:83--110, 2018.

\bibitem{DD2023}
D.~{D\~ung}.
\newblock {Numerical weighted integration of functions having mixed
  smoothness}.
\newblock {\em {J. Complexity}}, 78:101757, 2023.

\bibitem{DK2023}
D.~{D\~ung} and V.~K. Nguyen.
\newblock {Optimal numerical integration and approximation of functions on
  $\mathbb{R}^d$ equipped with Gaussian measure}.
\newblock {\em {IMA Journal of Numerical Analysis}}, page
  https://doi.org/10.1093/imanum/drad051, 2023.

\bibitem{DTU18B}
D.~{D\~ung}, V.~N. Temlyakov, and T.~Ullrich.
\newblock {\em {Hyperbolic Cross Approximation}}.
\newblock Advanced Courses in Mathematics - CRM Barcelona,
  Birkh\"auser/Springer, 2018.

\bibitem{DU2015}
D.~{D\~ung} and T.~Ullrich.
\newblock {Lower bounds for the integration error for multivariate functions
  with mixed smoothness and optimal Fibonacci cubature for functions on the
  square}.
\newblock {\em Math. Nachr.}, 288:743--762, 2015.

\bibitem{DT2024}
F.~Dai and V.~Temlyakov.
\newblock {Random points are good for universal discretization}.
\newblock {\em J. Math. Anal. Appl.}, 529, 2024.

\bibitem{DM2003}
B.~{Della Vecchia} and G.~Mastroianni.
\newblock {Gaussian rules on unbounded intervals}.
\newblock {\em J. Complexity}, 19:247--258, 2003.

\bibitem{Dick2007}
J.~Dick.
\newblock {Explicit constructions of quasi-Monte Carlo rules for the numerical
  integration of high-dimensional periodic functions}.
\newblock {\em SIAM J. Numer. Anal.}, 45:2141--2176, 2007.

\bibitem{Dick2008}
J.~Dick.
\newblock {Walsh spaces containing smooth functions and quasi-Monte Carlo rules
  of arbitrary high order }.
\newblock {\em SIAM J. Numer. Anal.}, 46:1519--1553, 2008.

\bibitem{DILP18}
J.~Dick, C.~Irrgeher, G.~Leobacher, and F.~Pillichshammer.
\newblock {On the optimal order of integration in Hermite spaces with finite
  smoothness}.
\newblock {\em SIAM J. Numer. Anal.}, 56:684--707, 2018.

\bibitem{DKS13}
J.~Dick, F.~Y. Kuo, and I.~H. Sloan.
\newblock {High-dimensional integration: the quasi-Monte Carlo way}.
\newblock {\em Acta Numer.}, 22:133--288, 2013.

\bibitem{DP2010}
J.~Dick and F.~Pillichshammer.
\newblock {\em { Digital Nets and Sequences. Discrepancy Theory and Quasi-Monte
  Carlo Integration}}.
\newblock Cambridge University Press, Cambridge, 2010.

\bibitem{DKU2023}
M.~Dolbeault, D.~Krieg, and M.~Ullrich.
\newblock {A sharp upper bound for sampling numbers in $L_2$}.
\newblock {\em Appl. Comput. Harmon. Anal.}, 63:113--134, 2023.

\bibitem{Florov1976}
K.~K. Frolov.
\newblock {Upper error bounds for quadrature formulas on function classes}.
\newblock {\em Dokl. Akad. Nauk SSSR}, 231:818--821, 1976.

\bibitem{GSY2018}
T.~Goda, K.~Suzuki, and T.~Yoshiki.
\newblock { Optimal order quadrature error bounds for infinite dimensional
  higher order digital sequences}.
\newblock {\em Found. Comput. Math.}, 18:433--458, 2018.

\bibitem{IKLP2015}
C.~Irrgeher, P.~Kritzer, G.~Leobacher, and F.~Pillichshammer.
\newblock {Integration in Hermite spaces of analytic functions}.
\newblock {\em J. Complexity}, 31:308--404, 2015.

\bibitem{IL2015}
C.~Irrgeher and G.~Leobacher.
\newblock {High-dimensional integration on the $\mathbb{R}^d$, weighted Hermite
  spaces, and orthogonal transforms}.
\newblock {\em J. Complexity}, 31:174--205, 2015.

\bibitem{JUV2023}
T.~Jahn, T.~Ullrich, and F.~Voigtlaender.
\newblock {Sampling number sof smoothness classes via $\ell^1$-minimization}.
\newblock {\em J. Complexity}, 79, 2023.

\bibitem{JMN2021}
P.~Junghanns, G.~Mastroianni, and I.~Notarangelo.
\newblock {\em {Weighted Polynomial Approximation and Numerical Methods for
  Integral Equations}}.
\newblock Birkh\"auser, 2021.

\bibitem{KPUU2023}
D.~Krieg, K.~Pozharska, M.~Ullrich, and T.~Ullrich.
\newblock {Sampling recovery in uniform and other norms}.
\newblock {\em arXiv: 2305.07539}, 2023.

\bibitem{Lu07B}
D.~S. Lubinsky.
\newblock {A survey of weighted polynomial approximation with exponential
  weights}.
\newblock {\em Surveys in Approximation Theory}, 3:1--105, 2007.

\bibitem{MO2004}
G.~Mastroianni and D.~Occorsio.
\newblock {Markov-Sonin Gaussian rule for singular functions}.
\newblock {\em J. Comput. Appl. Math.}, 169(1):197--212, 2004.

\bibitem{Mha1996B}
H.~N. Mhaskar.
\newblock {\em {Introduction to the Theory of Weighted Polynomial
  Approximation}}.
\newblock World Scientific, Singapore, 1996.

\bibitem{NgS17}
V.~Nguyen and W.~Sickel.
\newblock {Pointwise multipliers for Sobolev and Besov spaces of dominating
  mixed smoothness}.
\newblock {\em J. Math. Anal. Appl.}, 452:62--90, 2017.

\bibitem{NT2006}
E.~Novak and H.~Triebel.
\newblock {Function spaces in Lipschitz domains and optimal rates of
  convergence for sampling}.
\newblock {\em Constr. Approx.}, 23:325--350, 2006.

\bibitem{Skriganov1994}
M.~M. Skriganov.
\newblock {Constructions of uniform distributions in terms of geometry of
  numbers}.
\newblock {\em Algebra i Analiz}, 6:200--230, 1994.

\bibitem{Smo63}
S.~Smolyak.
\newblock {Quadrature and interpolation formulas for tensor products of certain
  classes of functions}.
\newblock {\em Dokl. Akad. Nauk}, 148:1042--1045, 1963.

\bibitem{Stein1970}
E.~M. Stein.
\newblock {\em {Singular Integrals and Differentiability Properties of
  Functions}}.
\newblock Princeton Univ. Press, Princeton, NJ, 1970.

\bibitem{Tem1990}
V.~N. Temlyakov.
\newblock {On a way of obtaining lower estimates for the errors of quadrature
  formulas}.
\newblock {\em Matem. Sb.}, pages 1403--1413, 1990.

\bibitem{Tem1991}
V.~N. Temlyakov.
\newblock {Error estimates for Fibonacci quadrature formulas for classes of
  functions with a boundedmixed derivative}.
\newblock {\em Trudy Mat. Inst. Steklov}, 200:327--335, 1991.

\bibitem{Tem93B}
V.~N. Temlyakov.
\newblock {\em {Approximation of Periodic Functions}}.
\newblock Computational Mathematics and Analysis Series, Nova Science
  Publishers, Inc., Commack, NY., 1993.

\bibitem{Tem1993}
V.~N. Temlyakov.
\newblock {On approximate recovery of functions with bounded mixed derivative}.
\newblock {\em J. Complexity}, 9:41--59, 1993.

\bibitem{Tem2015}
V.~N. Temlyakov.
\newblock {Constructive sparse trigonometric approximation and other problems
  for functions with mixed smoothness}.
\newblock {\em Matem. Sb.}, 206:131--160, 2015.

\bibitem{Tem18B}
V.~N. Temlyakov.
\newblock {\em {Multivariate Approximation}}.
\newblock Cambridge University Press, 2018.

\bibitem{TWW1988B}
J.~Traub, G.~Wasilkowski, and H.~Wozniakowski.
\newblock {\em {Information-Based Complexity}}.
\newblock Academic Press, Inc., 1988.

\bibitem{Tri2010B}
H.~Triebel.
\newblock {\em {Bases in Function Spaces, Sampling, Discrepancy, Numerical
  Integration}}.
\newblock European Math. Soc. Publishing House, Z\"urich, 2010.

\bibitem{TB18}
R.~Tripathy and I.~Bilionis.
\newblock {Deep UQ: Learning deep neural network surrogate models for high
  dimensional uncertainty quantification}.
\newblock {\em J. Comput. Phys.}, 375:565--588, 2018.

\end{thebibliography}
\end{document}